\documentclass[11pt]{article}

\usepackage[margin=1in]{geometry}
\usepackage{amsmath,amssymb,amsthm}
\usepackage{amsthm}
\usepackage{graphicx} % Required for inserting images
\usepackage{hyperref}
\newcommand{\email}[1]{\href{mailto:#1}{\nolinkurl{#1}}}
\usepackage{bm}
\usepackage{xcolor}
\usepackage{algorithm}
\usepackage{algpseudocode}
\usepackage{subfig}

\newcommand{\x}{{\bf x}}
\newcommand{\ii}{{\rm i}}
\newcommand{\bu}{{\bf u}}
\newcommand{\bv}{{\bf v}}
\newcommand{\bff}{{\bf f}}
\newcommand{\bh}{{\bf h}}
\newcommand{\Div}{{\rm div}}

\newtheorem{Theorem}{Theorem}[section]
\newtheorem{Problem}{Problem}[section]
\newtheorem{Proposition}{Proposition}[section]
\newtheorem{Remark}{Remark}[section]
\newtheorem{Lemma}{Lemma}[section]
\newtheorem{Corollary}{Corollary}[section]

\title{Reconstruction of the initial wave field for a nonlinear Schr\"odinger equation from lateral boundary data via Carleman contraction}

\author{%
\begin{minipage}{0.9\textwidth}
\raggedright
{\large\bfseries Navaraj Neupane}\\
{\small
Department of Mathematics and Statistics, University of North Carolina
at Charlotte, Charlotte, NC 28223, USA,\\
Email: \email{nneupan2@charlotte.edu} }\\[1.2em]
{\large\bfseries Loc Hoang Nguyen} (corresponding author)\\
{\small
Department of Mathematics and Statistics, University of North Carolina
at Charlotte, Charlotte, NC 28223, USA,\\
Email: \email{loc.nguyen@charlotte.edu}}
\end{minipage}}

\date{}
\numberwithin{equation}{section}
\begin{document}
\maketitle
\begin{abstract}
We study the reconstruction of the initial wave field of a nonlinear Schr\"odinger equation from lateral measurements. Our approach combines a Legendre-polynomial-exponential-time dimensional reduction with a Carleman-based contraction principle. First, we expand the solution in a weighted Legendre basis in time and truncate the expansion to obtain a coupled nonlinear elliptic system for the spatial coefficients.  Next, we solve this reduced system by constructing a contraction map on a suitable admissible set. This contraction map admits a unique fixed point, which is the limit of the corresponding Picard iteration.  We also establish a stability estimate showing that this fixed point remains close to the exact reduced solution in the noisy-data case. Finally, we present numerical experiments in two space dimensions for several different inclusion geometries and nonlinear exponents. The numerical results show that the proposed method recovers the main geometric features of the initial wave field, namely the locations and shapes of the inclusions, and that the relative reconstruction errors change only slightly for the tested additive noise levels of $5\%$, $10\%$, and $20\%$ in the boundary data.
\end{abstract}

\noindent \textbf{Keywords:} nonlinear Schr\"odinger equation; Carleman estimate; time-dimensional reduction; Legendre polynomial-exponential basis; contraction mapping; Picard iteration.

\noindent \textbf{MSC 2020:} 35R30, 35Q55, 35J57, 35B45, 65N21.

\section{Introduction}

Let $d\ge 1$ be the spatial dimension, let $\Omega$ be a bounded domain of $\mathbb{R}^d$ with smooth boundary, and let $T>0$ be a final time. We consider the nonlinear Schr\"odinger equation
\begin{equation}\label{maineqn}
\begin{cases}
\ii u_t + \Delta u + q(\x,t)|u|^{p-1}u = 0, & \text{in } \Omega\times(0,T),\\
u(\x,t)=0, & \text{on } \partial\Omega\times(0,T),\\
u(\x,0)=u^0(\x), & \text{in } \Omega,
\end{cases}
\end{equation}
where $u:\Omega\times(0,T)\to \mathbb{C}$ is the wave field and $u^0:\Omega\to\mathbb{C}$ is the initial wave field. The function $q(\x,t)$ is a given real-valued coefficient describing the strength of the nonlinear interaction, and $p>1$ is the exponent of the nonlinearity. In particular, $p=2$, $p=3$, and $p=5$ correspond to quadratic, cubic, and quintic nonlinear Schr\"odinger models, respectively. Among these, the cubic case $p=3$ is the most classical and widely studied, especially in nonlinear optics, Bose--Einstein condensation, and wave propagation in dispersive media. Moreover, if $q\in L^\infty(\Omega\times(0,T))$ and $u^0\in H_0^1(\Omega)$, then the forward problem \eqref{maineqn} is locally well-posed for any finite $p>1$ when $d=1,2$, whereas for $d\ge 3$ one works in the range $1<p\le 1+\frac{4}{d-2}$; see, for instance, \cite{cazenave2003,sulem1999}.

Assuming that \eqref{maineqn} has a unique solution, we are interested in the following inverse problem.

\begin{Problem}[Reconstruction of the initial wave field]\label{idp}
Given the lateral Neumann data
\begin{equation}\label{data}
f(\x,t)=\partial_\nu u(\x,t)
\qquad \text{for all } (\x,t)\in \partial\Omega\times(0,T),
\end{equation}
reconstruct the initial wave field $u^0(\x)$ for $\x\in\Omega$.
\end{Problem}

For a prescribed initial field $u^0$, uniqueness of the forward solution to \eqref{maineqn} also fixes its lateral Neumann data \eqref{data}. Problem~\ref{idp} asks the converse question: can the same Neumann data arise from two different initial fields? When $q\in L^\infty(\Omega\times(0,T))$, the full lateral Neumann data uniquely determine $u^0$ within the class of strong solutions in $C([0,T];H^2(\Omega))\cap L^\infty(\Omega\times(0,T))$. If two such solutions have identical Neumann data, the zero extension of their difference satisfies a linear Schr\"odinger equation on $\mathbb{R}^d$ with a bounded, spatially compactly supported potential. Theorem~1 of \cite{escauriaza2008hardy} then yields equality of their initial fields. For broader solution classes, we treat inverse uniqueness as a working assumption.

This inverse problem is significant from both practical and mathematical points of view. Nonlinear Schr\"odinger equations arise in many applications, including nonlinear optics, Bose--Einstein condensation, plasma physics, and deep-water wave propagation \cite{lee2007,pitaevskii2016,sulem1999,vitanov2013}. In such settings, the initial wave field contains essential information about the state of the system at the initial time, but direct interior measurements are often difficult or impossible to obtain. By contrast, boundary observations are more accessible in experiments and monitoring processes. Therefore, recovering $u^0$ from lateral Neumann data provides a useful noninvasive way to identify the hidden initial state of the system. Once $u^0$ is reconstructed, the full wave field can then be recovered by solving the forward problem \eqref{maineqn}.

Inverse problems for Schr\"odinger equations have been studied extensively over the past several decades. Early works focused mainly on the recovery of electric potentials, coefficients, and magnetic fields from boundary measurements \cite{bellassoued_benfraj2020,bellassoued_choulli2010,bellassouedkiansoccorsi2016,choullikiansoccorsi2015,eskin2003,eskin2008}. In the linear Schr\"odinger setting, representative contributions include inverse potential recovery results from arbitrary boundary observation \cite{bellassoued_choulli2009} and under degenerate weights \cite{mercado2008}, problems with discontinuous and variable coefficients \cite{baudouin2008,deng2015}, magnetic and electromagnetic inverse problems in bounded and cylindrical geometries \cite{bellassoued_benfraj2020,bellassoued_choulli2010,bellassoued2018cylindrical,benaicha2018,cristofol2011,huang2019carleman,kian_soccorsi2019}, and a Neumann-boundary formulation \cite{saci2023}. More recently, inverse problems for nonlinear Schr\"odinger equations and partial boundary data have also attracted considerable attention. Uniqueness for nonlinear magnetic Schr\"odinger equations on conformally transversally anisotropic manifolds was established in \cite{krupchyk2023}. Partial-data inverse problems for nonlinear magnetic Schr\"odinger equations were studied in \cite{lai_zhou2023}, while partial-data determination of a time-dependent nonlinear coefficient was obtained in \cite{lai_lu_zhou2024}. Stable determination of coefficients in nonlinear dynamical Schr\"odinger equations from Neumann data was investigated in \cite{arrepu2025}. These works mainly recover potentials, coefficients, or magnetic fields. In contrast, Problem~\ref{idp} concerns the reconstruction of the initial wave field, while the coefficient $q$ and the exponent $p$ are known. For the linear Schr\"odinger equation, the determination of the initial state from boundary measurements is closely related to boundary observability. In particular, Tenenbaum and Tucsnak \cite[Theorem~1.4]{tenenbaum2009fast} proved that, on a rectangle with homogeneous Dirichlet boundary conditions, the $H^1_0$ norm of an initial state in $H^2\cap H^1_0$ is bounded by the $L^2$ norm of the Neumann trace on a boundary portion containing both a horizontal and a vertical segment, for every observation time $T>0$. Observation on the entire boundary satisfies this geometric condition. Hence, in that linear setting, the initial state is uniquely and stably determined by the Neumann data. The present paper addresses a different question: we develop a numerical reconstruction method for the nonlinear equation \eqref{maineqn}. The linear observability estimate does not by itself give a stability result for this nonlinear model.
Another possible approach to nonlinear inverse problems is to formulate a least-squares discrepancy functional and minimize it directly by an iterative optimization procedure. In this paper, we follow a different route. Our approach begins by eliminating the time variable through a Legendre polynomial-exponential expansion and truncating the solution to the first $N+1$ modes. This reduces the original inverse problem to a coupled nonlinear elliptic system for the spatial coefficients. We then solve the reduced system by a Carleman--Picard strategy: at each iteration, the nonlinear term is frozen at the current approximation, and the next iterate is defined as the unique minimizer of a Carleman-weighted regularized functional. This procedure generates a contraction map on a suitable admissible set, and hence the Picard iteration converges from an arbitrary initial guess to a unique fixed point. The approximate initial wave field is finally reconstructed by evaluating the truncated expansion at $t=0$.

The methodological background of the present paper comes from the combination of time-dimensional reduction and the Carleman contraction principle. This approach was first developed in \cite{le2022parabolic} for an inverse initial-value problem for a quasilinear parabolic equation. Later, \cite{nguyen2023carleman} showed that the method can be interpreted as the construction of a contraction mapping whose fixed point is the desired solution. Consequently, the associated Picard iteration converges globally, even when the initial guess is far from the true solution. Since then, this framework has been extended to a variety of inverse problems; see, for example, \cite{AbneyLeNguyenPeters, dang2024hyperbolic, klibanov2022contraction,  LeCON2023, le2024timedim, le2025maxwell, nguyen2022hyperbolic, NguyenNguyen2026, NguyenNguyenVu2026, van2025navierstokes}, in which inverse problems for hyperbolic, parabolic, elliptic, Maxwell, elasticity and Navier-Stokes equations were investigated. 

The present work extends the Carleman contraction framework of
\cite{le2022parabolic,nguyen2023carleman} to the reconstruction of the initial wave field of the complex-valued nonlinear Schr\"odinger equation with a
power-type nonlinearity. The Legendre polynomial--exponential time-dimensional reduction leads to
a coupled elliptic system with finitely many temporal modes, whose nonlinear
modal coupling can be analyzed within the contraction framework. To handle this nonlinear coupling, we work on a suitable
admissible set on which the modal vector is bounded in $L^\infty(\Omega)$
by an a~priori constant $M$. On this admissible set, the map
$z\mapsto |z|^{p-1}z$ is Lipschitz on bounded subsets of the complex
plane. This Lipschitz property allows us to establish the contraction
estimate for the reduced system.

The main contribution of this paper is to extend the Carleman contraction
framework to the reconstruction of the initial wave field of a complex-valued
nonlinear Schr\"odinger equation with power-type nonlinearity. Through the
Legendre polynomial--exponential time-dimensional reduction, the original
problem is transformed into a coupled nonlinear elliptic system. We show
that the nonlinear modal coupling arising from this reduction can be
incorporated into the contraction framework on a suitable admissible set.
We prove the global convergence of the associated Picard iteration,
establish the consistency of the fixed point with the exact reduced
solution, and derive a stability estimate for noisy data. In contrast to some previous Carleman-based frameworks
\cite{nguyen2023carleman}, where the noisy boundary data
are assumed to be traces of sufficiently regular vector-valued functions,
the stability estimate established in the present work only depends on the
noise level and the regularization parameter.

The remainder of the paper is organized as follows. Section \ref{sec2} recalls the analytical tools needed later, including the relevant Carleman estimate and the properties of the Legendre polynomial-exponential basis. Section \ref{sec3} derives the time-dimensional reduction model and the reduced boundary data. Section \ref{sec4} develops the Carleman-based contraction principle for the reduced system. Section \ref{sec5} proves the consistency of the fixed point with the exact reduced solution. Section \ref{sec6} presents the numerical algorithm and computational examples. Section \ref{sec7} is for the concluding remarks.

\section{Preliminary analytical tools} \label{sec2}

This section presents the main analytical ingredients used throughout the paper. First, we recall a Carleman estimate for an elliptic operator in divergence form, which will be used later in the analysis of the reduced system. Next, we summarize the basic properties of the Legendre polynomial-exponential basis underlying our time-dimensional reduction method. For the reader's convenience, we also include a convergence result for the expansion of the first-time derivative.

\subsection{A Carleman estimate}
\label{subsec:carleman_estimate}

A key tool in our analysis is a Carleman estimate for an elliptic operator in divergence form.
Let
\[
A:\overline{\Omega}\to\mathbb{R}^{d\times d}
\]
be a matrix-valued function of class $C^2(\overline{\Omega})$. Assume that
\begin{enumerate}
    \item $A$ is symmetric, that is,
    \[
    A^{\rm T}=A;
    \]
    \item $A$ is uniformly elliptic: there exists a constant $\Lambda>0$ such that
    \begin{equation}\label{ellipic_matrix}
    \Lambda^{-1}|\xi|^2 \le A(\x)\xi\cdot \xi \le \Lambda |\xi|^2
    \qquad \text{for all } \x\in\overline{\Omega},\ \xi\in\mathbb{R}^d.
    \end{equation}
\end{enumerate}
Let $\x_0\in\mathbb{R}^d\setminus \overline{\Omega}$ and define
\[
r(\x)=|\x-\x_0|,
\qquad \x\in\overline{\Omega}.
\]
Also, let
\[
R=\max_{\x\in\overline{\Omega}} r(\x).
\]

\begin{Lemma}
\label{thmCarpointwise}
Let $u\in C^2(\overline{\Omega})$. Then there exists a constant $\beta_0>0$, depending only on $\|A\|_{C^1(\overline{\Omega})}$ and $\Lambda$, such that for every $\beta\ge \beta_0$ and every $\lambda\ge \lambda_0:=2R^\beta$,
\begin{equation}
r^{\beta+2} e^{2\lambda r^{-\beta}} |\Div(A\nabla u)|^2
\ge
C\Big[
\Div(U)
+\lambda^3\beta^4 e^{2\lambda r^{-\beta}} r^{-2\beta-2}|u|^2
+\lambda\beta e^{2\lambda r^{-\beta}} |\nabla u|^2
\Big]
\label{CarEst}
\end{equation}
in $\Omega$, where $U$ is a vector-valued function satisfying
\begin{equation}
|U|
\le
C e^{2\lambda r^{-\beta}}
\big(
\lambda^3\beta^3 r^{-2\beta-2}|u|^2
+\lambda\beta |\nabla u|^2
\big),
\label{divU}
\end{equation}
and where $C>0$ depends only on $\x_0$, $\Omega$, $\|A\|_{C^1(\overline{\Omega})}$, $\Lambda$, and $d$.
\end{Lemma}

The proof is based on the behavior of the exponential weight $e^{\lambda r^{-\beta}}$. The main task is to estimate the quantity
\[
e^{2\lambda r^{-\beta}} |\Div(A\nabla u)|^2.
\]
To do so, one applies the product rule to the weighted operator. In the course of this computation, the second derivatives of $u$ are redistributed into terms involving first-order derivatives and zeroth-order terms. At the same time, each differentiation of the exponential weight produces a factor containing the large parameter $\lambda$. As a result, after expanding the weighted operator, one obtains dominant positive terms with high powers of $\lambda$, as seen in \eqref{CarEst}. We omit the proof here and refer the reader to \cite{LeLeNguyen:2024} for the full details.

A convenient consequence of Lemma~\ref{thmCarpointwise} is the following simplified form.

\begin{Corollary}
\label{cor:CarEst1}
Fix $\beta\ge \beta_0$. Then there exists a constant $\lambda_0>0$, depending only on $\Lambda$, $\|A\|_{C^2(\overline{\Omega})}$, $\x_0$, $\Omega$, $R$, $\beta$, and $d$, such that for all $\lambda\ge \lambda_0$,
\begin{equation}
e^{2\lambda r^{-\beta}} |\Div(A\nabla u)|^2
\ge
C\Big[
\Div(U)
+\lambda^3 e^{2\lambda r^{-\beta}} |u|^2
+\lambda e^{2\lambda r^{-\beta}} |\nabla u|^2
\Big]
\label{CarEst1}
\end{equation}
in $\Omega$, where $C>0$ depends only on $\Lambda$, $\|A\|_{C^2(\overline{\Omega})}$, $\x_0$, $\Omega$, $R$, $\beta$, and $d$.
\end{Corollary}

Integrating \eqref{CarEst1} over $\Omega$ and using \eqref{divU}, we obtain the following global form.

\begin{Corollary}
\label{col222}
There exists a constant $C>0$, depending only on $\Lambda$, $\|A\|_{C^2(\overline{\Omega})}$, $\x_0$, $\Omega$, $R$, $\beta$, and $d$, such that
\begin{multline}
\int_{\Omega} e^{2\lambda r^{-\beta}} |\Div(A\nabla u)|^2\,d\x
\ge
C \int_{\Omega} e^{2\lambda r^{-\beta}}
\big[
\lambda^3 |u|^2+\lambda |\nabla u|^2
\big]\,d\x
\\
-
C\int_{\partial\Omega} e^{2\lambda r^{-\beta}}
\big[
\lambda^3 |u|^2+\lambda |\nabla u|^2
\big]\,d\sigma(\x).
\label{3.40}
\end{multline}
In particular, if
\[
u|_{\partial\Omega}=0
\qquad \text{and} \qquad
\nabla u|_{\partial\Omega}=0,
\]
then
\begin{equation}
\int_{\Omega} e^{2\lambda r^{-\beta}} |\Div(A\nabla u)|^2\,d\x
\ge
C \int_{\Omega} e^{2\lambda r^{-\beta}}
\big[
\lambda^3 |u|^2+\lambda |\nabla u|^2
\big]\,d\x.
\label{CarlemanExercise}
\end{equation}
\end{Corollary}

\begin{Remark}
Estimate \eqref{3.40} is closely related to \cite[Lemma 5]{nguyen2015cloaking}. The main difference is that the result in \cite[Lemma 5]{nguyen2015cloaking} was established for annular domains, whereas \eqref{3.40} is valid for more general bounded domains. It is worth mentioning that the Carleman estimate in \cite[Lemma 5]{nguyen2015cloaking} was used there to prove a cloaking phenomenon. The reader can find many other variants of Carleman estimates in \cite{beilina2012,klibanov2022convexification, klibanov2021book,  nguyen2019, protter1960}. Such estimates have become an essential tool in the study of inverse problems; see, for example, \cite{khoa2020, le2022parabolic, nguyen2022hyperbolic}.
\end{Remark}

\begin{Corollary}
\label{cor:carleman_laplacian}
Assume that $A=I$, the identity matrix. Then $\Div(A\nabla u)=\Delta u$, and Corollary~\ref{col222} yields the following estimate
\begin{multline}
\int_{\Omega} e^{2\lambda r^{-\beta}} |\Delta u|^2\,d\x
\ge
C \int_{\Omega} e^{2\lambda r^{-\beta}}
\big[
\lambda^3 |u|^2+\lambda |\nabla u|^2
\big]\,d\x
\\
-
C\int_{\partial\Omega} e^{2\lambda r^{-\beta}}
\big[
\lambda^3 |u|^2+\lambda |\nabla u|^2
\big]\,d\sigma(\x).
\label{3.41}
\end{multline}
In particular, if $u|_{\partial \Omega} = 0$ then $|\nabla u| = |\partial_{\nu} u|$. In this case, \eqref{3.41} becomes
\begin{equation}
\int_{\Omega} e^{2\lambda r^{-\beta}} |\Delta u|^2\,d\x
\ge
C \int_{\Omega} e^{2\lambda r^{-\beta}}
\big[
\lambda^3 |u|^2+\lambda |\nabla u|^2
\big]\,d\x
-
C\int_{\partial\Omega} e^{2\lambda r^{-\beta}}
\lambda |\partial_{\nu} u|^2d\sigma(\x).
\label{3.42}
\end{equation}

\end{Corollary}

\begin{Remark}
The estimate in Lemma \ref{thmCarpointwise} is a particular instance of the classical
theory of Carleman estimates for elliptic operators. Different weight
functions can be used under suitable conditions on the operator, the
geometry of the domain, and the boundary terms needed in the application.
For general treatments and applications to Cauchy and inverse problems,
we refer to
\cite{choulli2016elliptic,hormander2007iii,isakov2017inverse,yamamoto2025inverse}.
In the present work, we use the weight $e^{2\lambda r^{-\beta}}$ and the
corresponding boundary estimate in the analysis of the reduced elliptic
system.
\end{Remark}

\subsection{The Legendre polynomial-exponential basis}
\label{sec:legendre_review}

Our time-dimensional reduction method is based on the Legendre polynomial-exponential basis introduced in \cite{trong2025elastic}; see also \cite{van2025navierstokes} for related properties used in the reduction process. For the reader's convenience, we briefly summarize the main definitions and facts needed later.

Let $\{P_n\}_{n\ge 0}$ be the classical Legendre polynomials on $(-1,1)$, given by Rodrigues' formula
\[
P_n(x)=\frac{1}{2^n n!}\frac{d^n}{dx^n}(x^2-1)^n.
\]
To transfer this family to the interval $(0,T)$, we use the affine change of variables
\[
x=\frac{2t}{T}-1,
\]
and define
\[
Q_n(t):=\sqrt{\frac{2n+1}{T}}\,P_n\!\left(\frac{2t}{T}-1\right),
\qquad t\in(0,T),\quad n\ge 0.
\]
Then $\{Q_n\}_{n\ge 0}$ is an orthonormal basis of $L^2(0,T)$.

Following \cite{trong2025elastic}, we introduce the weighted functions
\begin{equation}
\label{et}
\Psi_n(t):=e^t Q_n(t), \qquad t\in(0,T),\quad n\ge 0.
\end{equation}
The family $\{\Psi_n\}_{n\ge 0}$ is orthonormal with respect to the weighted inner product
\[
\langle u,v\rangle_{e^{-2t}}
:=
\int_0^T e^{-2t}u(t)v(t)\,dt,
\]
and therefore forms an orthonormal basis in the weighted space
\[
L^2_{e^{-2t}}(0,T)
:=
\left\{
u\in L^2(0,T):
\int_0^T e^{-2t}|u(t)|^2\,dt<\infty
\right\}.
\]

\begin{Remark}
The weighted space $L^2_{e^{-2t}}(0,T)$ coincides with the classical space $L^2(0,T)$, since the weight $e^{-2t}$ is positive and bounded above and below on the finite interval $(0,T)$. In particular, the corresponding norms are equivalent. We use the notation $L^2_{e^{-2t}}(0,T)$ in order to emphasize the presence of the weight $e^{-2t}$ in the associated inner product and norm.
\end{Remark}

We next recall several properties of this basis that will be used throughout the paper.

\begin{Proposition}[See \cite{trong2025elastic}]
\label{prop:basis-properties}
The Legendre polynomial-exponential basis functions $\Psi_n$, $n\ge 0$, satisfy the following properties.
\begin{enumerate}
    \item For each $n\ge 0$, the function $\Psi_n$ is infinitely differentiable on $(0,T)$, and none of its derivatives of any order vanishes identically on this interval.
    
    \item For every integer $\ell\in\mathbb{N}$, there exists a constant $C>0$, depending only on $\ell$ and $T$, such that for all $u\in H^\ell(0,T)$,
    \begin{equation}\label{eq:coeff_decay}
        \sum_{n=0}^{\infty} n^{2\ell}
        \left|
        \langle u,\Psi_n\rangle_{e^{-2t}}
        \right|^2
        \le C\|u\|_{H^\ell(0,T)}^2.
    \end{equation}
    
    \item There exists a constant $C>0$, depending only on $T$, such that for all $n\ge 1$,
    \begin{equation}\label{eq:psi_prime_growth}
        \|\Psi_n'\|_{L^2_{e^{-2t}}(0,T)} \le C n^{3/2},
        \qquad
        \|\Psi_n''\|_{L^2_{e^{-2t}}(0,T)} \le C n^{7/2}.
    \end{equation}
\end{enumerate}
\end{Proposition}

\begin{Remark}
The statements in Proposition~\ref{prop:basis-properties} follow from the results established in \cite{trong2025elastic}; see in particular Proposition~2.1, Lemma~2.1, and Lemma~2.2 there. The role of the exponential factor in the definition $\Psi_n=e^tQ_n$ is discussed in Remark~\ref{rem:expfactor}.
\end{Remark}

The next proposition provides the counterpart, at the level of the first time derivative, of the second-derivative convergence result established in \cite{trong2025elastic}.

\begin{Proposition}
\label{prop:first_derivative_basis}
Let $\rho\ge 0$ and assume that
\[
u\in H^\ell\big((0,T);H^\rho(\Omega)\big)\qquad \text{for some }\ell\ge 3.
\]
Denote the Legendre-exponential coefficients of $u$ by
\[
u_n(\cdot):=\left\langle u(\cdot,\cdot),\Psi_n\right\rangle_{L^2_{e^{-2t}}(0,T)}
=\int_0^T e^{-2t}u(\cdot,t)\Psi_n(t)\,dt,\qquad n\ge 0.
\]
Then $u_t\in L^2\big((0,T);H^\rho(\Omega)\big)$ and
\begin{equation*}
\partial_t u(\cdot,t)=\sum_{n=0}^{\infty} u_n(\cdot)\,\Psi_n'(t)
\quad \text{in } L^2\big((0,T);H^\rho(\Omega)\big).
\end{equation*}
\end{Proposition}

\begin{Remark}
The proof of Proposition~\ref{prop:first_derivative_basis} follows the same line of argument as the proof of the corresponding second-derivative result in \cite{trong2025elastic}. The condition $\ell\ge 3$ comes from combining the coefficient decay estimate \eqref{eq:coeff_decay} with the derivative bound \eqref{eq:psi_prime_growth}; see also \cite[Theorem 1]{van2025navierstokes} for further details.
\end{Remark}

\section{The time-dimensional reduction model} \label{sec3}

Let $\{\Psi_n\}_{n \geq 0}$ be the Legendre exponential-polynomial basis of $L^2(0,T)$, introduced in \cite{trong2025elastic}. We write
\begin{equation}\label{3}
u(\x,t)=\sum_{n=0}^\infty u_n(\x)\Psi_n(t)
\qquad \text{for } (\x,t)\in \Omega\times(0,T),
\end{equation}
where
\[
u_n(\x)=\int_0^T e^{-2t}u(\x,t)\Psi_n(t)\,dt.
\]
By Proposition \ref{prop:first_derivative_basis}, see also \cite[Theorem 1]{van2025navierstokes},
\begin{equation}\label{4}
u_t(\x,t)=\sum_{n=0}^\infty u_n(\x)\Psi_n'(t)
\qquad \text{for } (\x,t)\in \Omega\times(0,T).
\end{equation}
Plugging \eqref{3} and \eqref{4} into the Schr\"odinger equation \eqref{maineqn}, we obtain
\begin{equation}\label{5}
\ii \sum_{n=0}^\infty u_n(\x)\Psi_n'(t)
+ \sum_{n=0}^\infty \Delta u_n(\x)\Psi_n(t)
+ q(\x,t)\Big|\sum_{l=0}^\infty u_l(\x)\Psi_l(t)\Big|^{p-1}
\sum_{n=0}^\infty u_n(\x)\Psi_n(t)
=0
\end{equation}
for $(\x, t) \in \Omega \times (0, T).$

For each $m\ge 0$, multiply both sides of \eqref{5} by $e^{-2t}\Psi_m(t)$ and integrate over $t\in(0,T)$. Using the orthonormality relation
\[
\int_0^T e^{-2t}\Psi_n(t)\Psi_m(t)\,dt=\delta_{mn},
\]
and denoting
\[
\bu=
\begin{bmatrix}
u_0 & u_1 & \dots
\end{bmatrix}^{\top},
\]
we obtain
\begin{equation}\label{6}
\ii \sum_{n=0}^\infty s_{mn} u_n(\x)
+ \Delta u_m(\x)
+ \sum_{n=0}^\infty b_{mn}(\bu,\x)\,u_n(\x)
=0,
\end{equation}
where
\begin{align}
s_{mn}
&=
\int_0^T e^{-2t}\Psi_n'(t)\Psi_m(t)\,dt,\\
b_{mn}(\bu,\x)
&=
\int_0^T e^{-2t}q(\x,t)
\Big|\sum_{l=0}^\infty u_l(\x)\Psi_l(t)\Big|^{p-1}
\Psi_n(t)\Psi_m(t)\,dt.
\end{align}

Fix a cutoff number $N\ge 0$. By truncating the series in \eqref{6}, we approximate it by
\begin{equation}\label{9}
\ii \sum_{n=0}^N s_{mn} u_n(\x)
+ \Delta u_m(\x)
+ \sum_{n=0}^N b_{mn}^N(\bu^N,\x)\,u_n(\x)
=0,
\qquad m=0,1,\dots,N,
\end{equation}
where
\begin{align}
\bu^N
&=
\begin{bmatrix}
u_0 & u_1 & \dots & u_N
\end{bmatrix}^{\top},\\
b_{mn}^N(\bu^N,\x)
&=
\int_0^T e^{-2t}q(\x,t)
\Big|\sum_{l=0}^N u_l(\x)\Psi_l(t)\Big|^{p-1}
\Psi_n(t)\Psi_m(t)\,dt.
\end{align}

Equation \eqref{9} serves as the time-dimensional reduction model, which approximates the original time-dependent Schr\"odinger equation. We next compute the boundary conditions for $u_m$.

Using \eqref{3} and the homogeneous Dirichlet boundary condition in \eqref{maineqn}, we have
\[
\sum_{n=0}^\infty u_n(\x)\Psi_n(t)=0
\qquad \text{for } (\x,t)\in \partial\Omega\times(0,T).
\]
Multiplying both sides by $e^{-2t}\Psi_m(t)$ and integrating over $(0,T)$, we obtain
\begin{equation}\label{bc-Dirichlet-um}
u_m(\x)=0
\qquad \text{for } \x\in\partial\Omega,\quad m\ge 0.
\end{equation}

Next, differentiating \eqref{3} in the outward normal direction yields
\[
\partial_\nu u(\x,t)=\sum_{n=0}^\infty \partial_\nu u_n(\x)\Psi_n(t)
\qquad \text{for } (\x,t)\in \partial\Omega\times(0,T).
\]
Using the Neumann data \eqref{data}, we obtain
\[
f(\x,t)=\sum_{n=0}^\infty \partial_\nu u_n(\x)\Psi_n(t)
\qquad \text{for } (\x,t)\in \partial\Omega\times(0,T).
\]
Multiplying both sides by $e^{-2t}\Psi_m(t)$ and integrating over $(0,T)$, we arrive at
\begin{equation}\label{bc-Neumann-um}
\partial_\nu u_m(\x)=f_m(\x)
\qquad \text{for } \x\in\partial\Omega,\quad m\ge 0,
\end{equation}
where
\begin{equation}\label{def-fm}
f_m(\x):=\int_0^T e^{-2t}f(\x,t)\Psi_m(t)\,dt.
\end{equation}

Hence, for each $m=0,1,\dots,N$, the function $u_m$ satisfies the boundary conditions \eqref{bc-Dirichlet-um} and \eqref{bc-Neumann-um} on $\partial\Omega$. Therefore, combining \eqref{9}, \eqref{bc-Dirichlet-um}, and \eqref{bc-Neumann-um}, we obtain the following coupled elliptic system:
\begin{equation}\label{15}
\begin{cases}
\ii \displaystyle\sum_{n=0}^N s_{mn} u_n(\x)
+ \Delta u_m(\x)
+ \displaystyle\sum_{n=0}^N b_{mn}^N(\bu^N,\x)\,u_n(\x)
=0, & \x \in \Omega,\\
u_m(\x) = 0, & \x \in \partial\Omega,\\
\partial_{\nu} u_m(\x) = f_m(\x), & \x \in \partial\Omega,
\end{cases}
\qquad m=0,1,\dots,N.
\end{equation}
Solving system \eqref{15} is the next step of our method.
\begin{Remark}
System \eqref{15} will be referred to as the {\bf time-dimensional reduction model}. It provides an approximate reduction of the original time-dependent Schr\"odinger problem by eliminating the explicit time variable through the truncated expansion \eqref{3}. As a result, instead of working on the $(d+1)$-dimensional space-time domain $\Omega\times(0,T)$, one only needs to solve a coupled system on the $d$-dimensional spatial domain $\Omega$. This reduction significantly decreases the computational cost.

In addition, the derivation of \eqref{15} involves truncating the expansion to the first $N+1$ modes. Therefore, the high-oscillation components of the data are discarded. This truncation acts as a filtering step and can help reduce the influence of noise in practical computations.
\end{Remark}

\begin{Remark}[Role of the exponential factor in the modal coupling]\label{rem:expfactor}
The factor $e^t$ in $\Psi_n(t)=e^tQ_n(t)$ in \eqref{et}
compensates for the weight $e^{-2t}$ in the temporal inner
product, while also affecting the time-derivative term in
\eqref{4}. Indeed,
\[
\Psi_n'(t)=e^t\bigl(Q_n(t)+Q_n'(t)\bigr).
\]
Since $Q_n'$ has degree at most $n-1$, it is orthogonal to
$Q_n$ in $L^2(0,T)$. Consequently, the diagonal coefficient
in the time-derivative coupling satisfies
\[
s_{nn}
=\int_0^T e^{-2t}\Psi_n'(t)\Psi_n(t)\,dt
=\int_0^T\bigl(Q_n(t)+Q_n'(t)\bigr)Q_n(t)\,dt
=1,\qquad n\geq0.
\]
In particular, the lowest mode contributes to the
time-derivative term of the reduced system \eqref{15}.
By contrast, for the shifted Legendre basis without the
exponential factor, $Q_0'(t)=0$, so the lowest mode does
not contribute to that term. This is one reason for our
choice of the exponential-polynomial basis.
\end{Remark}

\section{A Carleman-contraction principle for the time-dimensional \\reduction model}\label{sec4}

Let $s$ be an integer such that
$
s>\frac d2+2.
$
Then, by the Sobolev embedding theorem,
\[
H^s(\Omega)\hookrightarrow C^2(\overline{\Omega}),
\]
and in particular
\[
H^s(\Omega)\hookrightarrow L^\infty(\Omega).
\]
Moreover, the mappings
\[
u\mapsto \Delta u,\qquad
u\mapsto u|_{\partial\Omega},\qquad
u\mapsto \partial_\nu u|_{\partial\Omega}
\]
are continuous from $H^s(\Omega)$ into $L^2(\Omega)$, $L^2(\partial\Omega)$, and
$L^2(\partial\Omega)$, respectively.
We seek a solution to \eqref{15} in the admissible set
\[
H=\left\{
\bm{\varphi}\in [H^s(\Omega)]^{N+1}: \bm{\varphi}|_{\partial \Omega} = \mathbf{0} \mbox{ and }
\|\bm{\varphi}\|_{[L^\infty(\Omega)]^{N+1}} \le M
\right\},
\]
where $M>0$ is a fixed constant chosen sufficiently large.

\begin{Remark}
The restriction to the admissible set $H$ is imposed as an a priori regularity assumption on the exact coefficient vector. More precisely, we assume that the exact coefficient vector belongs to $[H^s(\Omega)]^{N+1}$. Hence, for $M$ sufficiently large, it belongs to $H$. Therefore, $H$ should be viewed as a natural class of physically meaningful solutions rather than as a restrictive assumption.
\end{Remark}

Fix $\beta \geq \beta_0$ where $\beta_0$ is as in Lemma \ref{thmCarpointwise}. For each $\lambda > \lambda_0$, where $\lambda_0$ is also defined in Lemma \ref{thmCarpointwise}, and $\epsilon > 0$,
for each $\bm {\varphi} = \begin{bmatrix}
	\varphi_0 &\varphi_1 & \dots &\varphi_N
\end{bmatrix}^\top \in H$, define
\begin{multline}
J_{\bm{\varphi}}^{\lambda, \epsilon}(\bu)
:=
\sum_{m=0}^N \Bigg[
\int_\Omega e^{2\lambda r^{-\beta}}
\Big|
\ii \sum_{n=0}^N s_{mn}u_n(\x)
+\Delta u_m(\x)
+\sum_{n=0}^N b_{mn}^N(\bm{\varphi},\x)\,\varphi_n(\x)
\Big|^2\,d\x
\\
+\lambda^3\int_{\partial\Omega} e^{2\lambda r^{-\beta}} |\partial_\nu u_m-f_m|^2\,d\sigma(\x)
+\epsilon \|u_m\|_{{H^s_{\lambda,\beta}(\Omega)}}^2
\Bigg].
\label{eq:frozen_functional}
\end{multline}
Here, for $u,v\in H^s(\Omega)$,
\[
\langle u,v\rangle_{H^s_{\lambda,\beta}(\Omega)}
:=\sum_{|\alpha|\le s}\int_\Omega e^{2\lambda r^{-\beta}}D^\alpha u\,\overline{D^\alpha v}\,d\x,
\qquad
\|u\|_{H^s_{\lambda,\beta}(\Omega)}^2:=\langle u,u\rangle_{H^s_{\lambda,\beta}(\Omega)},
\]
and $\|\bu\|_{[H^s_{\lambda,\beta}(\Omega)]^{N+1}}^2:=\sum_{m=0}^N\|u_m\|_{H^s_{\lambda,\beta}(\Omega)}^2$.
Since $1\le e^{2\lambda r^{-\beta}}\le e^{2\lambda d_0^{-\beta}}$ on $\overline{\Omega}$, where $d_0=\min_{\overline{\Omega}}r>0$, we have
\begin{equation}\label{eq:weighted_equiv}
\|u\|_{H^s(\Omega)}\le\|u\|_{H^s_{\lambda,\beta}(\Omega)}\le e^{\lambda d_0^{-\beta}}\|u\|_{H^s(\Omega)}
\qquad\text{for all } u\in H^s(\Omega).
\end{equation}
Hence $\langle\cdot,\cdot\rangle_{H^s_{\lambda,\beta}(\Omega)}$ is an inner product on $H^s(\Omega)$ whose norm is equivalent to $\|\cdot\|_{H^s(\Omega)}$.

\begin{Remark}[Relative strength of regularization]
\label{rem:weighted_regularization}
We use the same Carleman weight in the $H^s$ regularization and the
interior residual term to keep their relative strength balanced
throughout $\Omega$. Without this weight, the regularization would
be relatively stronger where the Carleman weight is small and
relatively weaker where it is large. In regions of large Carleman
weight, its effect could therefore become negligible compared with
the residual term.
\end{Remark}

\begin{Proposition}\label{prop:well_defined}
Given $\lambda>\lambda_0$ and $\epsilon>0$, the functional $J_{\bm{\varphi}}^{\lambda,\epsilon}$ admits a unique minimizer on $H$ for every $\bm{\varphi}\in H$.
\end{Proposition}

\begin{proof}
The existence of a minimizer follows from the direct method in the calculus of variations; see, for example, \cite{deimling1985, zeidler1985}.
Fix $\bm{\varphi}\in H$. Clearly, $H$ is nonempty, since $0\in H$. In addition, the continuity of the embedding $H^s(\Omega)\hookrightarrow L^\infty(\Omega)$ implies that $H$ is a closed and convex subset of $[H^s(\Omega)]^{N+1}$. Since $[H^s(\Omega)]^{N+1}$ is a Hilbert space, $H$ is also weakly closed.

Since $\bm{\varphi}\in H$ and $q\in L^\infty(\Omega\times(0,T))$, we have $b_{mn}^N(\bm{\varphi},\cdot)\in L^\infty(\Omega)$ for all $m,n=0,\dots,N$. Moreover, since $s>\frac d2+2$, the operators $\Delta$, the Dirichlet trace, and the Neumann trace are continuous on $H^s(\Omega)$. Also, since $\x_0\notin \overline{\Omega}$, the function $r(\x)=|\x-\x_0|$ is bounded above and below by positive constants on $\overline{\Omega}$. Hence the weight $e^{2\lambda r^{-\beta}}$ is bounded above and below by positive constants on $\Omega$ and on $\partial\Omega$. Therefore, $J_{\bm{\varphi}}^{\lambda,\epsilon}$ is well defined on $H$.

We first prove the existence of a minimizer. Since every term in $J_{\bm{\varphi}}^{\lambda,\epsilon}$ is nonnegative,
\[
J_{\bm{\varphi}}(\bu)\ge \epsilon \sum_{m=0}^N \|u_m\|_{{H^s_{\lambda,\beta}(\Omega)}}^2
\mathrel{{\ge}} \epsilon \|\bu\|_{[H^s(\Omega)]^{N+1}}^2
\qquad \text{for all } \bu\in H,
\]
{where we have used \eqref{eq:weighted_equiv}.} Thus $J_{\bm{\varphi}}^{\lambda,\epsilon}$ is coercive on $H$. Let $\alpha:=\inf_{\bu\in H} J_{\bm{\varphi}}^{\lambda,\epsilon}(\bu)$, and let $\{\bu^k\}_{k=1}^\infty\subset H$ be a minimizing sequence such that $J_{\bm{\varphi}}^{\lambda,\epsilon}(\bu^k)\to \alpha$ as $k\to\infty$. By coercivity, $\{\bu^k\}_{k = 1}^\infty$ is bounded in $[H^s(\Omega)]^{N+1}$. Since $[H^s(\Omega)]^{N+1}$ is a Hilbert space, there exist a subsequence, still denoted by $\{\bu^k\}_{k = 1}^\infty$, and an element $\bu^*\in [H^s(\Omega)]^{N+1}$ such that
\[
\bu^k\rightharpoonup \bu^*
\qquad \text{weakly in } [H^s(\Omega)]^{N+1}.
\]
Because $H$ is weakly closed, we have $\bu^*\in H$.

The operators $\Delta$, the Dirichlet trace, and the Neumann trace are continuous and linear with respect to $\bu$. Therefore, under the weak convergence above, the corresponding images converge weakly in their target spaces. It follows from the weak lower semicontinuity of the norms, including $\|\cdot\|_{H^s_{\lambda,\beta}(\Omega)}$, which is equivalent to $\|\cdot\|_{H^s(\Omega)}$ by \eqref{eq:weighted_equiv}, that
\[
J_{\bm{\varphi}}^{\lambda,\epsilon}(\bu^*)
\le \liminf_{k\to\infty} J_{\bm{\varphi}}^{\lambda,\epsilon}(\bu^k)
= \alpha.
\]
Hence $\bu^*$ is a minimizer of $J_{\bm{\varphi}}^{\lambda,\epsilon}$ on $H$.
The uniqueness of the minimizer follows from the strict convexity of $J_{\bm{\varphi}}^{\lambda,\epsilon}$ on the convex set $H$.
\end{proof}

 Define $F_{\lambda, \epsilon}:H\to H$ by letting $F(\bm{\varphi})$ be the unique minimizer of $J_{\bm{\varphi}}^{\lambda,\epsilon}$ on $H$, that is,
\[
F_{\lambda, \epsilon}(\bm{\varphi}) := \underset{\bu\in H}{\operatorname{argmin}}\, J_{\bm{\varphi}}^{\lambda, \epsilon}(\bu).
\]
The well-definedness of $F_{\lambda, \epsilon}$ follows from Proposition \ref{prop:well_defined}.
We next show that when $\lambda$ is sufficiently large, the map $F_{\lambda, \epsilon}$ is contractive with respect to the norm
\begin{equation}\label{norm}
    \|\bu\|_{\lambda, \epsilon}^2
    = \int_\Omega e^{2\lambda r^{-\beta}}
|
 \bu 
|^2\,d\x
+2 \int_{\partial\Omega} e^{2\lambda r^{-\beta}} |\partial_\nu \bu |^2\,d\sigma(\x)
+\frac{2\epsilon}{\lambda^3} \|\bu\|_{{[H^s_{\lambda,\beta}(\Omega)]^{N+1}}}^2
\end{equation}

\begin{Theorem}\label{thm:contraction}	
	Let $\epsilon>0$. Fix $\beta > \beta_0$. There exist $\lambda_1$ and $C$ depending only on $\Omega$, $\beta$, $s$, $N$, $p$, $T$, $\|q\|_{L^\infty(\Omega\times(0,T))}$, $\{\Psi_\ell\}_{\ell=0}^N$, $\{s_{mn}\}_{m,n=0}^N$, and $M$ such that for all $\lambda > \lambda_1$, we have
	\begin{equation}\label{contraction}
\|F_{\lambda, \epsilon}(\bm \varphi) - F_{\lambda, \epsilon}(\bm \psi)\|_{\lambda, \epsilon} \leq \sqrt{\frac{C}{\lambda^3}} \|\bm \varphi - \bm \psi\|_{\lambda, \epsilon}
	\end{equation}	
	for all $\bm \varphi$ and $\bm \psi$ in $H$.
 \end{Theorem}

\begin{proof}
Throughout the proof, $C>0$ denotes a generic constant depending only on the parameters in the statement of the theorem.
    The constant $C$ might vary from estimate to estimate.

Let $\bm{\varphi} = \begin{bmatrix}
	\varphi_0 & \dots &\varphi_N
\end{bmatrix}^{\top}$ and $\bm{\psi} = \begin{bmatrix}
	\psi_0 & \dots &\psi_N
\end{bmatrix}^{\top}$ be two vector-valued functions in $H$.  Set
\[
    \bu = F_{\lambda, \epsilon}(\bm{\varphi}),
    \quad
    \bv = F_{\lambda, \epsilon}(\bm{\psi}), \quad \mbox{and } 
    \quad\bh = \bu-\bv.
\]
Since $H$ is convex, for all $\theta\in(0,1)$,
\[
\theta \bv+(1-\theta)\bu=\bu-\theta(\bu-\bv)=\bu-\theta\bh\in H.
\]
Since $\bu$ is the minimizer of $J_{\bm{\varphi}}^{\lambda,\epsilon}$, for all $\theta\in(0,1)$, writing \[\bu = \begin{bmatrix}
	u_0 & \dots &u_N
\end{bmatrix}^{\top} \quad \mbox{and} \quad \bv = \begin{bmatrix}
	v_0 & \dots &v_N
\end{bmatrix}^{\top}\] gives
\begin{align*}
0
&\le
\frac{J_{\bm{\varphi}}^{\lambda,\epsilon}(\bu-\theta\bh)-J_{\bm{\varphi}}^{\lambda,\epsilon}(\bu)}{\theta} \\
&=
\frac{1}{\theta}\sum_{m=0}^N \Bigg[
\Big\|
e^{\lambda r^{-\beta}}
\Big(
\ii \sum_{n=0}^N s_{mn}(u_n-\theta h_n)
+\Delta (u_m-\theta h_m)
+\sum_{n=0}^N b_{mn}^N(\bm{\varphi},\x)\varphi_n\Big)
\Big\|_{L^2(\Omega)}^2
\\ 
&
-
\Big\|
e^{\lambda r^{-\beta}}
\Big(
\ii \sum_{n=0}^N s_{mn}u_n
+\Delta u_m
+\sum_{n=0}^N b_{mn}^N(\bm{\varphi},\x)\varphi_n
\Big)
\Big\|_{L^2(\Omega)}^2
\\&
+\lambda^3\Big(
\|e^{\lambda r^{-\beta}}(\partial_\nu u_m-f_m-\theta \partial_\nu h_m)\|_{L^2(\partial\Omega)}^2
-
\|e^{\lambda r^{-\beta}}(\partial_\nu u_m-f_m)\|_{L^2(\partial\Omega)}^2
\Big) \\
&\hspace{8cm}
+\epsilon\Big(
\|u_m-\theta h_m\|_{{H^s_{\lambda,\beta}(\Omega)}}^2-\|u_m\|_{{H^s_{\lambda,\beta}(\Omega)}}^2
\Big)
\Bigg].
\end{align*}
Expanding each square and letting $\theta\to0^+$, we obtain
\begin{multline}\label{4.3}
\sum_{m=0}^N \Bigg[
\operatorname{Re}\Big\langle
e^{2\lambda r^{-\beta}}
\Big(
\ii \sum_{n=0}^N s_{mn}u_n
+\Delta u_m
+\sum_{n=0}^N b_{mn}^N(\bm{\varphi},\x)\,\varphi_n
\Big),
\ii \sum_{n=0}^N s_{mn}h_n
+\Delta h_m
\Big\rangle_{L^2(\Omega)}
\\+\lambda^3 \operatorname{Re}\langle e^{2\lambda r^{-\beta}}(\partial_\nu u_m-f_m),\partial_\nu h_m\rangle_{L^2(\partial\Omega)}
+\epsilon \operatorname{Re}\langle u_m,h_m\rangle_{{H^s_{\lambda,\beta}(\Omega)}}
\Bigg]
\le 0.
\end{multline}
Similarly, interchanging the roles of $\bu$ and $\bv$, we obtain
\begin{multline}\label{4.4}
\sum_{m=0}^N \Bigg[
\operatorname{Re}\Big\langle
e^{2\lambda r^{-\beta}}
\Big(
\ii \sum_{n=0}^N s_{mn}v_n
+\Delta v_m
+\sum_{n=0}^N b_{mn}^N(\bm{\psi},\x)\,\psi_n
\Big),
\ii \sum_{n=0}^N s_{mn}h_n
+\Delta h_m
\Big\rangle_{L^2(\Omega)}
\\
+\lambda^3 \operatorname{Re}\langle e^{2\lambda r^{-\beta}}(\partial_\nu v_m-f_m),\partial_\nu h_m\rangle_{L^2(\partial\Omega)}
+\epsilon \operatorname{Re}\langle v_m,h_m\rangle_{{H^s_{\lambda,\beta}(\Omega)}}
\Bigg]
\ge 0.
\end{multline}
Subtracting \eqref{4.4} from \eqref{4.3}, we obtain
\begin{multline}\label{4.5}
\sum_{m=0}^N \Bigg[
\operatorname{Re}\Big\langle
e^{2\lambda r^{-\beta}}
\Big( 
\ii \sum_{n=0}^N s_{mn}h_n
+\Delta h_m
+\sum_{n=0}^N b_{mn}^N(\bm{\varphi},\x)\,\varphi_n
-\sum_{n=0}^N b_{mn}^N(\bm{\psi},\x)\,\psi_n
\Big),
\\
\ii \sum_{n=0}^N s_{mn}h_n
+\Delta h_m
\Big\rangle_{L^2(\Omega)}
+\lambda^3 \int_{\partial\Omega} e^{2\lambda r^{-\beta}} |\partial_\nu h_m|^2\,d\sigma(\x)
+\epsilon \|h_m\|_{{H^s_{\lambda,\beta}(\Omega)}}^2
\Bigg]
\le 0.
\end{multline}
Rearranging \eqref{4.5} and moving the frozen nonlinear term to the right-hand side, we obtain
\begin{multline}\label{4.6}
\sum_{m=0}^N \Bigg[
\int_\Omega e^{2\lambda r^{-\beta}}
\Big|
\ii \sum_{n=0}^N s_{mn}h_n+\Delta h_m
\Big|^2\,d\x
+\lambda^3 \int_{\partial\Omega} e^{2\lambda r^{-\beta}} |\partial_\nu h_m|^2\,d\sigma(\x)
+\epsilon \|h_m\|_{{H^s_{\lambda,\beta}(\Omega)}}^2
\Bigg]
\\
\le
-\sum_{m=0}^N
\operatorname{Re}\Bigg\langle
e^{2\lambda r^{-\beta}}
\Big(
\sum_{n=0}^N b_{mn}^N(\bm{\varphi},\x)\,\varphi_n
-\sum_{n=0}^N b_{mn}^N(\bm{\psi},\x)\,\psi_n
\Big),
\ii \sum_{n=0}^N s_{mn}h_n+\Delta h_m
\Bigg\rangle_{L^2(\Omega)}.
\end{multline}
Applying the inequality $ab \le \frac12(a^2+b^2)$ to the right-hand side of \eqref{4.6}, we obtain
\begin{multline}\label{4.7}
\sum_{m=0}^N \Bigg[
\int_\Omega e^{2\lambda r^{-\beta}}
\Big|
\ii \sum_{n=0}^N s_{mn}h_n+\Delta h_m
\Big|^2\,d\x
+2\lambda^3 \int_{\partial\Omega} e^{2\lambda r^{-\beta}} |\partial_\nu h_m|^2\,d\sigma(\x)
+2\epsilon \|h_m\|_{{H^s_{\lambda,\beta}(\Omega)}}^2
\Bigg]
\\
\le
 \sum_{m=0}^N \int_\Omega e^{2\lambda r^{-\beta}}
\Big|
\sum_{n=0}^N b_{mn}^N(\bm{\varphi},\x)\,\varphi_n
-\sum_{n=0}^N b_{mn}^N(\bm{\psi},\x)\,\psi_n
\Big|^2\,d\x
\end{multline}
Since $H$ is bounded in $[L^\infty(\Omega)]^{N+1}$, the map
\[
\bm\zeta \mapsto \sum_{n=0}^N b_{mn}^N(\bm\zeta,\x)\,\zeta_n
\]
is locally Lipschitz on $H$, uniformly in $\x\in\Omega$. Hence,\[
\left|
\sum_{n=0}^N b_{mn}^N(\bm{\varphi},\x)\,\varphi_n
-\sum_{n=0}^N b_{mn}^N(\bm{\psi},\x)\,\psi_n
\right|
\le C |\bm\varphi(\x)-\bm\psi(\x)|.
\]
It follows from \eqref{4.7} that
\begin{multline}\label{4.8}
\sum_{m=0}^N \Bigg[
\int_\Omega e^{2\lambda r^{-\beta}}
\Big|
\ii \sum_{n=0}^N s_{mn}h_n+\Delta h_m
\Big|^2\,d\x
+2\lambda^3 \int_{\partial\Omega} e^{2\lambda r^{-\beta}} |\partial_\nu h_m|^2\,d\sigma(\x)
+2\epsilon \|h_m\|_{{H^s_{\lambda,\beta}(\Omega)}}^2
\Bigg]
\\
\le
C \int_\Omega e^{2\lambda r^{-\beta}} |\bm\varphi-\bm\psi|^2\,d\x.
\end{multline}
Using the inequality $|a - b|^2 \geq\frac 12 |a|^2 - |b|^2$ gives
\begin{multline}\label{4.9}
\sum_{m=0}^N \Bigg[
\frac 12\int_\Omega e^{2\lambda r^{-\beta}}
|
\Delta h_m
|^2\,d\x
-C \int_\Omega e^{2\lambda r^{-\beta}}
|
 h_m
|^2\,d\x
\\
+2\lambda^3 \int_{\partial\Omega} e^{2\lambda r^{-\beta}} |\partial_\nu h_m|^2\,d\sigma(\x)
+2\epsilon \|h_m\|_{{H^s_{\lambda,\beta}(\Omega)}}^2
\Bigg]
\le
C \int_\Omega e^{2\lambda r^{-\beta}} |\bm\varphi-\bm\psi|^2\,d\x.
\end{multline}

By applying the Carleman estimate \eqref{3.42} with $u=h_m$, we find
\begin{multline}\label{4.10}
\int_{\Omega} e^{2\lambda r^{-\beta}} |\Delta h_m|^2\,d\x
\\
\ge
C \int_{\Omega} e^{2\lambda r^{-\beta}}
\bigl[
\lambda^3 |h_m|^2+\lambda |\nabla h_m|^2
\bigr]\,d\x
-
C\lambda\int_{\partial\Omega} e^{2\lambda r^{-\beta}}|\partial_{\nu} h_m|^2 d\sigma(\x).
\end{multline}
Combining \eqref{4.9} and \eqref{4.10} gives
\begin{multline}\label{4.11}
 \lambda^3 \int_\Omega e^{2\lambda r^{-\beta}}
|
 \bh
|^2\,d\x
\\
+2\lambda^3 \int_{\partial\Omega} e^{2\lambda r^{-\beta}} |\partial_\nu \bh|^2\,d\sigma(\x)
+2\epsilon \|\bh\|_{{[H^s_{\lambda,\beta}(\Omega)]^{N+1}}}^2
\le
C \int_\Omega e^{2\lambda r^{-\beta}} |\bm\varphi-\bm\psi|^2\,d\x.
\end{multline}
Add the nonnegative term \[2C\int_{\partial\Omega} e^{2\lambda r^{-\beta}} |\partial_\nu (\bm\varphi-\bm\psi)|^2\,d\sigma(\x)
+\frac{2C\epsilon}{\lambda^3} \|\bm\varphi-\bm\psi\|_{{[H^s_{\lambda,\beta}(\Omega)]^{N+1}}}^2\] into the right-hand side of \eqref{4.11} and recall $\bh = \bu - \bv$. We obtain
\begin{multline}
 \lambda^3 \Big[\int_\Omega e^{2\lambda r^{-\beta}}
|
 \bu - \bv
|^2\,d\x
+2 \int_{\partial\Omega} e^{2\lambda r^{-\beta}} |\partial_\nu (\bu - \bv)|^2\,d\sigma(\x)
+\frac{2\epsilon}{\lambda^3} \|\bu - \bv\|_{{[H^s_{\lambda,\beta}(\Omega)]^{N+1}}}^2
\\
\le
C \Big[\int_\Omega e^{2\lambda r^{-\beta}} |\bm\varphi-\bm\psi|^2\,d\x
+ 2 \int_{\partial\Omega} e^{2\lambda r^{-\beta}} |\partial_\nu (\bm\varphi-\bm\psi)|^2\,d\sigma(\x)
+\frac{2\epsilon}{\lambda^3} \|\bm\varphi-\bm\psi\|_{{[H^s_{\lambda,\beta}(\Omega)]^{N+1}}}^2\Big].
\end{multline}
Estimate \eqref{contraction} follows.
\end{proof}

\begin{Corollary}\label{cor:picard}
Let $\epsilon>0$, and let $\lambda\ge \lambda_1$, where $\lambda_1$ is as in Theorem \ref{thm:contraction}. For an arbitrary initial guess $\bu^{(0)}\in H$, define the Picard iteration
\[
\bu^{(k+1)}=F_{\lambda,\epsilon}(\bu^{(k)}), \qquad k=0,1,2,\dots.
\]
Then the sequence $\{\bu^{(k)}\}_{k=0}^\infty$ converges in $(H,\|\cdot\|_{\lambda,\epsilon})$ to a unique fixed point $\bu_{\lambda,\epsilon}\in H$ satisfying
\[
F_{\lambda,\epsilon}(\bu_{\lambda,\epsilon})=\bu_{\lambda,\epsilon}.
\]
More precisely, $\bu_{\lambda,\epsilon}$ is the unique minimizer of the functional $J_{\bu_{\lambda,\epsilon}}^{\lambda,\epsilon}$ over the admissible set $H$.

In addition, if $\mu \in(0,1)$ denotes the contraction constant of $F_{\lambda,\epsilon}$, then
\[
\|\bu^{(k)}-\bu_{\lambda,\epsilon}\|_{\lambda,\epsilon}
\le \mu ^k \|\bu^{(0)}-\bu_{\lambda,\epsilon}\|_{\lambda,\epsilon},
\qquad k\ge 0.
\]
\end{Corollary}

\begin{proof}
By Theorem \ref{thm:contraction}, the map $F_{\lambda,\epsilon}:H\to H$ is contractive with respect to the norm $\|\cdot\|_{\lambda,\epsilon}$. Since $H$ is a closed subset of $[H^s(\Omega)]^{N+1}$ and the norm $\|\cdot\|_{\lambda,\epsilon}$ is equivalent to the norm of $[H^s(\Omega)]^{N+1}$, the metric space $(H,\|\cdot\|_{\lambda,\epsilon})$ is complete. Therefore, the conclusion follows from the Banach fixed-point theorem.
\end{proof}

The uniqueness in Corollary~\ref{cor:picard} refers to the fixed point of $F_{\lambda,\epsilon}$ on $H$, that is, to the regularized truncated reduced problem. It does not imply uniqueness for the original inverse problem, Problem~\ref{idp}.

\begin{Remark}[Dependence on the truncation number]
\label{rem:truncation_lambda}
Recall that $S_N=(s_{mn})_{m,n=0}^N$. By the orthonormality of the temporal basis and \eqref{eq:psi_prime_growth},
\[
\|S_N\|_{\mathrm{op}}^2
\le \sum_{n=0}^N\sum_{m=0}^N |s_{mn}|^2
\le \sum_{n=0}^N \|\Psi_n'\|_{L^2_{e^{-2t}}(0,T)}^2
\le C_0(N+1)^4,
\]
where $C_0$ is independent of $N$ for fixed problem parameters. Hence, the linear modal coupling term in \eqref{4.9} can be absorbed in the proof of Theorem~\ref{thm:contraction} under a sufficient condition of the form $\lambda^3\ge C_1(N+1)^4$, or equivalently $\lambda\ge C_2(N+1)^{4/3}$. This gives an $O((N+1)^{4/3})$ sufficient scaling for the part of the threshold arising from $S_N$.
\end{Remark}

\section{The consistency of the fixed-point}\label{sec5}

In inverse problems, it is essential to address noisy data. If the boundary measurement $f(\x,t)$, $(\x,t)\in \partial\Omega\times(0,T)$, in \eqref{data} is contaminated by noise, then the induced boundary data $f_m$, $m=0,\dots,N$, in \eqref{15} are also noisy. Let $f_m^*$, $m=0,\dots,N$, denote the corresponding exact data. Let $\bu^* = \begin{bmatrix}
u_0& \dots, u_N
\end{bmatrix}^\top$ be the solution of the time-dimensional reduction model associated with the exact data, that is, $\bu^*$ solves
\begin{equation} \label{5.1}
\begin{cases}
\ii \displaystyle\sum_{n=0}^N s_{mn} u_n^*(\x)
+ \Delta u_m^*(\x)
+ \displaystyle\sum_{n=0}^N b_{mn}^N(\bu^*,\x)\,u_n^*(\x)
=0, & \x \in \Omega,\\
u_m^*(\x) = 0, & \x \in \partial\Omega,\\
\partial_{\nu} u_m^*(\x) = f_m^*(\x), & \x \in \partial\Omega,
\end{cases}
\qquad m=0,1,\dots,N.
\end{equation}
In this section, we show that the fixed point $\bu_{\lambda,\epsilon}$ is close to $\bu^*$.
Writing $\bff^* = \begin{bmatrix}
f_0^*& \dots, f_N^*
\end{bmatrix}^\top$ and $\bff = \begin{bmatrix}
f_0& \dots, f_N
\end{bmatrix}^\top$, we have the theorem.

\begin{Theorem}
\label{thm:stability}
Assume that $\bu^*\in H$ solves the exact reduced system \eqref{5.1}.
Let $\bu_{\lambda,\epsilon}\in H$ be the fixed point associated with the noisy data
$\bff=(f_0,\dots,f_N)^\top$, and let $\bff^*=(f_0^*,\dots,f_N^*)^\top$ denote the exact data.
Fix $\epsilon>0$ and $\beta\ge \beta_0$. Then there exist $\lambda_1\ge \lambda_0$ and $C>0$,
depending only on $\Omega$, $\beta$, $s$, $N$, $p$, $T$,
$\|q\|_{L^\infty(\Omega\times(0,T))}$, $\{\Psi_\ell\}_{\ell=0}^N$,
$\{s_{mn}\}_{m,n=0}^N$, and $M$, such that for all $\lambda\ge \lambda_1$,
\begin{equation}\label{5.2}
\|\bu_{\lambda,\epsilon}-\bu^*\|_{\lambda,\epsilon}^2
\le
C\frac{\epsilon}{\lambda^3}\|\bu^*\|_{{[H^s_{\lambda,\beta}(\Omega)]^{N+1}}}^2
+
C\int_{\partial\Omega} e^{2\lambda r^{-\beta}} |\bff-\bff^*|^2\,d\sigma(\x).
\end{equation}
\end{Theorem}

\begin{proof}
Set
\[
\bh:=\bu_{\lambda,\epsilon}-\bu^*
=
\begin{bmatrix}
h_0 & \dots & h_N
\end{bmatrix}^{\top}.
\]
Since $\bu^*\in H$ and $H$ is convex, the same argument used to derive \eqref{4.3} yields
\begin{multline}\label{5.3}
\sum_{m=0}^N \Bigg[
\operatorname{Re}\Big\langle
e^{2\lambda r^{-\beta}}
\Big(
\ii \sum_{n=0}^N s_{mn}u_{\lambda,\epsilon,n}
+\Delta u_{\lambda,\epsilon,m}
+\sum_{n=0}^N b_{mn}^N(\bu_{\lambda,\epsilon},\x)\,u_{\lambda,\epsilon,n}
\Big),
\ii \sum_{n=0}^N s_{mn}h_n+\Delta h_m
\Big\rangle_{L^2(\Omega)}
\\
+\lambda^3 \operatorname{Re}\Big\langle
e^{2\lambda r^{-\beta}}(\partial_\nu u_{\lambda,\epsilon,m}-f_m),
\partial_\nu h_m
\Big\rangle_{L^2(\partial\Omega)}
+\epsilon\,\operatorname{Re}\langle u_{\lambda,\epsilon,m},h_m\rangle_{{H^s_{\lambda,\beta}(\Omega)}}
\Bigg]
\le 0.
\end{multline}
Using $u_{\lambda,\epsilon,m}=u_m^*+h_m$ and the fact that $\bu^*$ solves \eqref{5.1}, we obtain
\begin{multline}\label{5.4}
\sum_{m=0}^N \Bigg[
\int_\Omega e^{2\lambda r^{-\beta}}
\left|\ii\sum_{n=0}^N s_{mn}h_n+\Delta h_m\right|^2\,d\x
+\lambda^3\int_{\partial\Omega} e^{2\lambda r^{-\beta}}|\partial_\nu h_m|^2\,d\sigma(\x)
+\epsilon\|h_m\|_{{H^s_{\lambda,\beta}(\Omega)}}^2
\Bigg]
\\
\le
-\sum_{m=0}^N \operatorname{Re}\Bigg\langle
e^{2\lambda r^{-\beta}}
\Big(
\sum_{n=0}^N b_{mn}^N(\bu_{\lambda,\epsilon},\x)\,u_{\lambda,\epsilon,n}
-
\sum_{n=0}^N b_{mn}^N(\bu^*,\x)\,u_n^*
\Big),
\ii\sum_{n=0}^N s_{mn}h_n+\Delta h_m
\Bigg\rangle_{L^2(\Omega)}
\\
-\lambda^3\sum_{m=0}^N \operatorname{Re}\Big\langle
e^{2\lambda r^{-\beta}}(f_m^*-f_m),\partial_\nu h_m
\Big\rangle_{L^2(\partial\Omega)}
-\epsilon\sum_{m=0}^N \operatorname{Re}\langle u_m^*,h_m\rangle_{{H^s_{\lambda,\beta}(\Omega)}}.
\end{multline}

We now estimate the three terms on the right-hand side of \eqref{5.4}. Since $H$ is bounded in
$[L^\infty(\Omega)]^{N+1}$, the map
\[
\bm{\zeta}\longmapsto \sum_{n=0}^N b_{mn}^N(\bm{\zeta},\x)\,\zeta_n
\]
is locally Lipschitz on $H$, uniformly in $\x\in\Omega$. Hence
\[
\left|
\sum_{n=0}^N b_{mn}^N(\bu_{\lambda,\epsilon},\x)\,u_{\lambda,\epsilon,n}
-
\sum_{n=0}^N b_{mn}^N(\bu^*,\x)\,u_n^*
\right|
\le C|\bh(\x)|.
\]
Therefore, by $ab\le \frac12(a^2+b^2)$,
\begin{multline}\label{5.5}
\left|
\operatorname{Re}\Bigg\langle
e^{2\lambda r^{-\beta}}
\Big(
\sum_{n=0}^N b_{mn}^N(\bu_{\lambda,\epsilon},\x)\,u_{\lambda,\epsilon,n}
-
\sum_{n=0}^N b_{mn}^N(\bu^*,\x)\,u_n^*
\Big),
\ii\sum_{n=0}^N s_{mn}h_n+\Delta h_m
\Bigg\rangle_{L^2(\Omega)}
\right|
\\
\le
\frac12\int_\Omega e^{2\lambda r^{-\beta}}
\left|\ii\sum_{n=0}^N s_{mn}h_n+\Delta h_m\right|^2\,d\x
+
C\int_\Omega e^{2\lambda r^{-\beta}}|\bh|^2\,d\x.
\end{multline}
Similarly,
\begin{multline}\label{5.6}
\lambda^3\left|
\operatorname{Re}\Big\langle
e^{2\lambda r^{-\beta}}(f_m^*-f_m),\partial_\nu h_m
\Big\rangle_{L^2(\partial\Omega)}
\right|
\\
\le
\frac{\lambda^3}{2}\int_{\partial\Omega} e^{2\lambda r^{-\beta}}|\partial_\nu h_m|^2\,d\sigma(\x)
+
\frac{\lambda^3}{2}\int_{\partial\Omega} e^{2\lambda r^{-\beta}}|f_m-f_m^*|^2\,d\sigma(\x),
\end{multline}
and
\begin{equation}\label{5.7}
\epsilon\left|\operatorname{Re}\langle u_m^*,h_m\rangle_{{H^s_{\lambda,\beta}(\Omega)}}\right|
\le
\frac{\epsilon}{2}\|h_m\|_{{H^s_{\lambda,\beta}(\Omega)}}^2
+
\frac{\epsilon}{2}\|u_m^*\|_{{H^s_{\lambda,\beta}(\Omega)}}^2.
\end{equation}

Substituting \eqref{5.5}--\eqref{5.7} into \eqref{5.4} and absorbing the half terms into the left-hand side, we obtain
\begin{multline}\label{5.8}
\sum_{m=0}^N \Bigg[
\int_\Omega e^{2\lambda r^{-\beta}}
\left|\ii\sum_{n=0}^N s_{mn}h_n+\Delta h_m\right|^2\,d\x
+\lambda^3\int_{\partial\Omega} e^{2\lambda r^{-\beta}}|\partial_\nu h_m|^2\,d\sigma(\x)
+\epsilon\|h_m\|_{{H^s_{\lambda,\beta}(\Omega)}}^2
\Bigg]
\\
\le
C\int_\Omega e^{2\lambda r^{-\beta}}|\bh|^2\,d\x
+
C\lambda^3\int_{\partial\Omega} e^{2\lambda r^{-\beta}}|\bff-\bff^*|^2\,d\sigma(\x)
+
C\epsilon\|\bu^*\|_{{[H^s_{\lambda,\beta}(\Omega)]^{N+1}}}^2.
\end{multline}

Next, using
\[
|a+b|^2\ge \frac12|a|^2-|b|^2
\]
with
\[
a=\Delta h_m,
\qquad
b=\ii\sum_{n=0}^N s_{mn}h_n,
\]
we obtain
\[
\left|\ii\sum_{n=0}^N s_{mn}h_n+\Delta h_m\right|^2
\ge
\frac12|\Delta h_m|^2-C|\bh|^2.
\]
Hence \eqref{5.8} implies
\begin{multline}\label{5.9}
\sum_{m=0}^N \Bigg[
\frac12\int_\Omega e^{2\lambda r^{-\beta}}|\Delta h_m|^2\,d\x
+\lambda^3\int_{\partial\Omega} e^{2\lambda r^{-\beta}}|\partial_\nu h_m|^2\,d\sigma(\x)
+\epsilon\|h_m\|_{{H^s_{\lambda,\beta}(\Omega)}}^2
\Bigg]
\\
\le
C\int_\Omega e^{2\lambda r^{-\beta}}|\bh|^2\,d\x
+
C\lambda^3\int_{\partial\Omega} e^{2\lambda r^{-\beta}}|\bff-\bff^*|^2\,d\sigma(\x)
+
C\epsilon\|\bu^*\|_{{[H^s_{\lambda,\beta}(\Omega)]^{N+1}}}^2.
\end{multline}

Since both $\bu_{\lambda,\epsilon}$ and $\bu^*$ satisfy the homogeneous Dirichlet boundary condition, we have
$h_m=0$ on $\partial\Omega$. Therefore, applying the Carleman estimate \eqref{3.42} to each $h_m$, we get
\begin{equation}\label{5.10}
\int_\Omega e^{2\lambda r^{-\beta}}|\Delta h_m|^2\,d\x
\ge
C\int_\Omega e^{2\lambda r^{-\beta}}
\bigl(\lambda^3|h_m|^2+\lambda|\nabla h_m|^2\bigr)\,d\x
-
C\lambda\int_{\partial\Omega} e^{2\lambda r^{-\beta}}|\partial_\nu h_m|^2\,d\sigma(\x).
\end{equation}
Substituting \eqref{5.10} into \eqref{5.9}, we obtain
\begin{multline}\label{5.11}
\sum_{m=0}^N \Bigg[
C\int_\Omega e^{2\lambda r^{-\beta}}
\bigl(\lambda^3|h_m|^2+\lambda|\nabla h_m|^2\bigr)\,d\x
-C\lambda\int_{\partial\Omega} e^{2\lambda r^{-\beta}}|\partial_\nu h_m|^2\,d\sigma(\x)
\\
+\lambda^3\int_{\partial\Omega} e^{2\lambda r^{-\beta}}|\partial_\nu h_m|^2\,d\sigma(\x)
+\epsilon\|h_m\|_{{H^s_{\lambda,\beta}(\Omega)}}^2
\Bigg]
\\
\le
C\int_\Omega e^{2\lambda r^{-\beta}}|\bh|^2\,d\x
+
C\lambda^3\int_{\partial\Omega} e^{2\lambda r^{-\beta}}|\bff-\bff^*|^2\,d\sigma(\x)
+
C\epsilon\|\bu^*\|_{{[H^s_{\lambda,\beta}(\Omega)]^{N+1}}}^2.
\end{multline}

Choosing $\lambda$ sufficiently large, we absorb the lower-order bulk term on the right-hand side and the
$C\lambda$ boundary term on the left into the corresponding $\lambda^3$ terms. Consequently,
\begin{multline}\label{5.12}
\lambda^3\int_\Omega e^{2\lambda r^{-\beta}}|\bh|^2\,d\x
+
\lambda^3\int_{\partial\Omega} e^{2\lambda r^{-\beta}}|\partial_\nu\bh|^2\,d\sigma(\x)
+
\epsilon\|\bh\|_{{[H^s_{\lambda,\beta}(\Omega)]^{N+1}}}^2
\\
\le
C\lambda^3\int_{\partial\Omega} e^{2\lambda r^{-\beta}}|\bff-\bff^*|^2\,d\sigma(\x)
+
C\epsilon\|\bu^*\|_{{[H^s_{\lambda,\beta}(\Omega)]^{N+1}}}^2.
\end{multline}
Dividing by $\lambda^3$ yields
\begin{multline}\label{5.13}
\int_\Omega e^{2\lambda r^{-\beta}}|\bh|^2\,d\x
+
\int_{\partial\Omega} e^{2\lambda r^{-\beta}}|\partial_\nu\bh|^2\,d\sigma(\x)
+
\frac{\epsilon}{\lambda^3}\|\bh\|_{{[H^s_{\lambda,\beta}(\Omega)]^{N+1}}}^2
\\
\le
C\int_{\partial\Omega} e^{2\lambda r^{-\beta}}|\bff-\bff^*|^2\,d\sigma(\x)
+
C\frac{\epsilon}{\lambda^3}\|\bu^*\|_{{[H^s_{\lambda,\beta}(\Omega)]^{N+1}}}^2.
\end{multline}
By the definition of $\|\cdot\|_{\lambda,\epsilon}$ and after adjusting the constant $C$, \eqref{5.2} follows.
\end{proof}

\begin{Corollary}[Stability of the reconstructed initial field]
\label{cor:stability_initial}
Under the hypotheses of Theorem~\ref{thm:stability}, define
\[
u^{\mathrm{rec}}_{\lambda,\epsilon}(\x)
=
\sum_{n=0}^{N}u_{\lambda,\epsilon,n}(\x)\Psi_n(0)
\]
and
\[
u_N^*(\x)
=
\sum_{n=0}^{N}u_n^*(\x)\Psi_n(0).
\]
Since the basis functions $\Psi_n\in C([0,T])$, the values $\Psi_n(0)$
are well-defined real constants for $n=0,\ldots,N$. Consequently, the
finite linear combination
\[
(u_0,\ldots,u_N)
\mapsto
\sum_{n=0}^{N}u_n\Psi_n(0)
\]
defines a bounded linear operator from $[L^2(\Omega)]^{N+1}$ into
$L^2(\Omega)$, with operator norm bounded by
$\bigl(\sum_{n=0}^{N}|\Psi_n(0)|^2\bigr)^{1/2}$. In particular,
\[
\left\|
u^{\mathrm{rec}}_{\lambda,\epsilon}-u_N^*
\right\|_{L^2(\Omega)}
\leq
C_N
\left\|
\bu_{\lambda,\epsilon}-\bu^*
\right\|_{[L^2(\Omega)]^{N+1}} .
\]
Combining this estimate with Theorem~\ref{thm:stability}, we obtain a
stability estimate for the reconstructed initial field of the truncated
system with respect to the boundary data perturbation and the
regularization parameter.

The quantity $u_N^*$ is the reconstructed initial field associated with
the $N$-mode truncated expansion. The difference between $u_N^*$ and the
true initial datum $u^0$ is a separate truncation error, whose behavior
depends on the regularity of the solution and the convergence properties
of the Legendre expansion.

The stability estimate is established for the truncated system with
finitely many temporal modes. Its constant depends on $N$, so the
estimate does not establish stability uniformly in $N$ for the
original time-dependent inverse problem.
\end{Corollary}

\begin{Proposition}[Truncation residual estimate]
\label{prop:truncation_residual}
Let $u$ be an exact solution of the nonlinear Schr\"odinger equation
with $q\in L^\infty(\Omega\times(0,T))$ and $p>1$. Assume that
\[
u\in H^\ell((0,T);H^s(\Omega)),
\qquad \ell\ge3,\qquad s>d/2+2 .
\]
Let
\[
u(\x,t)=\sum_{n=0}^{\infty}u_n(\x)\Psi_n(t),
\qquad
u^N(\x,t)=\sum_{n=0}^{N}u_n(\x)\Psi_n(t).
\]
For each $m=0,\ldots,N$, the first $N+1$ modal equations of the exact
solution can be written, in $L^2(\Omega)$, as
\[
\ii\sum_{n=0}^{N}s_{mn}u_n
+\Delta u_m
+\mathcal{N}_{m}^{N}(\bu^N,\x)
=R_m^N,
\]
where
\[
\mathcal{N}_{m}^{N}(\bu^N,\x)
=
\int_0^T e^{-2t}q(\x,t)
|u^N(\x,t)|^{p-1}u^N(\x,t)\Psi_m(t)\,dt
\]
and
\[
\begin{aligned}
R_m^N(\x)
={}&-\ii\sum_{n=N+1}^{\infty}s_{mn}u_n(\x)\\
&-\int_0^T e^{-2t}q(\x,t)
\left(
|u(\x,t)|^{p-1}u(\x,t)
-
|u^N(\x,t)|^{p-1}u^N(\x,t)
\right)\Psi_m(t)\,dt .
\end{aligned}
\]
Then, for every fixed $m\ge0$,
\[
\|R_m^N\|_{L^2(\Omega)}\longrightarrow0
\qquad\text{as }N\to\infty,\quad N\ge m .
\]
\end{Proposition}

\begin{proof}
The coefficient decay estimate \eqref{eq:coeff_decay}, applied with
values in $H^s(\Omega)$, gives
\[
\sum_{n=0}^{\infty}n^{2\ell}
\|u_n\|_{H^s(\Omega)}^2
\le C\|u\|_{H^\ell((0,T);H^s(\Omega))}^2.
\]
In particular, for $n\ge1$,
\[
\|u_n\|_{H^s(\Omega)}
\le Cn^{-\ell}
\|u\|_{H^\ell((0,T);H^s(\Omega))}.
\]

By orthonormality and \eqref{eq:psi_prime_growth}, for $n\ge1$,
\[
|s_{mn}|
=
\left|
\langle\Psi_n',\Psi_m\rangle_{L^2_{e^{-2t}}(0,T)}
\right|
\le
\|\Psi_n'\|_{L^2_{e^{-2t}}(0,T)}
\|\Psi_m\|_{L^2_{e^{-2t}}(0,T)}
\le Cn^{3/2}.
\]
Consequently, for $N\ge m$,
\[
\begin{aligned}
\left\|\sum_{n>N}s_{mn}u_n\right\|_{L^2(\Omega)}
&\le \sum_{n>N}|s_{mn}|\|u_n\|_{H^s(\Omega)}\\
&\le C\|u\|_{H^\ell((0,T);H^s(\Omega))}
       \sum_{n>N}n^{3/2-\ell}
\longrightarrow0,
\end{aligned}
\]
since $\ell\ge3$.

It remains to estimate the nonlinear term. The shifted Legendre
polynomials satisfy
$\|\Psi_n\|_{L^\infty(0,T)}\le C(n+1)^{1/2}$.
Together with the coefficient estimate and $\ell\ge3$, this yields
\[
\sup_{N\ge0}\sup_{t\in[0,T]}
\|u^N(\cdot,t)\|_{H^s(\Omega)}
\le
C\left(
\|u_0\|_{H^s(\Omega)}
+
\|u\|_{H^\ell((0,T);H^s(\Omega))}
\sum_{n=1}^{\infty}n^{1/2-\ell}
\right)
<\infty.
\]
Also, $u\in C([0,T];H^s(\Omega))$. Since
$H^s(\Omega)\hookrightarrow L^\infty(\Omega)$, there is a constant
$M_u$ such that
\[
|u(\x,t)|,\ |u^N(\x,t)|\le M_u
\quad\text{for almost every }(\x,t)\in\Omega\times(0,T)
\quad\text{and every }N.
\]
The scalar map $F(z)=|z|^{p-1}z$ is Lipschitz on the disk
$\{|z|\le M_u\}$. Therefore,
\[
\|F(u)-F(u^N)\|_{L^2_{e^{-2t}}((0,T);L^2(\Omega))}
\le
C_{M_u,p}
\|u-u^N\|_{L^2_{e^{-2t}}((0,T);L^2(\Omega))}
\longrightarrow0,
\]
where the final convergence follows from convergence of the
orthonormal expansion.

For each fixed $m$, the Cauchy--Schwarz inequality in time and
$\|\Psi_m\|_{L^2_{e^{-2t}}(0,T)}=1$ now give
\[
\begin{aligned}
&\left\|
\int_0^T e^{-2t}q(\cdot,t)
\bigl(F(u(\cdot,t))-F(u^N(\cdot,t))\bigr)
\Psi_m(t)\,dt
\right\|_{L^2(\Omega)}\\
&\qquad\le
\|q\|_{L^\infty(\Omega\times(0,T))}
\|F(u)-F(u^N)\|_{L^2_{e^{-2t}}((0,T);L^2(\Omega))}
\longrightarrow0.
\end{aligned}
\]
Combining the linear and nonlinear estimates proves
$\|R_m^N\|_{L^2(\Omega)}\to0$ for every fixed $m$.
\end{proof}

\begin{Remark}    
We note that the exact modal vector $\bu^*=(u_0^*,\dots,u_N^*)^\top$
should be understood as the exact solution of the truncated reduced model.
If instead $\bu^*$ is taken to be the first $N+1$ modes of the exact
solution $u$ of the original nonlinear Schr\"odinger equation, then
$\bu^*$ satisfies the truncated system only up to a truncation residual:
\[
i\sum_{n=0}^N s_{mn}u_n^*(\x)+\Delta u_m^*(\x)
+\sum_{n=0}^N b_{mn}^N(\bu^*,\x)u_n^*(\x)
=
R_m^N(\x),\qquad m=0,\dots,N.
\]
Here $R_m^N$ represents the contribution of the discarded modes and the
error caused by replacing the full nonlinear term by its $N$-mode
approximation. If
\[
u\in H^\ell((0,T);H^s(\Omega)),\qquad \ell\ge 3,\qquad s>\frac d2+2,
\]
then the Legendre-polynomial-exponential expansion of $u$ converges in
$L^2((0,T);H^s(\Omega))$. Hence the tail
\[
u(\x,t)-\sum_{\ell=0}^N u_\ell^*(\x)\Psi_\ell(t)
\]
converges to $0$ in $L^2((0,T);H^s(\Omega))$ as $N\to\infty$. Since
$H^s(\Omega)\hookrightarrow L^\infty(\Omega)$ and the nonlinear map
$z\mapsto |z|^{p-1}z$ is locally Lipschitz on bounded subsets of
$\mathbb C$, the truncation error in the nonlinear term also converges to
$0$. Therefore, for each fixed $m\ge 0$,
\[
\int_\Omega e^{2\lambda r^{-\beta}}|R_m^N(\x)|^2\,d\x
\to 0
\qquad\mbox{as }N\to\infty .
\]
Thus $R_m^N$ is precisely the residual caused by replacing the full
projected system by its $N$-mode truncation.
\end{Remark}
\begin{Remark}
Fix the Carleman parameters $\lambda$, $\beta$, and $\x_0$. Then estimate \eqref{5.2} shows that if the boundary data $\mathbf f$ are close to the exact data $\mathbf f^*$, the fixed point $\bu_{\lambda,\epsilon}$ provides an approximation of the exact reduced solution $\bu^*$. More precisely, the reconstruction error is controlled by two terms: the data discrepancy term
$
\int_{\partial\Omega} e^{2\lambda r^{-\beta}} |\mathbf f-\mathbf f^*|^2\,d\sigma(\x),
$
and the regularization term
$
\frac{\epsilon}{\lambda^3}\|\bu^*\|_{{[H^s_{\lambda,\beta}(\Omega)]^{N+1}}}^2.
$
Therefore, for fixed Carleman parameters and small $\epsilon$, if $\mathbf f\to \mathbf f^*$, then $\bu_{\lambda,\epsilon}$ is close to $\bu^*$ in the norm $\|\cdot\|_{\lambda,\epsilon}$.
In \cite{nguyen2023carleman}, the perturbations in the
Dirichlet and conormal boundary data are assumed to be the
traces of a common function $e\in H^k(\Omega)$, with its
$H^k(\Omega)$ norm controlled by the noise level. By contrast,
estimate~\eqref{5.2} involves the weighted boundary discrepancy
$\mathbf f-\mathbf f^*$ directly and does not require such a
Sobolev extension of the Neumann data perturbation.
\end{Remark}

\begin{Remark}
Theorem~\ref{thm:stability} establishes stability for the regularized, truncated reduced problem. It does not by itself establish a stability theorem for the full nonlinear problem of reconstructing the initial wave field, and this paper does not assert either stability or instability for that problem. To obtain Theorem~\ref{thm:stability}, we first approximate the original problem by the time-dimensional reduction model \eqref{15}, which is obtained by truncating the Legendre polynomial-exponential expansion to the first $N+1$ modes. This truncation removes the high-oscillation components of the solution, which are typically the most sensitive to noise, and therefore acts as a filtering mechanism. After this reduction, we solve a regularized problem for the coupled elliptic system with Cauchy data by means of the weighted functional $J_{\bm{\varphi}}^{\lambda,\epsilon}$. The resulting fixed point $\bu_{\lambda,\epsilon}$ is thus the solution
of the regularized reduced problem with finitely many temporal modes. The stability estimate in Theorem~\ref{thm:stability} should be understood in this regularized sense.
Theorem~\ref{thm:stability} assumes the existence of an exact reduced solution $\bu^*$ and compares it with $\bu_{\lambda,\epsilon}$. Uniqueness of the original inverse problem is discussed separately after Problem~\ref{idp}: it holds within the stated class of bounded strong solutions and is a working assumption for broader solution classes.
\end{Remark}

\section{Numerical study}\label{sec6}

In this section, we present the numerical study for solving Problem \ref{idp}, including the Carleman contraction method in Algorithm \ref{alg:picard}, and show some numerical results.

\begin{algorithm}[h!]
\caption{Reconstruction of the initial wave field via the Carleman--Picard iteration}
\label{alg:picard}
\begin{algorithmic}[1]

\State Fix the artificial parameters $\lambda$, $\beta$, $\x_0$, $\epsilon$, and the truncation number $N$. \label{s1}

\State Construct the Legendre polynomial-exponential basis $\{\Psi_n\}_{n=0}^N$ and compute the projected boundary data
\[
f_m(\x)=\int_0^T e^{-2t} f(\x,t)\Psi_m(t)\,dt,
\qquad m=0,\dots,N.
\]

\State Choose an initial guess \label{s3}
\[
\bu^{(0)}=\begin{bmatrix}u_0^{(0)} & u_1^{(0)} & \cdots & u_N^{(0)}\end{bmatrix}^\top \in H.
\] 

\For{$k=0,1,2,\dots,K_{\max}-1$ (for some $K_{\max}\geq 1$)}
    \State For the frozen coefficient vector $\bu^{(k)}$, compute $\bu^{(k+1)}\in H$ as the unique minimizer of
    \[
    J_{\bu^{(k)}}^{\lambda,\epsilon}(\bu).
    \]
    Equivalently, \label{s5}
    \[
    \bu^{(k+1)} = F_{\lambda,\epsilon}(\bu^{(k)}).
    \]
\EndFor

\State Set the computed coefficient vector as
\[
\bu^{\rm comp}=\bu^{(K_{\max})}.
\]

\State Write
\[
\bu^{\rm comp}=
\begin{bmatrix}
u_0^{\rm comp} & \dots & u_N^{\rm comp}
\end{bmatrix}^\top
\]
and reconstruct the approximate space-time solution by
\[
u^{\rm comp}(\x,t)
:=
\sum_{n=0}^N u_n^{\rm comp}(\x)\Psi_n(t).
\]

\State Compute the approximate solution of the inverse problem, namely the initial wave field,
\[
u^{\rm comp}(\x,0)
=
\sum_{n=0}^N u_n^{\rm comp}(\x)\Psi_n(0).
\]

\end{algorithmic}
\end{algorithm}

\subsection{Forward problem and data generation}

In this subsection, we describe the numerical procedure used to generate synthetic data for the inverse problem. The computations are carried out on the square domain
$\Omega := (-1,1)^2$. We use a uniform Cartesian grid in space and a uniform partition in time. More precisely, for the inverse problem we set
\[
N_x = 61, \qquad x_i = -1 + (i-1)h_x, \qquad y_j = -1 + (j-1)h_y,
\]
for $i,j=1,\dots,N_x$, where
\[
h_x = h_y = \frac{2}{N_x-1}.
\]
In time, we choose
\[
\Delta t = 1.25\times 10^{-4}, \qquad T=0.2,
\]
and define
\[
t_n = n\Delta t, \qquad n=0,1,\dots,N_t,
\]
where $N_t=T/\Delta t$.
To avoid the inverse crime, the forward problem is solved on a finer grid. We set
\[
N_x^f = 151, \qquad h_f = \frac{2}{N_x^f-1} = \frac{2}{150}, \qquad x_i^f = -1 + (i-1)h_f, \qquad y_j^f = -1 + (j-1)h_f,
\]
for $i,j=1,\dots,N_x^f$, and
\[
\Delta t_f = \frac{\Delta t}{2} = 6.25\times 10^{-5}, \qquad \tau_k = k\Delta t_f, \qquad k=0,1,\dots,2N_t.
\]
The boundary data used for the inversion are given at the boundary nodes of the $61\times 61$ grid above and sampled at the times $t_n=\tau_{2n}$, $n=0,1,\dots,N_t$. The inverse solver itself uses no time stepping, since the time variable is eliminated by the time-dimensional reduction.

To generate the synthetic data, we solve the forward nonlinear Schr\"odinger equation with $q(\x, t) = 1$ (for simplicity)
\begin{equation}
\begin{cases}
\ii u_t + \Delta u +  |u|^{p-1}u = 0, & (\x,t)\in \Omega\times(0,T),\\
u(\x,t)=0, & (\x,t)\in \partial\Omega \times(0,T),\\
u(\x,0)=u^{0, \rm true}(\x), & \x\in \Omega,
\end{cases}
\label{eq:forward_nls_numerics}
\end{equation}
where $u^{0, \rm true}$ is the prescribed exact initial condition. In all computations, the boundary condition is homogeneous Dirichlet. Although the theoretical analysis is carried out for a general coefficient $q=q(\x,t)$, in the numerical experiments we restrict ourselves to the representative case $q\equiv 1$ for simplicity of implementation. This choice allows us to isolate the performance of the reconstruction method without reducing the scope of the analytical results.

We discretize \eqref{eq:forward_nls_numerics} by a scheme in which the Laplacian is treated by the Crank--Nicolson method \cite{crank1947,delfour1981} while the nonlinear term is evaluated explicitly at the previous time level. Let $u_f^k$ denote the numerical approximation of $u(\cdot,\tau_k)$ on the fine grid. Then, for $k=0,1,\dots,2N_t-1$, we compute $u_f^{k+1}$ from
\[
\ii \frac{u_f^{k+1}-u_f^k}{\Delta t_f} + \frac12\Delta_{h_f}\big(u_f^{k+1}+u_f^k\big) +  |u_f^k|^{p-1}u_f^k = 0,
\]
where $\Delta_{h_f}$ is the standard five-point finite difference approximation of the Laplacian on the fine grid,
\[
(\Delta_{h_f} u)_{i,j}
=
\frac{u_{i+1,j}-2u_{i,j}+u_{i-1,j}}{h_f^2}
+
\frac{u_{i,j+1}-2u_{i,j}+u_{i,j-1}}{h_f^2}.
\]
At each time step, this scheme yields a linear system for $u_f^{k+1}$. The homogeneous Dirichlet boundary condition is enforced by setting the values of $u_f^{k+1}$ at the boundary nodes of the fine grid equal to zero.

After solving the forward problem, we compute the boundary observation
\[
f^{*}(\x,t)=\partial_\nu u(\x,t)
\qquad \text{for } (\x,t)\in \partial\Omega \times(0,T),
\]
by second-order one-sided finite differences on the boundary of the fine grid. These values are then interpolated linearly along each edge of $\partial\Omega$ to the boundary nodes of the $61\times 61$ grid and sampled at the times $t_n=\tau_{2n}$, $n=0,1,\dots,N_t$.
The noisy data are defined by
\[
f^\delta(\x,t)=f^*(\x,t)+\delta\,\frac{\|f^*\|}{\|\xi\|}\,\xi(\x,t),
\qquad \delta\in\{5\%,10\%,20\%\},
\]
where $\xi$ is a complex-valued random field whose real and imaginary parts are independent standard normal random variables at each boundary grid point of the $61\times 61$ grid and each time $t_n$, and $\|\cdot\|$ denotes the discrete $\ell^2$ norm over all these boundary grid points and times. Thus, $\|f^\delta-f^*\|/\|f^*\|=\delta$. The random seed is fixed.

 The projected data used in the reduced inverse model are then obtained by
\[
f_m^\delta(\x)
=
\int_0^T e^{-2t} f^\delta(\x,t)\Psi_m(t)\,dt,
\qquad m=0,\dots,N.
\]
These quantities serve as the noisy boundary inputs in the time-dimensional reduction model. 

\subsection{Implementation}

In this subsection, we discuss several implementation details used in our numerical computations.

In Step~\ref{s1}, the artificial parameters are selected by manual tuning. More precisely, we adjust the parameters $\lambda$, $\beta$, $\x_0$, $\epsilon$, and $N$ until satisfactory numerical performance is obtained for a reference experiment, namely Test~1. In our implementation, we use $N=64$ (that is, $65$ temporal basis functions), $\epsilon=10^{-4}$, and $K_{\max}=10$. For the Carleman weight, we choose
$\lambda=300$, $\beta=2$, and $\x_0=(0,-11)$.
With the parameters above, $r$ ranges over $[10,\sqrt{145}\,]$ on $\overline{\Omega}$, and hence
\[
e^{2\lambda r^{-\beta}}\in\big[e^{600/145},\,e^{6}\big]\approx[62.67,\;403.43]
\qquad\text{on }\overline{\Omega}.
\]
In the computation, we replace the weight $e^{2\lambda r^{-\beta}}$ by
\[
\widetilde W(\x)=\frac{e^{2\lambda r^{-\beta}(\x)}}{\max_{\overline{\Omega}}e^{2\lambda r^{-\beta}}},
\qquad
\widetilde W\in\big[e^{600/145-6},\,1\big]\approx[0.155,\;1].
\]
This normalization step does not affect the analysis, since the Carleman estimate \eqref{3.42} still holds with $e^{2\lambda r^{-\beta}}$ replaced by $\widetilde W$; indeed, both sides of \eqref{3.42} are divided by the same positive constant. Once these parameters are determined from Test~1, the same values are used for all remaining tests.

In Step~\ref{s3}, we choose the initial guess $\bu^{(0)}=\mathbf{0}\in H$.

We now discuss the implementation in Step~\ref{s5}. At the $k$th Picard step, given the current iterate $\bu^{(k)}$, we first evaluate the frozen nonlinear term
\[
\sum_{n=0}^N b_{mn}^N(\bu^{(k)},\x)\,u_n^{(k)}(\x),
\qquad m=0,\dots,N,
\]
and then keep this term fixed in the reduced system. This leads to a linear weighted least-squares problem for the next iterate $\bu^{(k+1)}$. More precisely, $\bu^{(k+1)}$ is computed by minimizing the functional $J_{\bu^{(k)}}^{\lambda,\epsilon}$, which consists of the weighted residual of the frozen reduced equations, the Neumann boundary mismatch term, and the Sobolev regularization term. In the MATLAB implementation, this minimization problem is assembled as an overdetermined linear system of the form $A\mathbf{x}\approx \mathbf{b}$, where $A$ is the system matrix, $\mathbf{x}$ is the vector of unknown discrete values of the coefficients $u_0,\dots,u_N$, and $\mathbf{b}$ is the corresponding right-hand side vector. Let $\mathcal R$ be the matrix of the discrete version of the norm $\|\cdot\|_{[H^2_{\lambda,\beta}(\Omega)]^{N+1}}$; that is, we use $s=2$ in the regularization term for simplicity of implementation. We compute $\mathbf{x}$ from the normal equations
\[
\big(A^*A+\epsilon\,\mathcal R^*\mathcal R\big)\mathbf{x}=A^*\mathbf{b}
\]
by the MATLAB backslash operator. Repeating this procedure for $k=0,1,2,\dots$ generates the sequence $\{\bu^{(k)}\}_{k\ge 0}$, whose last iterate is taken as the computed approximation of the fixed point $\bu_{\lambda,\epsilon}$.

\begin{Remark}
We note that the theoretical minimization problem is posed over the
admissible set $H$, which includes the a priori bound
$\|\bu\|_{[L^\infty(\Omega)]^{N+1}}\le M$. This constraint is used in the
analysis to guarantee that the nonlinear modal map is Lipschitz on the
admissible set and hence to prove the contraction property of
$F_{\lambda,\epsilon}$. In the numerical implementation, however, we solve
the normal equations above at each Picard step without imposing
this constraint. This is a practical
implementation of the Carleman--Picard iteration. In all numerical tests
reported below, the computed iterates remained uniformly bounded, and the
homogeneous Dirichlet condition was included through the equations
$u_m=0$ at the boundary grid points in the least-squares system.
Thus the computed solutions stayed inside a bounded discrete analogue of
the admissible set.
\end{Remark}

All other steps in Algorithm~\ref{alg:picard}, including the projection of the boundary data and the reconstruction of the space-time solution and the initial data, are straightforward to implement once the coefficient vectors have been computed.

\subsection{Numerical examples}

In this subsection, we present some numerical tests obtained by Algorithm~\ref{alg:picard}.

{\bf Test 1.} Let
\[
\mathcal B(x,y)=
\begin{cases}
\exp\!\left(1-\dfrac{1}{1-x^2-y^2}\right), & x^2+y^2<1,\\[2mm]
0, & \text{otherwise},
\end{cases}
\]
be the standard smooth bump function. We choose
\[
u^0(x,y)
=4\,\mathcal B\!\left(\frac{x+0.5}{0.5},\frac{y-0.5}{0.25}\right)
+4\,\ii\,\mathcal B\!\left(\frac{x-0.5}{0.25},\frac{y+0.5}{0.5}\right).
\]
Thus, the real part is supported on an ellipse centered at $(-0.5,0.5)$ with
semi-axes $0.5$ and $0.25$, and the imaginary part is supported on an ellipse
centered at $(0.5,-0.5)$ with semi-axes $0.25$ and $0.5$. Both components have
amplitude $4$. In this test, we set $p=3$.

\begin{figure}[ht]
\centering
\subfloat[True $\Re(u^0)$.\label{fig:case1_true_real}]{
    \includegraphics[width=0.235\textwidth]{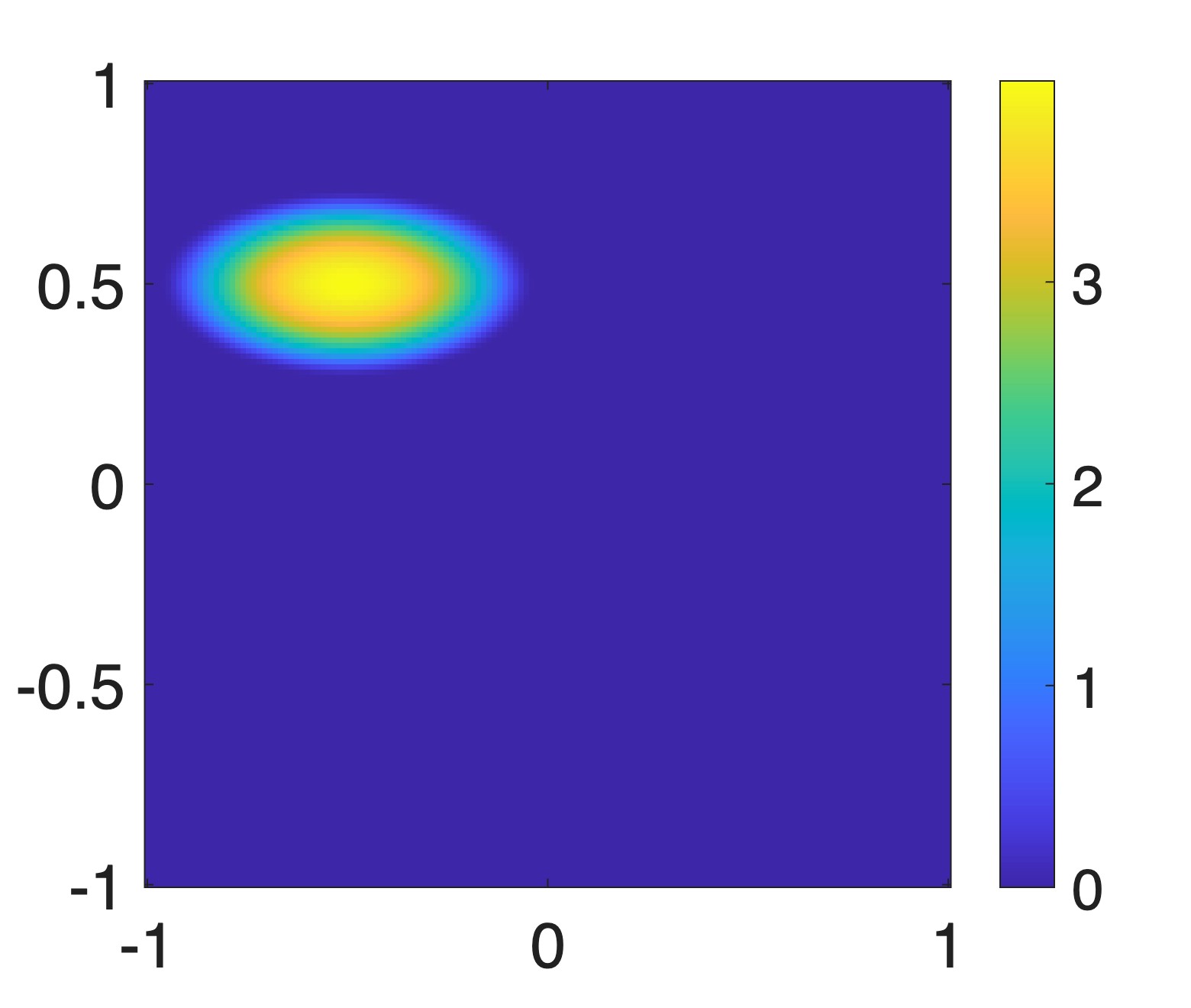}}
\hfill
\subfloat[Computed $\Re(u^0)$.\label{fig:case1_comp_real}]{
    \includegraphics[width=0.235\textwidth]{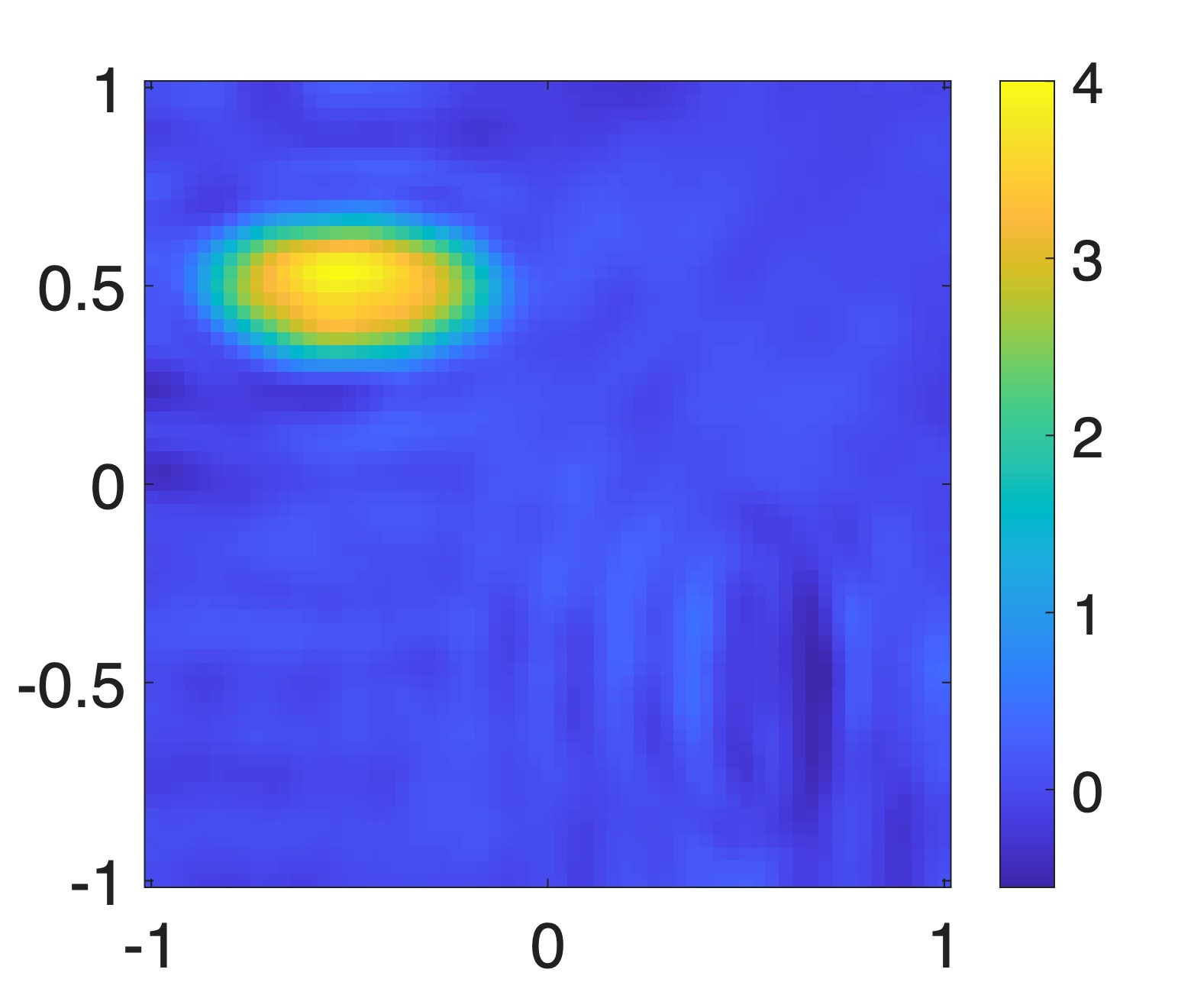}}
\hfill
\subfloat[$E_{\Re}$.\label{fig:case1_err_real}]{
    \includegraphics[width=0.235\textwidth]{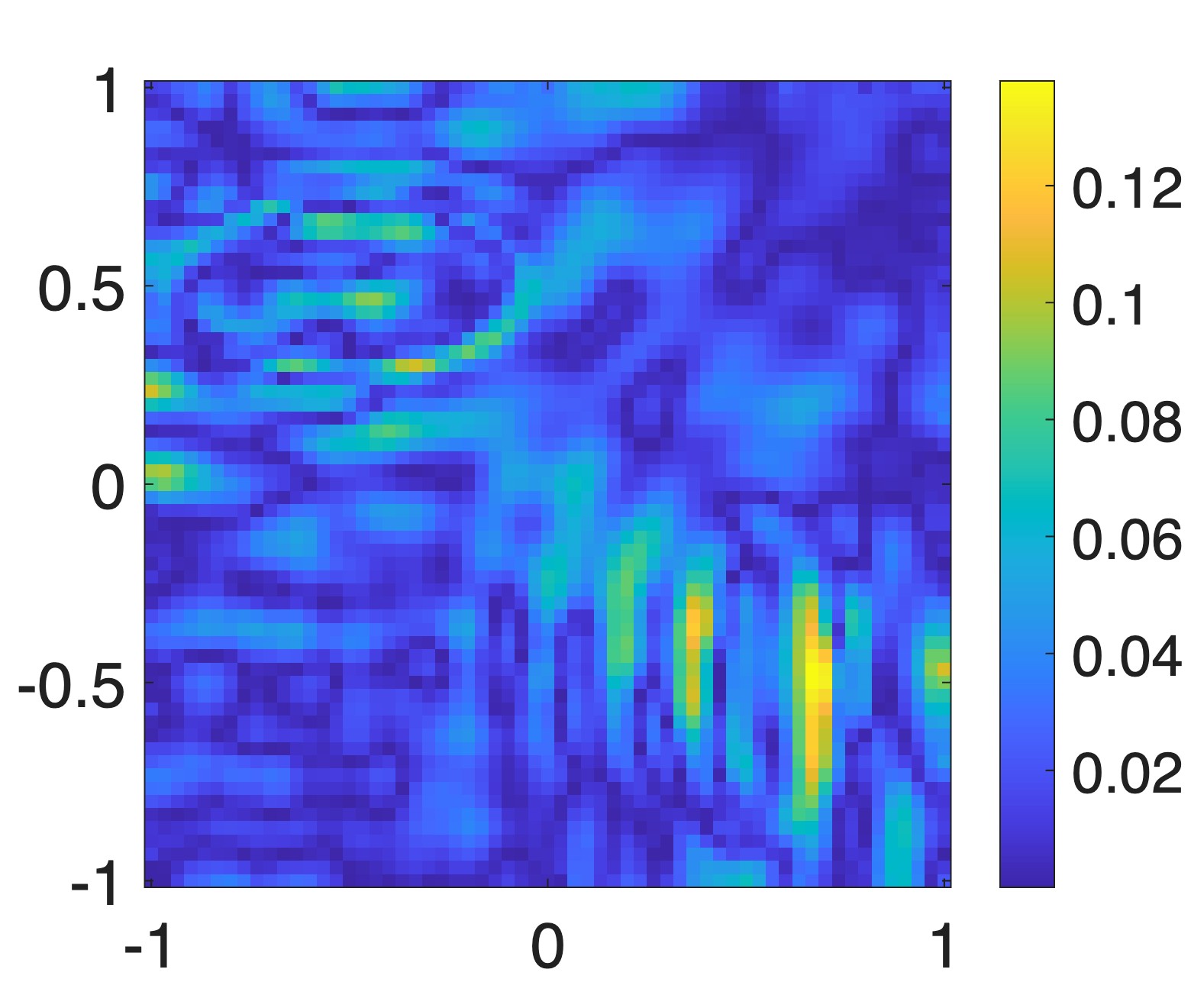}}
\hfill
\subfloat[Log Picard difference.\label{fig:case1_logdiff}]{
    \includegraphics[width=0.235\textwidth]{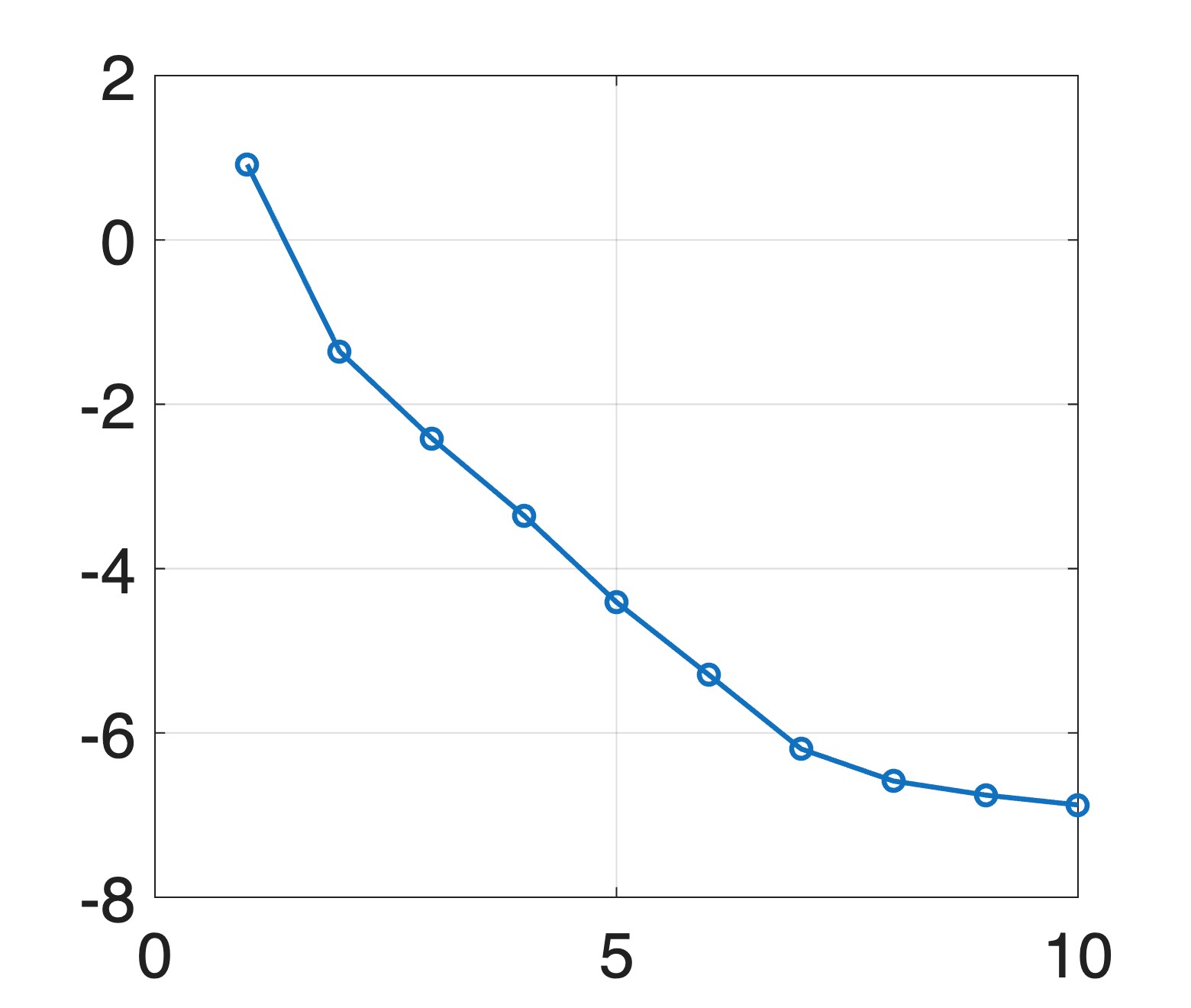}}

\subfloat[True $\Im(u^0)$.\label{fig:case1_true_imag}]{
    \includegraphics[width=0.235\textwidth]{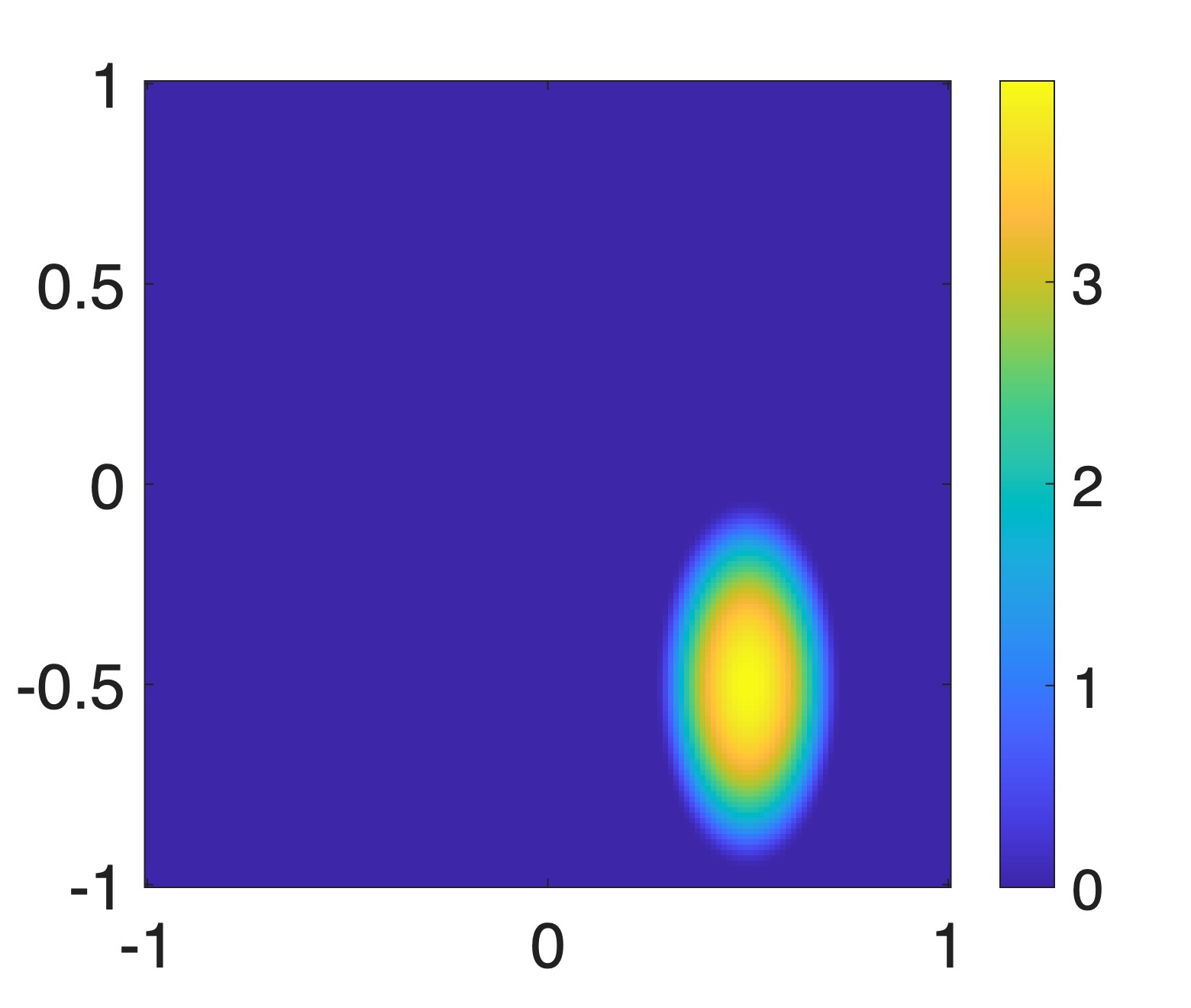}}
\hfill
\subfloat[Computed $\Im(u^0)$.\label{fig:case1_comp_imag}]{
    \includegraphics[width=0.235\textwidth]{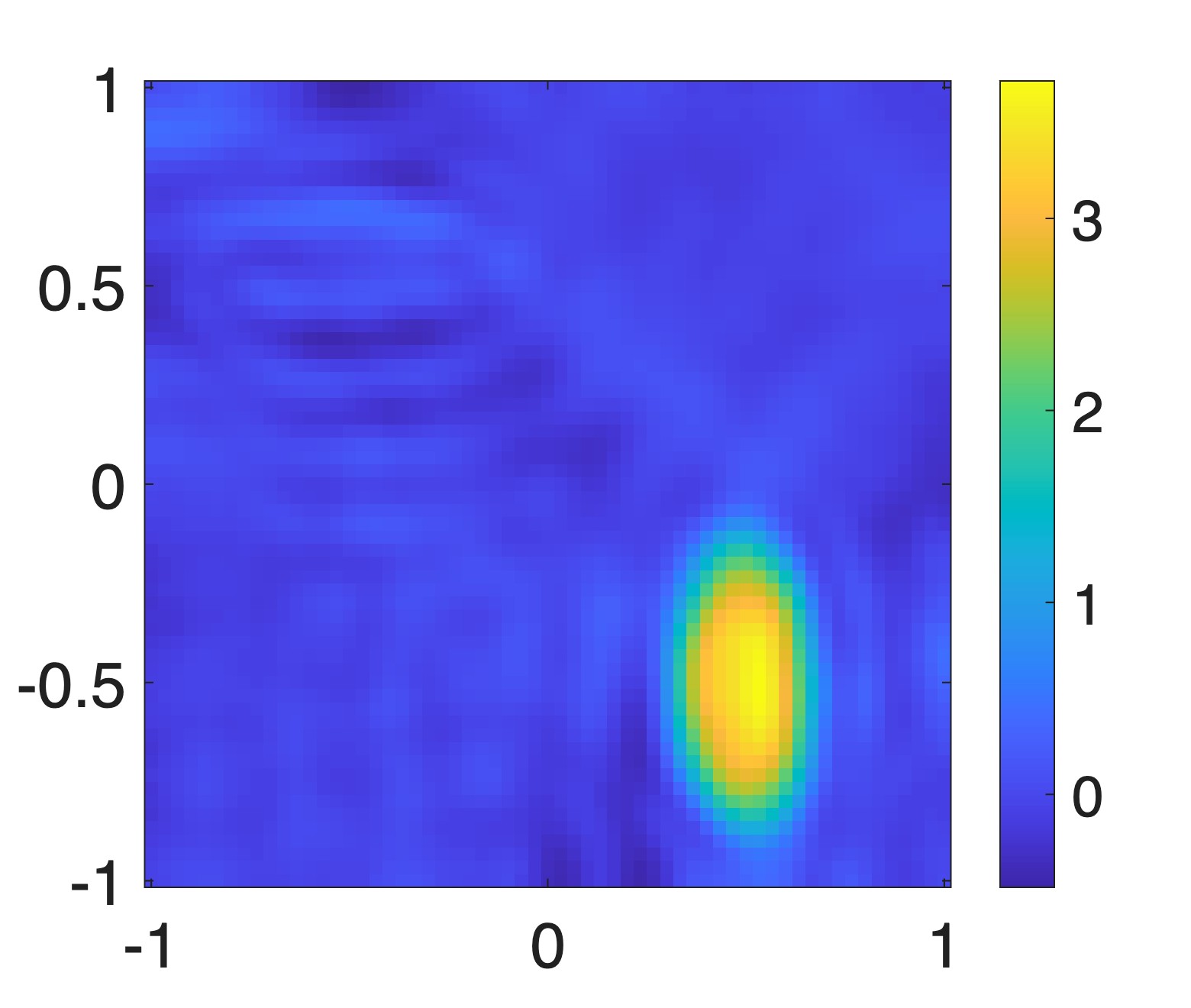}}
\hfill
\subfloat[$E_{\Im}$.\label{fig:case1_err_imag}]{
    \includegraphics[width=0.235\textwidth]{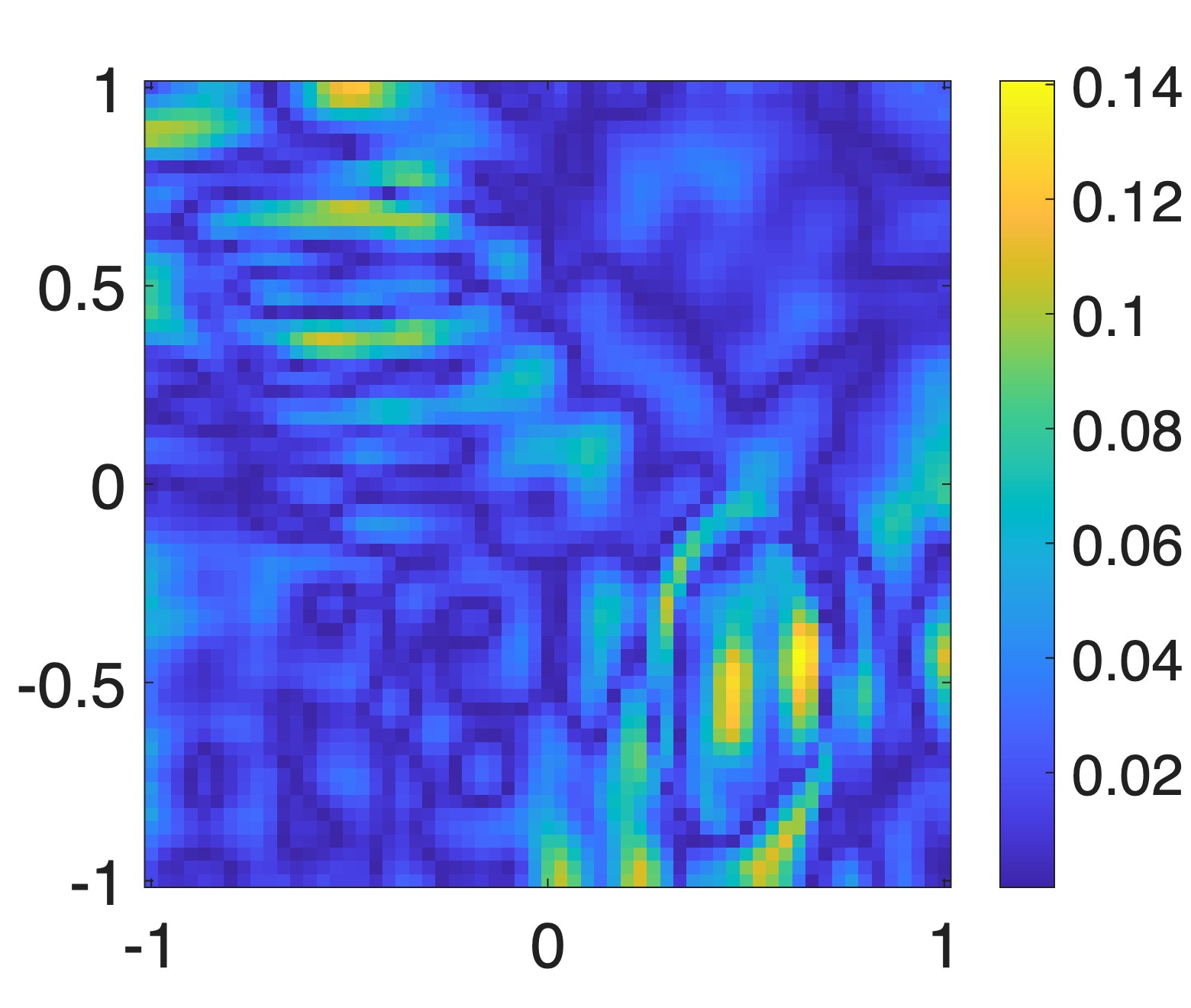}}
\hfill
\subfloat[Dimensionless residual.\label{fig:case1_residual}]{
    \includegraphics[width=0.235\textwidth]{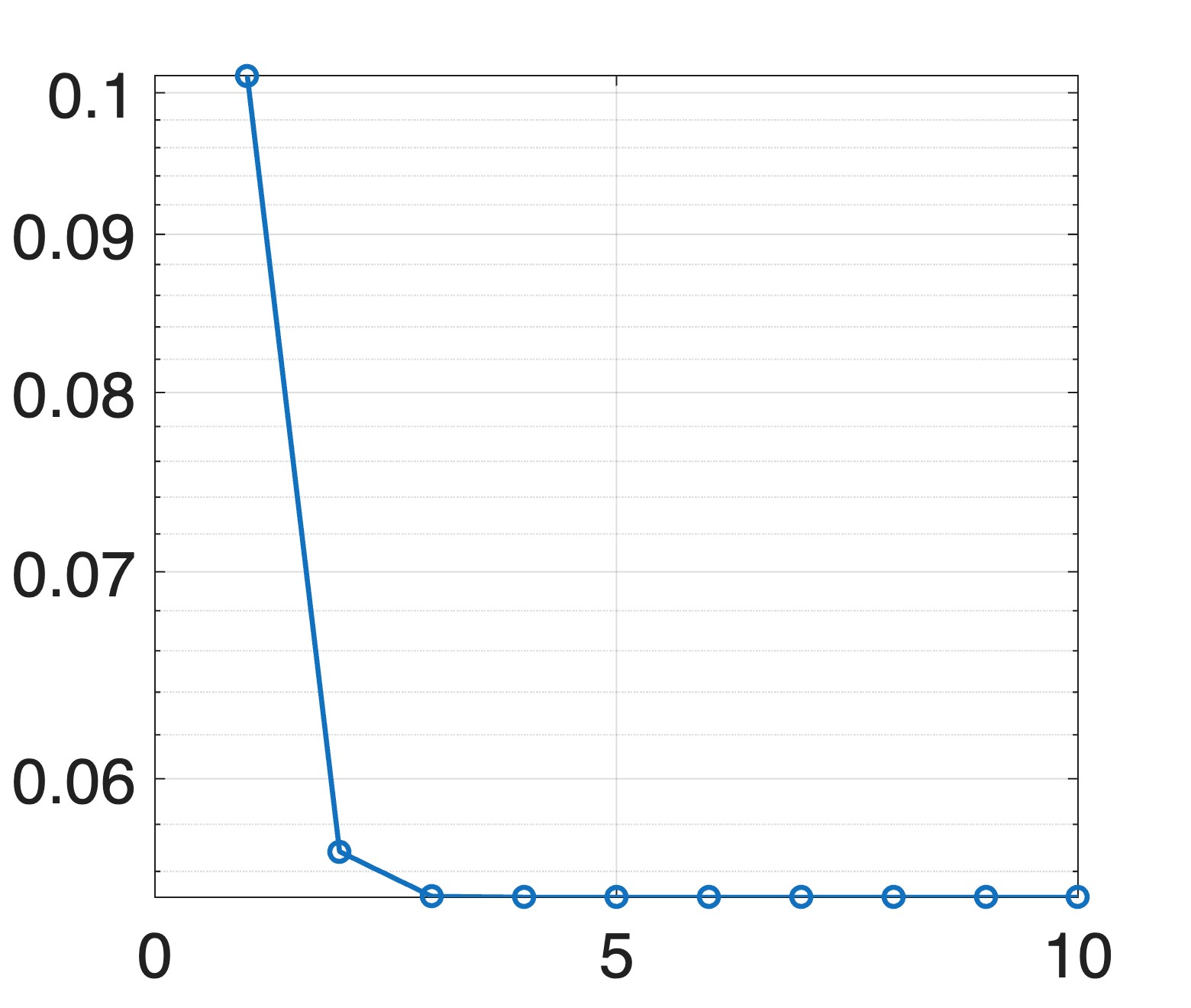}}
\caption{Test~1 with $10\%$ noise. Top row: true and computed real parts, the normalized pointwise error $E_{\Re}$, and the quantity \eqref{eq:picard_logdiff}. Bottom row: true and computed imaginary parts, the normalized pointwise error $E_{\Im}$, and the dimensionless residual \eqref{6.3}.}
\label{fig:case1_all}
\end{figure}

The results of Test~1 with $10\%$ noise are displayed in
Figure~\ref{fig:case1_all}. The normalized pointwise errors in
Figures~\ref{fig:case1_err_real} and \ref{fig:case1_err_imag} are defined by
\begin{equation}\label{eq:pointwise_error}
E_{\Re}(\x)=\frac{|\Re(u^{0,\rm comp}(\x)-u^{0,\rm true}(\x))|}{\max_{\overline{\Omega}}|\Re(u^{0,\rm true})|},
\qquad
E_{\Im}(\x)=\frac{|\Im(u^{0,\rm comp}(\x)-u^{0,\rm true}(\x))|}{\max_{\overline{\Omega}}|\Im(u^{0,\rm true})|}.
\end{equation}
To monitor the convergence of the Picard iteration, we plot in
Figure~\ref{fig:case1_logdiff} the quantity
\begin{equation}\label{eq:picard_logdiff}
\log_{10}\big\|\mathbf U^{(k+1)}-\mathbf U^{(k)}\big\|_{\lambda,\epsilon},
\qquad k=0,1,\dots,K_{\max}-1,
\end{equation}
where $\|\cdot\|_{\lambda,\epsilon}$ is the norm in \eqref{norm}, and
$\mathbf U^{(k)}=(u_0^{(k)},\dots,u_N^{(k)})^\top$. In
Figure~\ref{fig:case1_residual}, we plot the dimensionless residual
\begin{equation}\label{6.3}
\mathrm{Res}^{(k)}
:=
\frac{
\left\|
\ii S\mathbf U^{(k)}+\Delta \mathbf U^{(k)}+\mathbf B(\mathbf U^{(k)})
\right\|_{[L^2(\Omega)]^{N+1}}
}{
\max\left\{
\left\|
\Delta \mathbf U^{(k)}
\right\|_{[L^2(\Omega)]^{N+1}},
\,10^{-2}
\right\}
},
\end{equation}
where
\[
\mathbf{B}\bigl(\mathbf U^{(k)}\bigr)(\x)
:=
\left(
\int_0^T
e^{-2t} q(\x,t)
\left|
\sum_{\ell=0}^{N}u_\ell^{(k)}(\x)\Psi_\ell(t)
\right|^{p-1}
\sum_{\ell=0}^{N}u_\ell^{(k)}(\x)\Psi_\ell(t)\,
\Psi_m(t)\,dt
\right)_{m=0}^N.
\]
Both elliptic inclusions are recovered at the correct locations, with the
correct orientations; the amplitude of the inclusion in the imaginary part is slightly underestimated. A weak imprint of each
inclusion appears in the other component, and the largest pointwise errors are
located there and near the edges of the inclusions. The quantity
\eqref{eq:picard_logdiff} decreases by about eight orders of magnitude over ten
iterations, and the residual drops in the first three iterations and then
stays at a small constant level. These plots show convergence of the
implemented Picard iteration for the parameters used in this experiment.

To quantify the reconstruction, we use the relative $L^2$ errors
\[
e_{\Re}=\frac{\|\Re(u^{0,\rm comp}-u^{0,\rm true})\|_{L^2(\Omega)}}{\|\Re(u^{0,\rm true})\|_{L^2(\Omega)}},
\qquad
e_{\Im}=\frac{\|\Im(u^{0,\rm comp}-u^{0,\rm true})\|_{L^2(\Omega)}}{\|\Im(u^{0,\rm true})\|_{L^2(\Omega)}}.
\]
For the noise levels $5\%$, $10\%$, and $20\%$, we obtain $e_{\Re}=20.55\%$,
$20.57\%$, and $20.63\%$, and $e_{\Im}=21.46\%$, $21.49\%$, and $21.56\%$,
respectively. Thus, the errors change only slightly when the noise level
increases from $5\%$ to $20\%$.

\begin{Remark}[Smooth initial data in Test~1]
The initial field in Test~1 belongs to $C^\infty(\overline\Omega)$
and vanishes to infinite order on $\partial\Omega$. In particular,
it satisfies the homogeneous Dirichlet compatibility conditions at
$t=0$ and belongs to $H^s(\Omega)$ for every $s\geq0$. Moreover,
\[
\|u^0\|_{L^2(\Omega)}^2
=4\pi\int_0^1\exp\!\left(2-\frac{2}{1-z}\right)\,dz
\approx3.485<4.
\]
For the two-dimensional focusing cubic Dirichlet problem,
\cite{brezisgallouet1980} establishes global $H^2$ well-posedness
under this small-mass condition on domains with smooth boundary.
Thus, Test~1 provides an example with smooth, boundary-compatible
initial data.
\end{Remark}

{\bf Test 2.} Let $d(x,y)=\max\{|x-0.05|,|y-0.05|\}$ and
\[
\Phi(x,y)=
\begin{cases}
\exp\!\left(1-\dfrac{1}{1-s(x,y)^2}\right), & |s(x,y)|<1,\\[2mm]
0, & \text{otherwise},
\end{cases}
\qquad
s(x,y)=\frac{d(x,y)-0.50}{0.18}.
\]
The function $\Phi$ is a square-ring profile centered at $(0.05,0.05)$ and
supported in $\{0.32<d<0.68\}$. We choose
\[
u^0(x,y)=-2\,\Phi(x,y)+2\,\ii\,\Phi(x,y).
\]
Thus, the real and imaginary parts have the same support and opposite signs.
In this test, we set $p=3$.

\begin{figure}[ht]
\centering
\subfloat[True $\Re(u^0)$.\label{fig:case2_true_real}]{
    \includegraphics[width=0.235\textwidth]{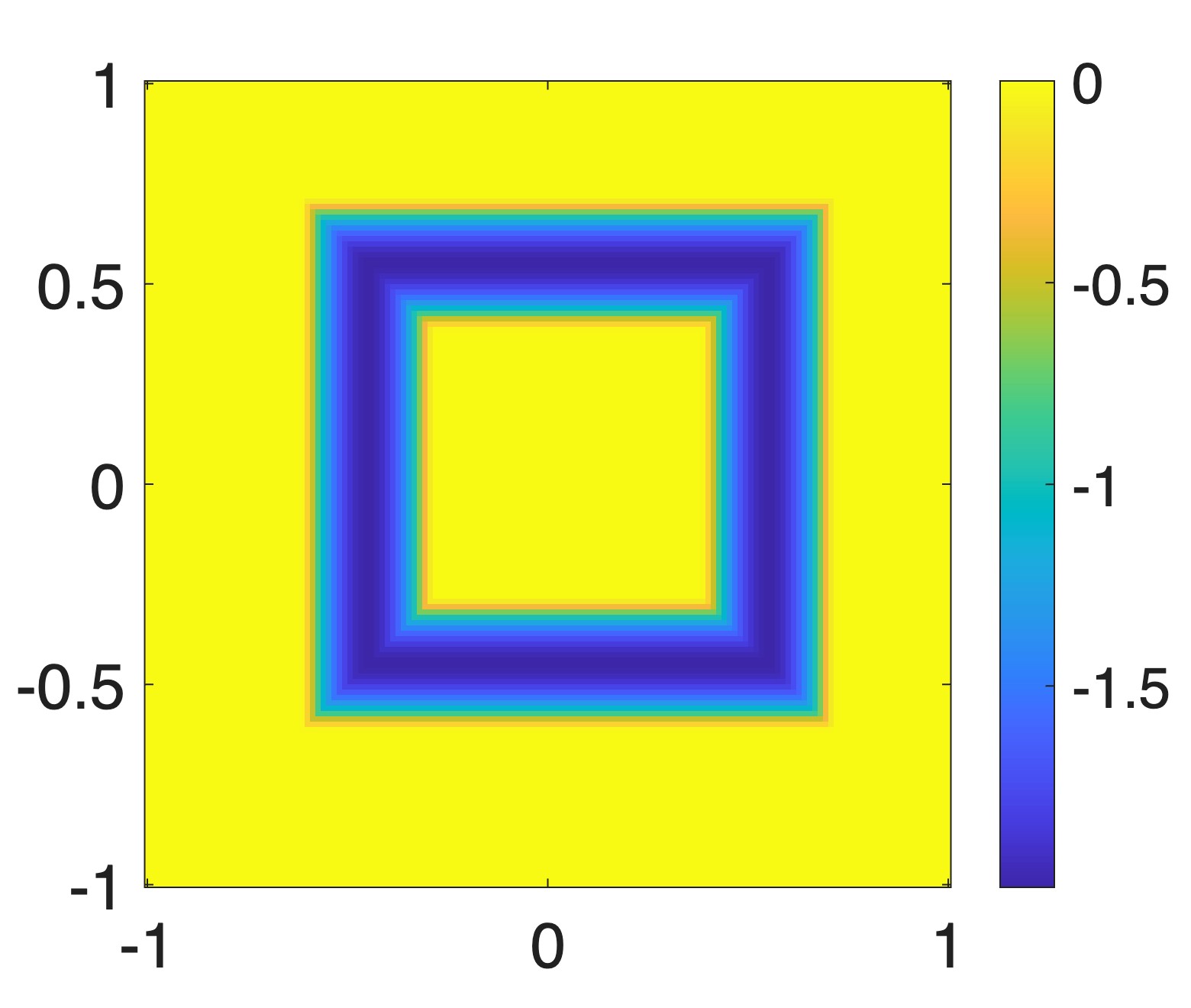}}
\hfill
\subfloat[Computed $\Re(u^0)$.\label{fig:case2_comp_real}]{
    \includegraphics[width=0.235\textwidth]{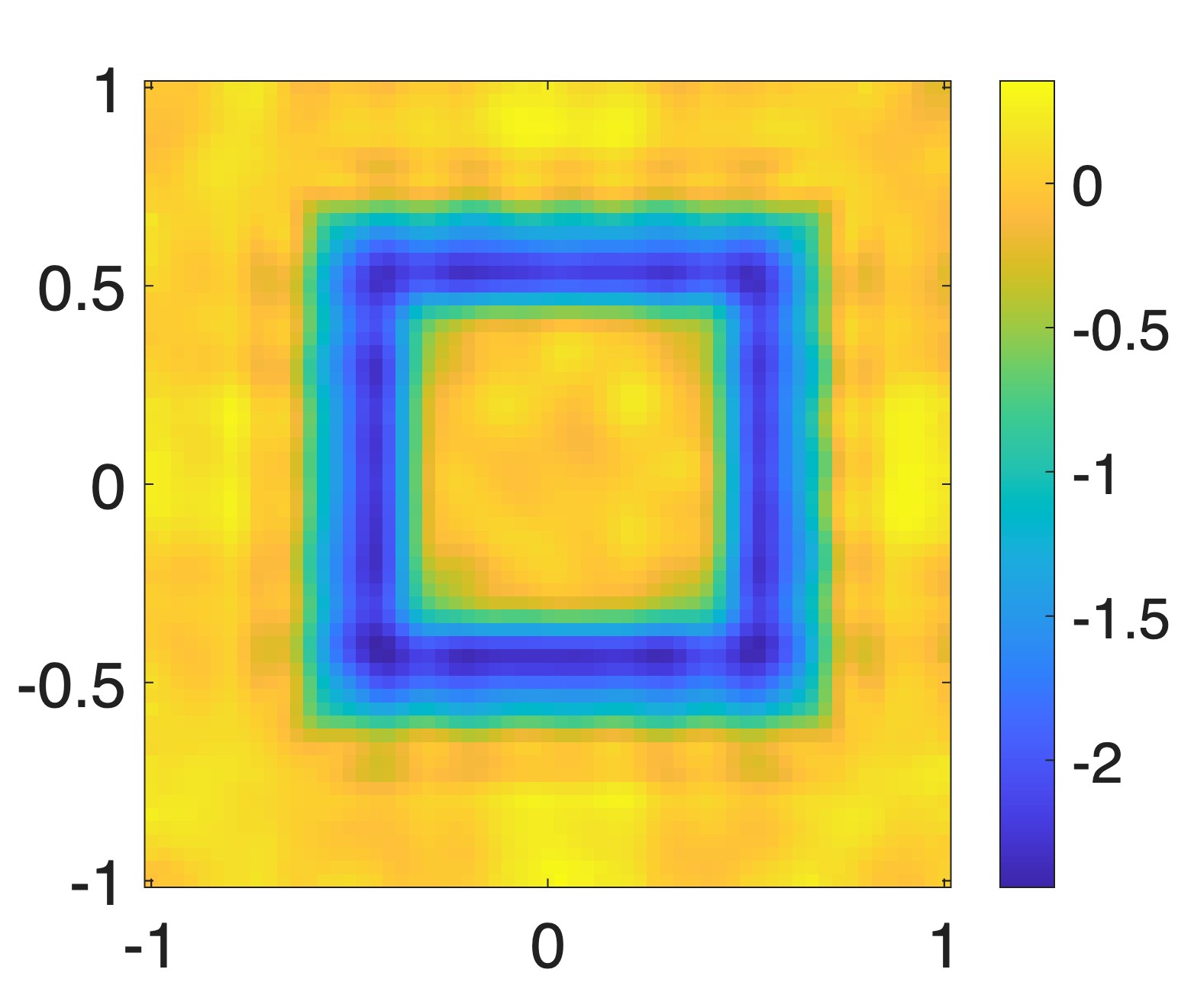}}
\hfill
\subfloat[$E_{\Re}$.\label{fig:case2_err_real}]{
    \includegraphics[width=0.235\textwidth]{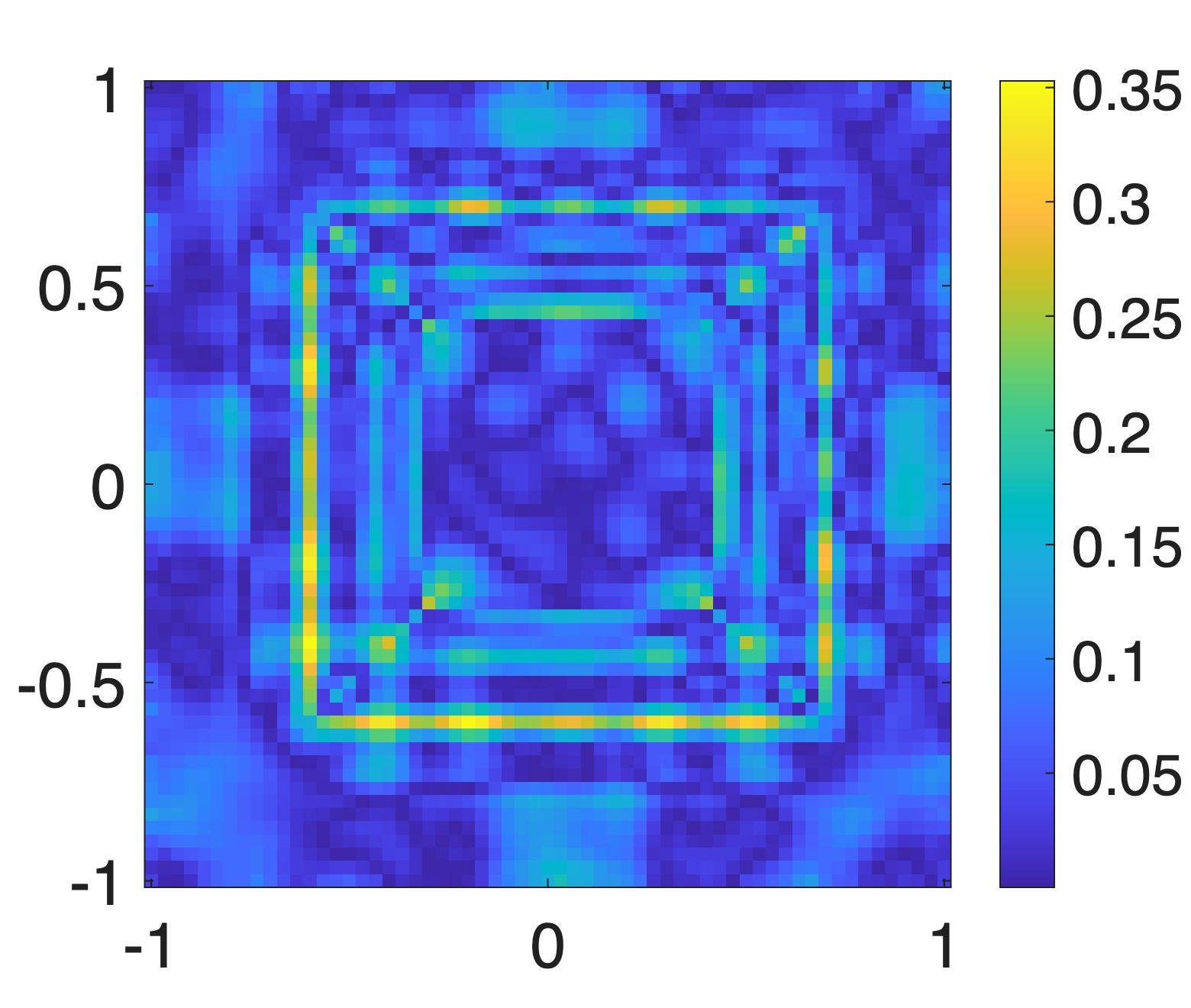}}
\hfill
\subfloat[Log Picard difference.\label{fig:case2_logdiff}]{
    \includegraphics[width=0.235\textwidth]{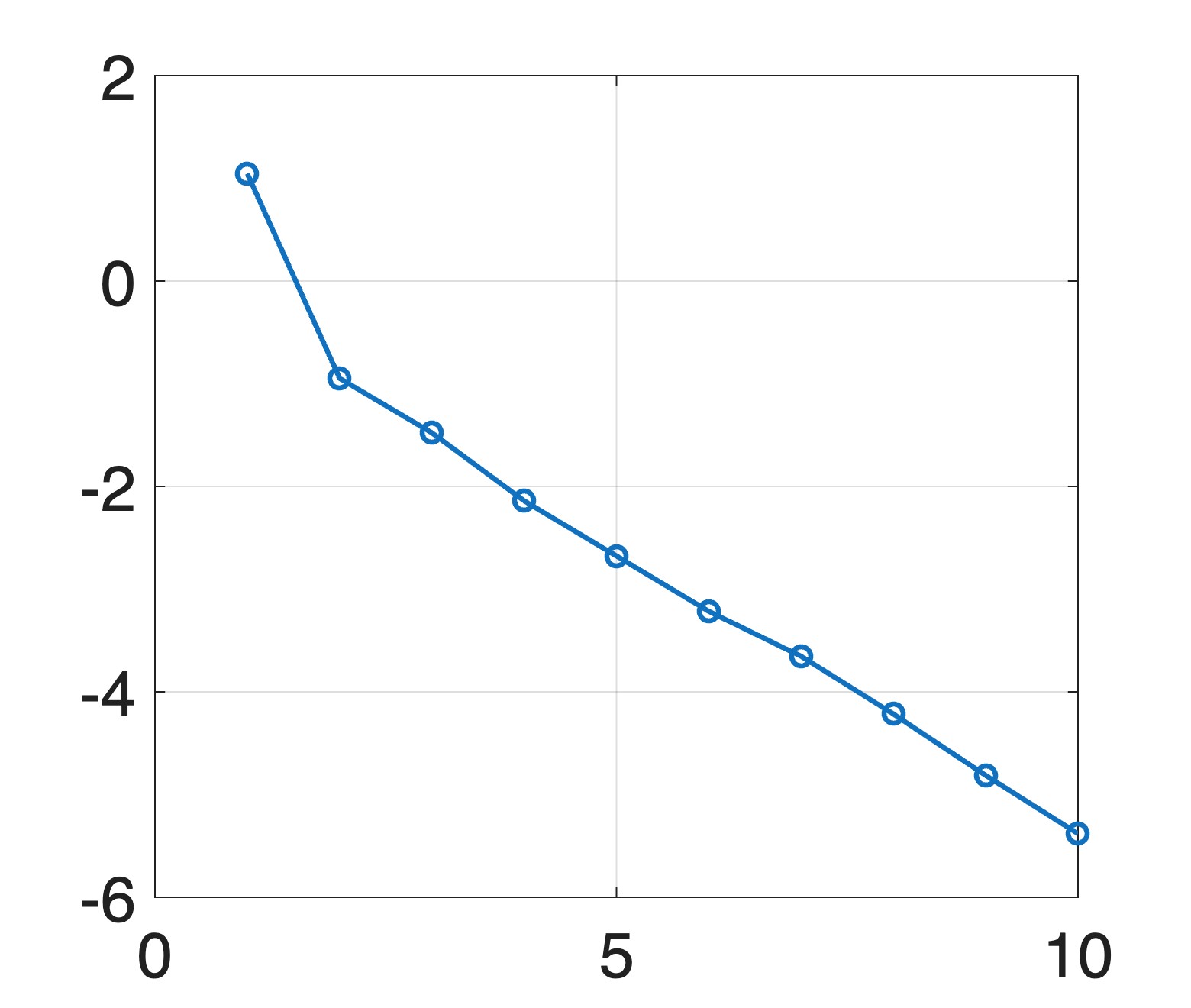}}

\subfloat[True $\Im(u^0)$.\label{fig:case2_true_imag}]{
    \includegraphics[width=0.235\textwidth]{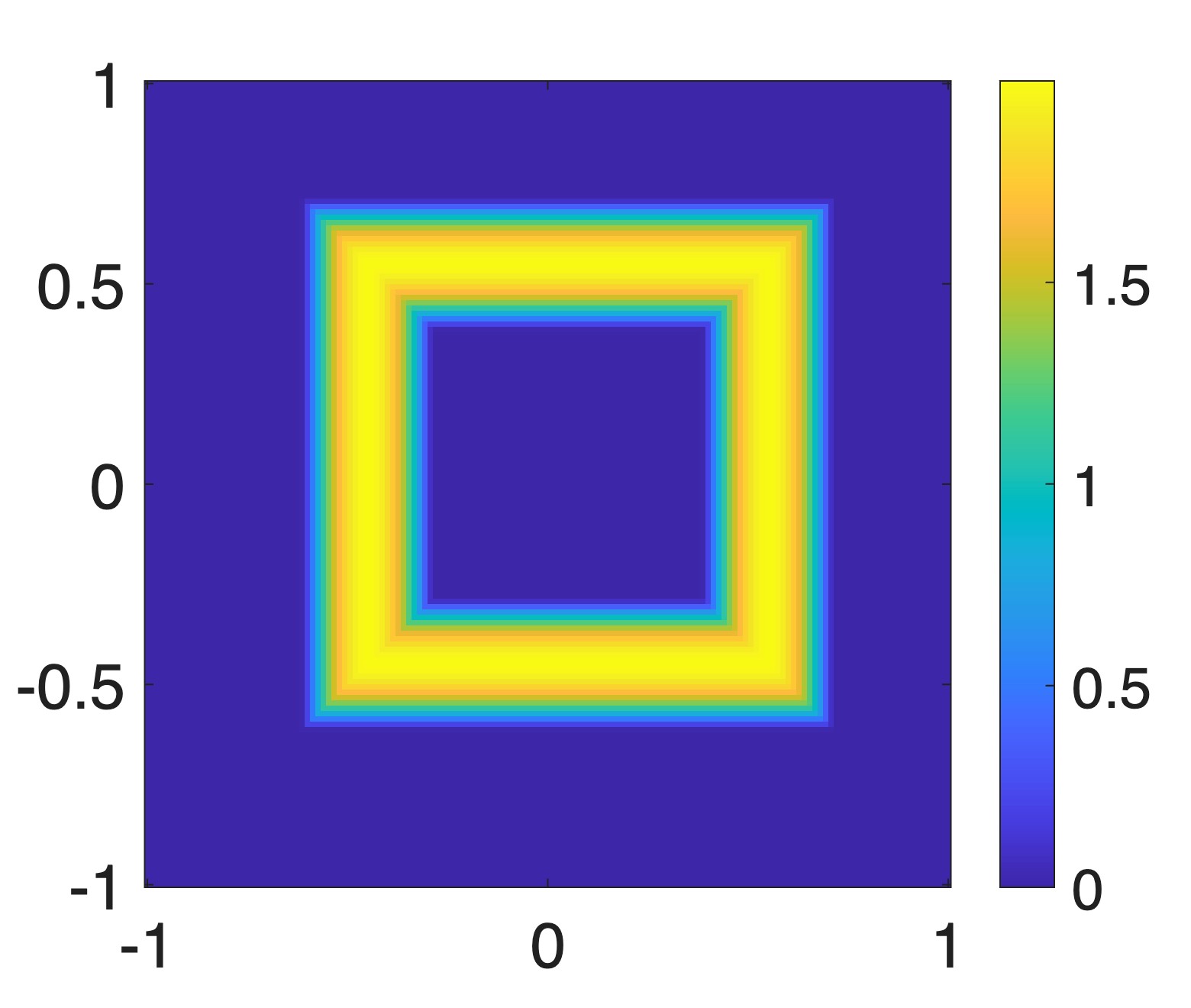}}
\hfill
\subfloat[Computed $\Im(u^0)$.\label{fig:case2_comp_imag}]{
    \includegraphics[width=0.235\textwidth]{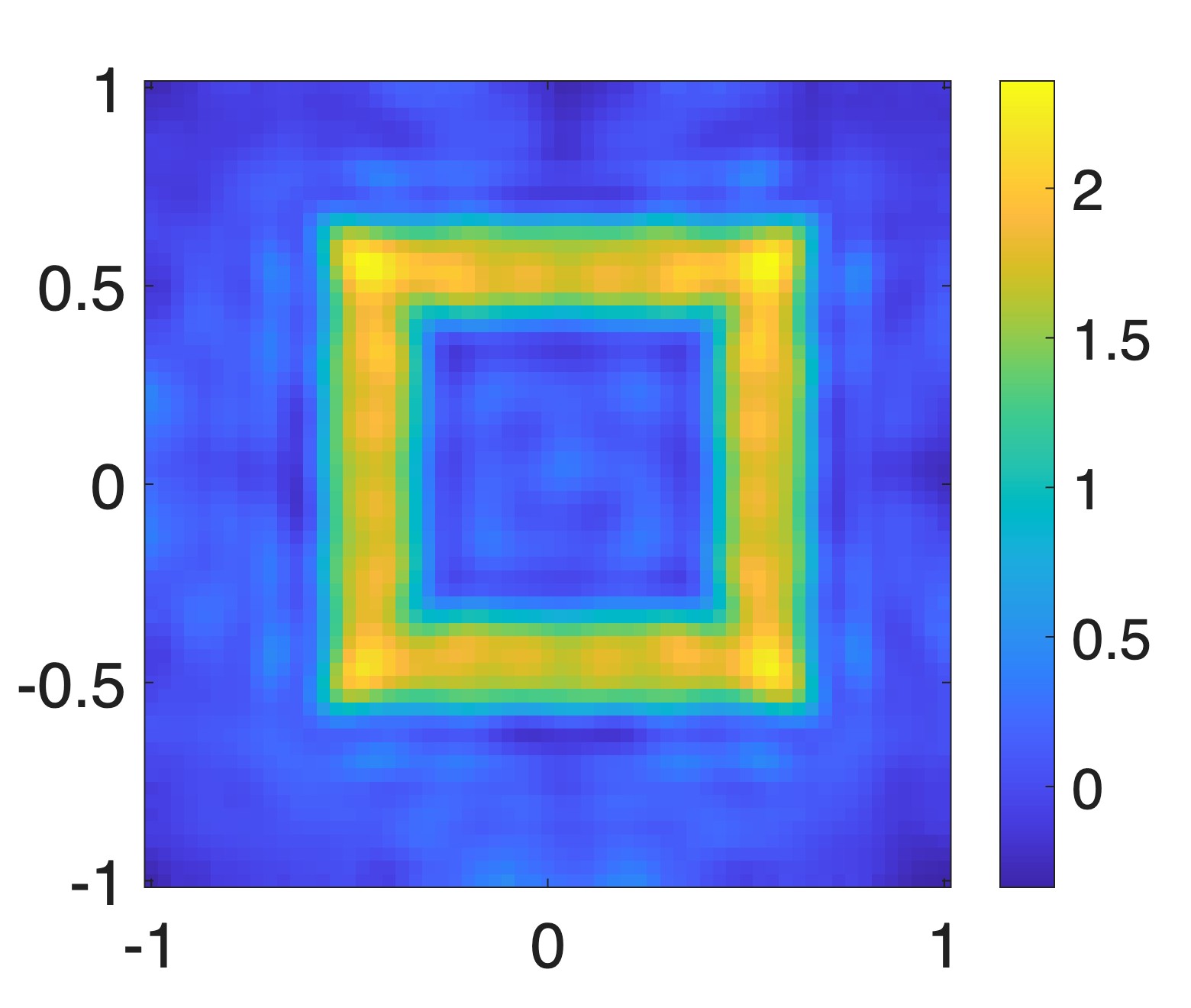}}
\hfill
\subfloat[$E_{\Im}$.\label{fig:case2_err_imag}]{
    \includegraphics[width=0.235\textwidth]{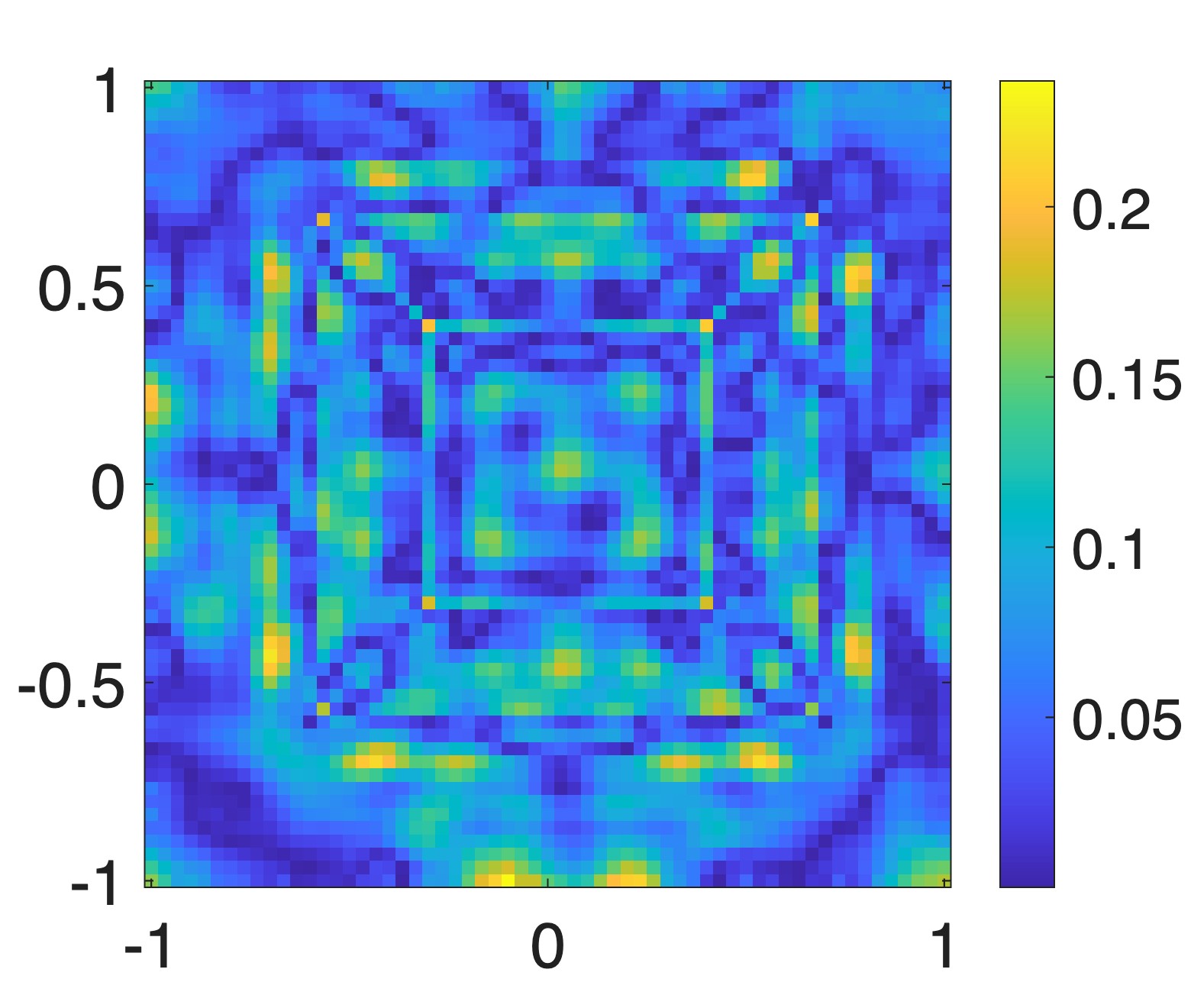}}
\hfill
\subfloat[Dimensionless residual.\label{fig:case2_residual}]{
    \includegraphics[width=0.235\textwidth]{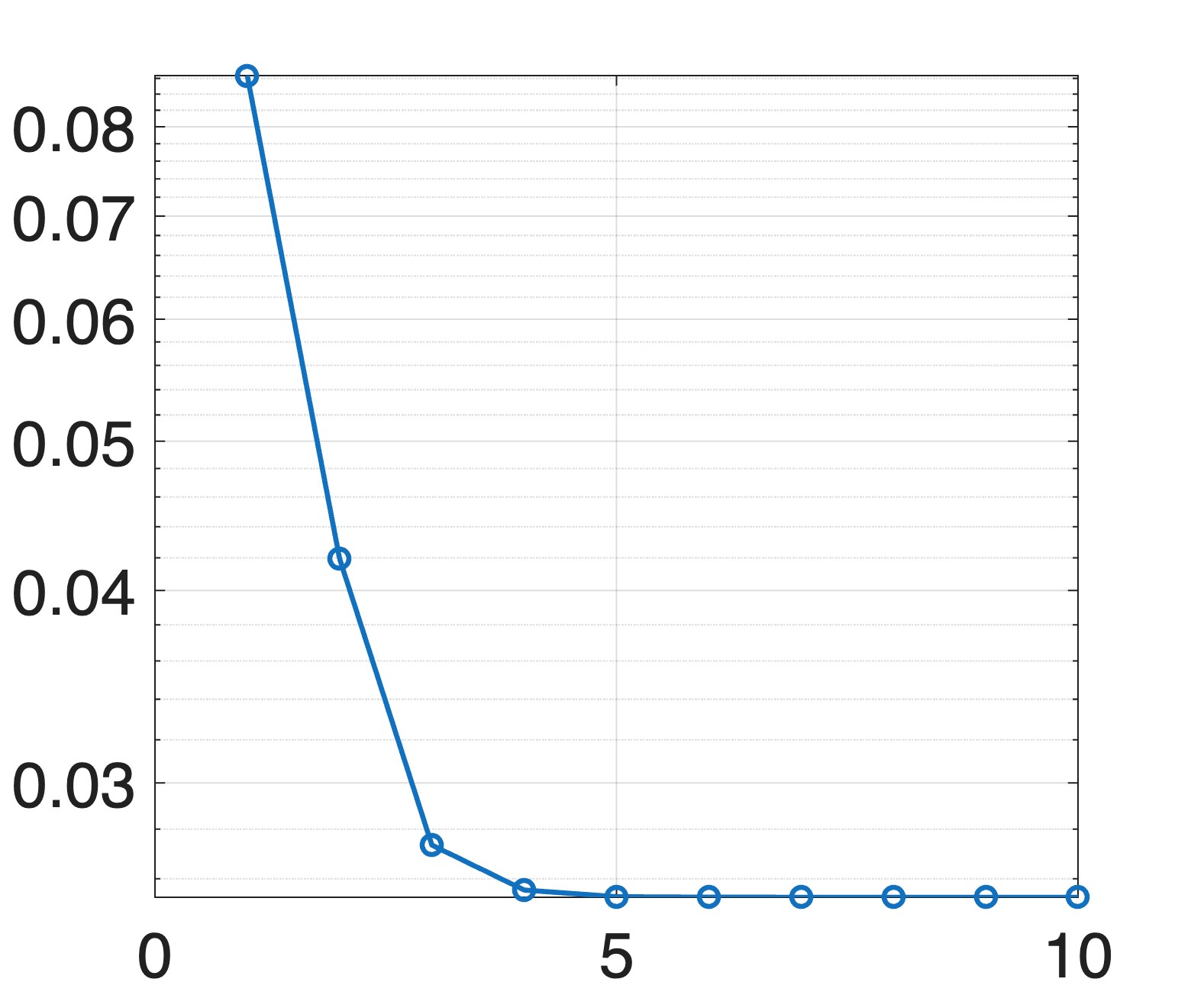}}
\caption{Test~2 with $10\%$ noise. The panels are arranged as in Figure~\ref{fig:case1_all}.}
\label{fig:case2_all}
\end{figure}

The results of Test~2 with $10\%$ noise are displayed in
Figure~\ref{fig:case2_all}. The square ring is recovered in both components,
with the correct location, width, and sign. The inner and outer edges are
slightly smoothed, and a mild background oscillation is visible in the inner
square. The quantity \eqref{eq:picard_logdiff} decreases steadily, and the
residual \eqref{6.3} stabilizes after about five iterations.
For the noise levels $5\%$, $10\%$, and $20\%$, the relative $L^2$ errors
are $e_{\Re}=21.37\%$, $21.39\%$, and $21.48\%$, and $e_{\Im}=18.99\%$,
$19.02\%$, and $19.12\%$, respectively.

{\bf Test 3.} In this test, the real and imaginary parts have different
geometric structures. The real part is an elliptic ring centered at $(0,0.5)$:
\[
u_R^0(x,y)=
\begin{cases}
1, & \dfrac{x^2}{0.85^2}+\dfrac{(y-0.5)^2}{0.32^2}\le 1
\ \text{ and }\
\dfrac{x^2}{0.42^2}+\dfrac{(y-0.5)^2}{0.16^2}>1,\\[3mm]
0, & \text{otherwise}.
\end{cases}
\]
The imaginary part is an N-shaped inclusion. Let
\[
\tilde x=\frac{x-0.35}{1.3}-0.02,
\qquad
\tilde y=\frac{y+0.35}{1.3}.
\]
We define
\[
u_I^0(x,y)=
\begin{cases}
1, & -0.42\le\tilde x\le -0.24,\ -0.42\le\tilde y\le 0.42,\\
1, & 0.20\le\tilde x\le 0.38,\ -0.42\le\tilde y\le 0.42,\\
1, & |\tilde y+1.55\,\tilde x+0.06|\le 0.10,\ -0.30\le\tilde x\le 0.26,\ -0.42\le\tilde y\le 0.42,\\
0, & \text{otherwise}.
\end{cases}
\]
Thus, $u_I^0$ consists of two vertical bars connected by a diagonal strip,
located in the lower-right part of $\Omega$. Both components are
discontinuous. In this test, we set $p=5$.

\begin{Remark}\label{rem:test3_forward}
For Test~3, we monitored the computed forward solution on the
$151\times151$ grid over $[0,T]$, where $T=0.2$. Its maximum amplitude
over this interval was $2.4299$, attained at approximately $t=0.083$;
the amplitudes at $t=0$ and $t=T$ were $1.4142$ and $1.7784$,
respectively. We also monitored the discrete mass $M_h^k$ and energy
$E_h^k$ of $u_f^k$, $k=0,1,\dots,2N_t$, defined by
\begin{align*}
M_h^k &= h_f^2\sum_{i=1}^{N_x^f}\sum_{j=1}^{N_x^f}\big|(u_f^k)_{i,j}\big|^2,\\
E_h^k &= h_f^2\sum_{i=1}^{N_x^f-1}\sum_{j=1}^{N_x^f}\Big|\frac{(u_f^k)_{i+1,j}-(u_f^k)_{i,j}}{h_f}\Big|^2
+ h_f^2\sum_{i=1}^{N_x^f}\sum_{j=1}^{N_x^f-1}\Big|\frac{(u_f^k)_{i,j+1}-(u_f^k)_{i,j}}{h_f}\Big|^2
- \frac{h_f^2}{3}\sum_{i=1}^{N_x^f}\sum_{j=1}^{N_x^f}\big|(u_f^k)_{i,j}\big|^6,
\end{align*}
where the boundary values $(u_f^k)_{i,j}=0$ are included in the sums.
The last term of $E_h^k$ is the focusing quintic term $2|u|^{p+1}/(p+1)$ with $p=5$.
The maximum relative drifts
\[
\max_{0\le k\le 2N_t}\frac{|M_h^k-M_h^0|}{M_h^0}
\qquad\text{and}\qquad
\max_{0\le k\le 2N_t}\frac{|E_h^k-E_h^0|}{|E_h^0|}
\]
were $0.74\%$ and $10.42\%$, respectively. Since the initial field of
Test~3 is discontinuous, its discrete energy is dominated by the jumps
and grows as the grid is refined; the energy drift is therefore less
informative than the mass drift for this test. Thus, this forward
computation shows no numerical indication of blow-up or near-blow-up
on $[0,0.2]$.
\end{Remark}

\begin{figure}[ht]
\centering
\subfloat[True $\Re(u^0)$.\label{fig:case3_true_real}]{
    \includegraphics[width=0.235\textwidth]{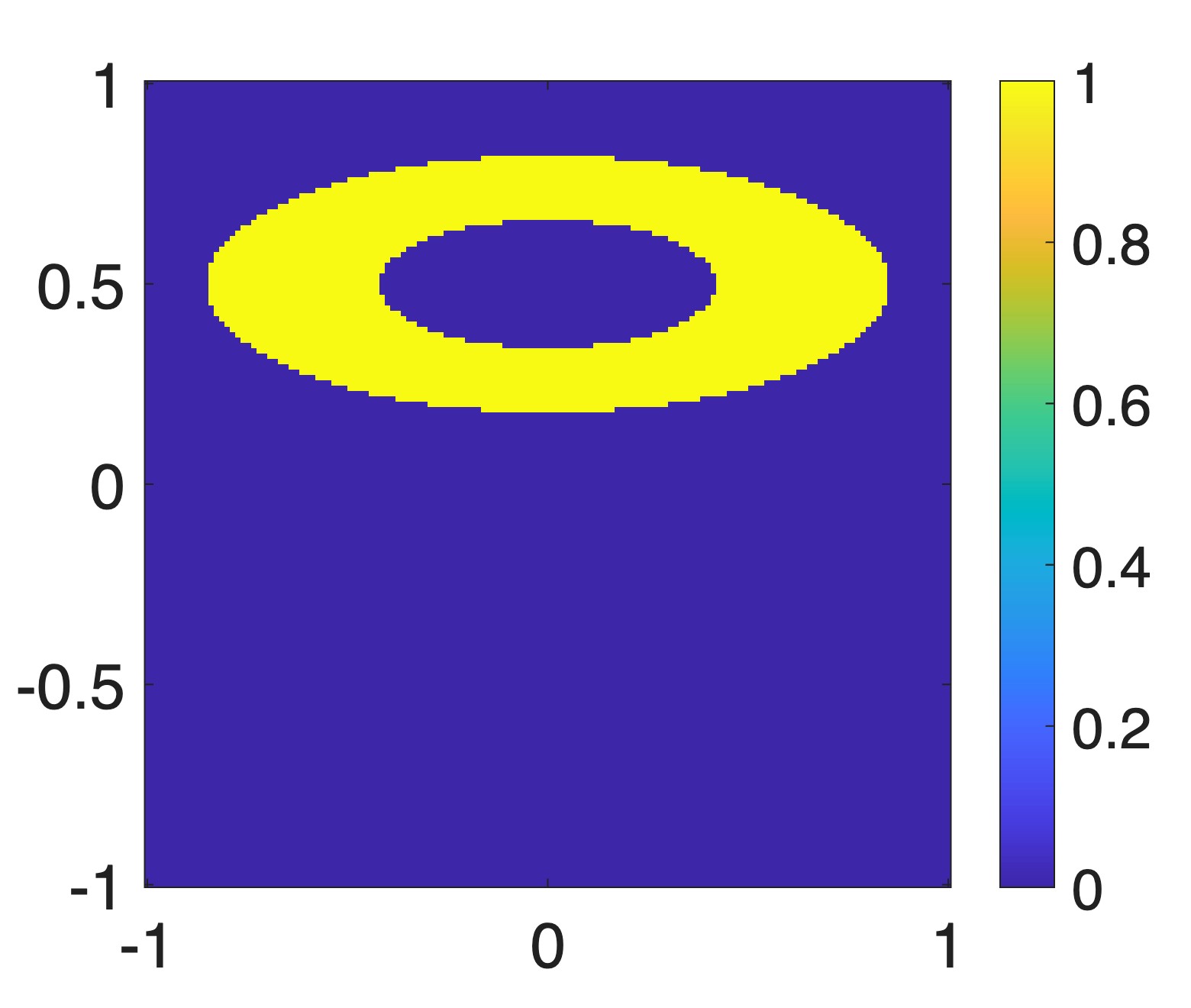}}
\hfill
\subfloat[Computed $\Re(u^0)$.\label{fig:case3_comp_real}]{
    \includegraphics[width=0.235\textwidth]{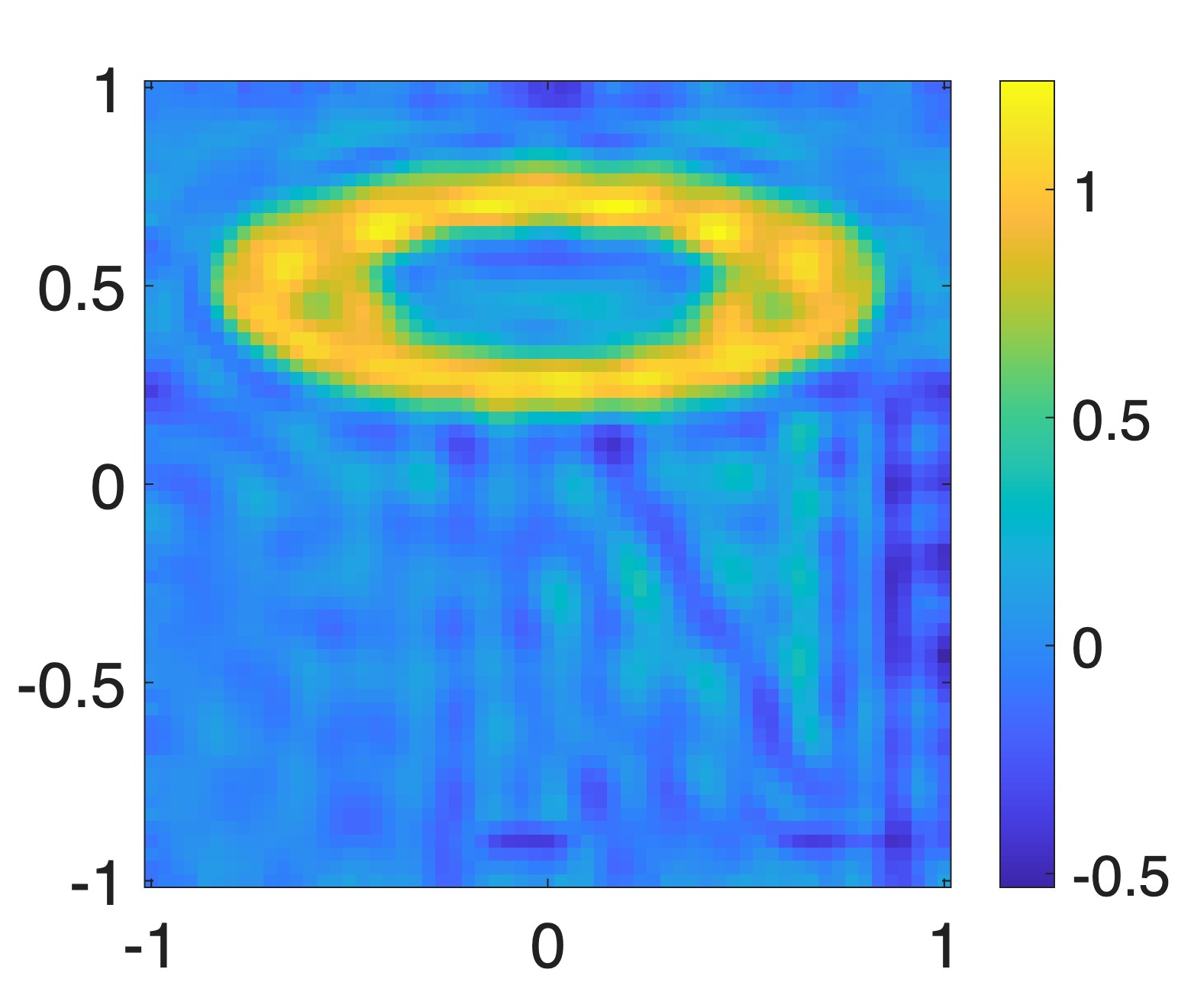}}
\hfill
\subfloat[$E_{\Re}$.\label{fig:case3_err_real}]{
    \includegraphics[width=0.235\textwidth]{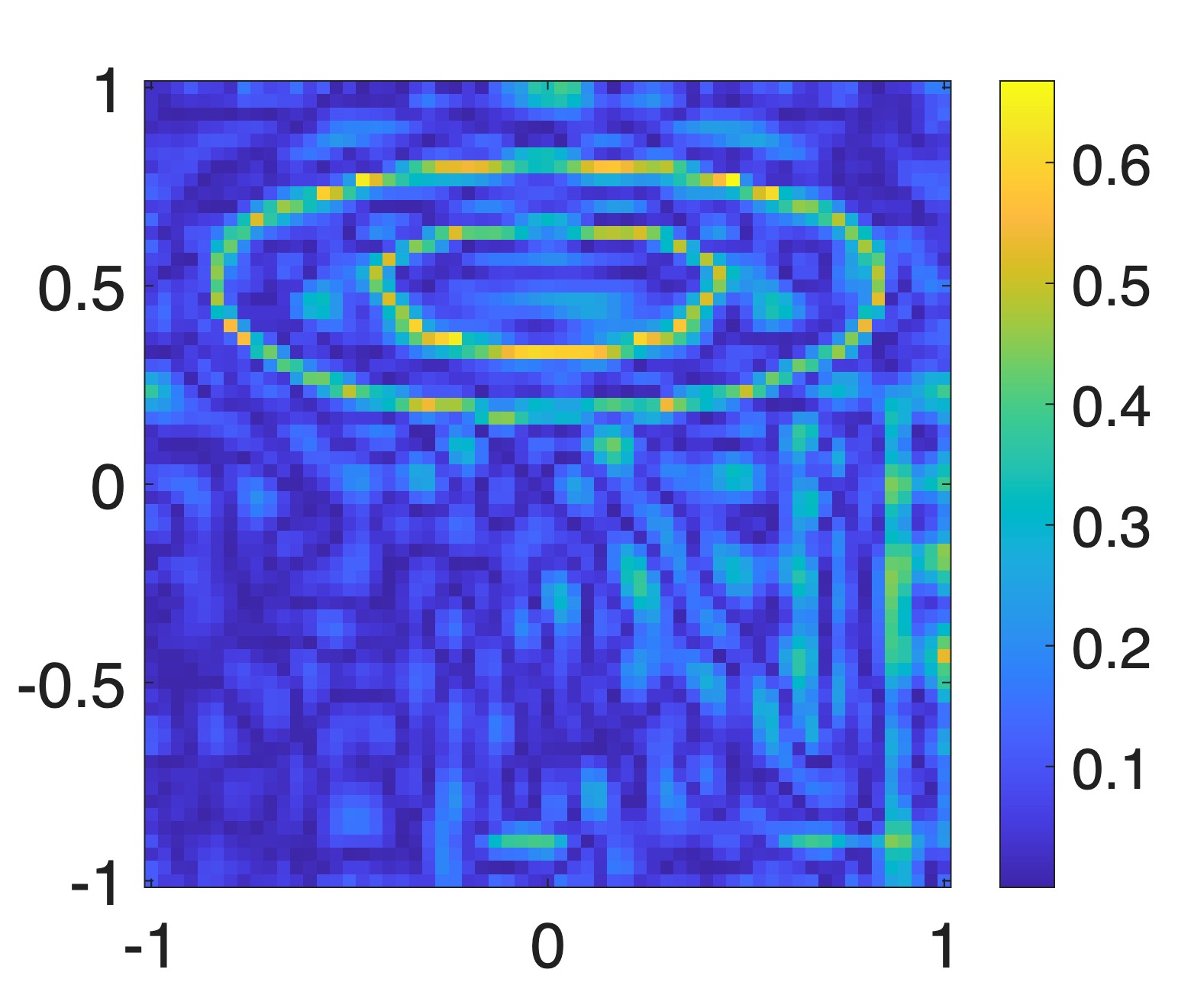}}
\hfill
\subfloat[Log Picard difference.\label{fig:case3_logdiff}]{
    \includegraphics[width=0.235\textwidth]{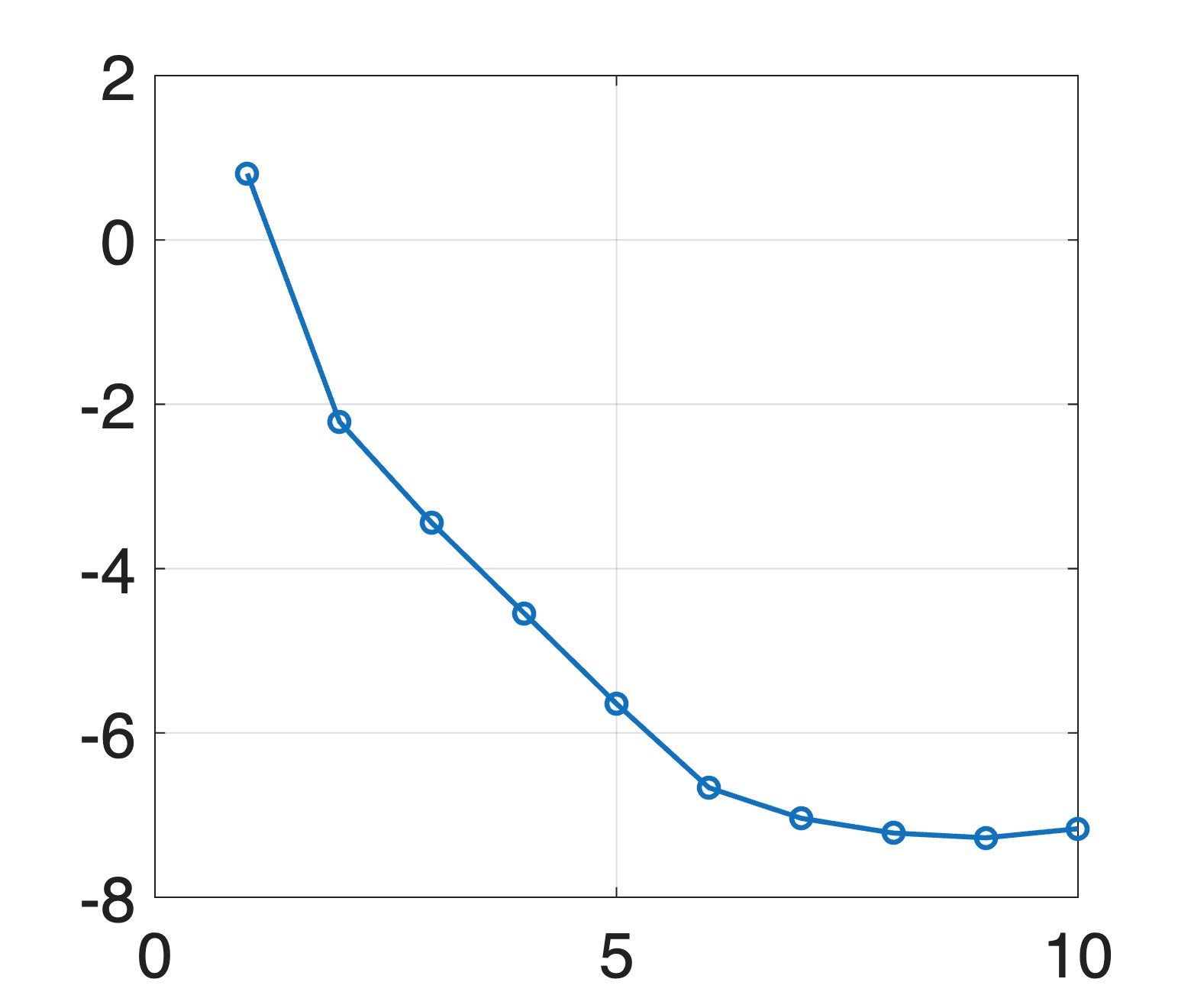}}

\subfloat[True $\Im(u^0)$.\label{fig:case3_true_imag}]{
    \includegraphics[width=0.235\textwidth]{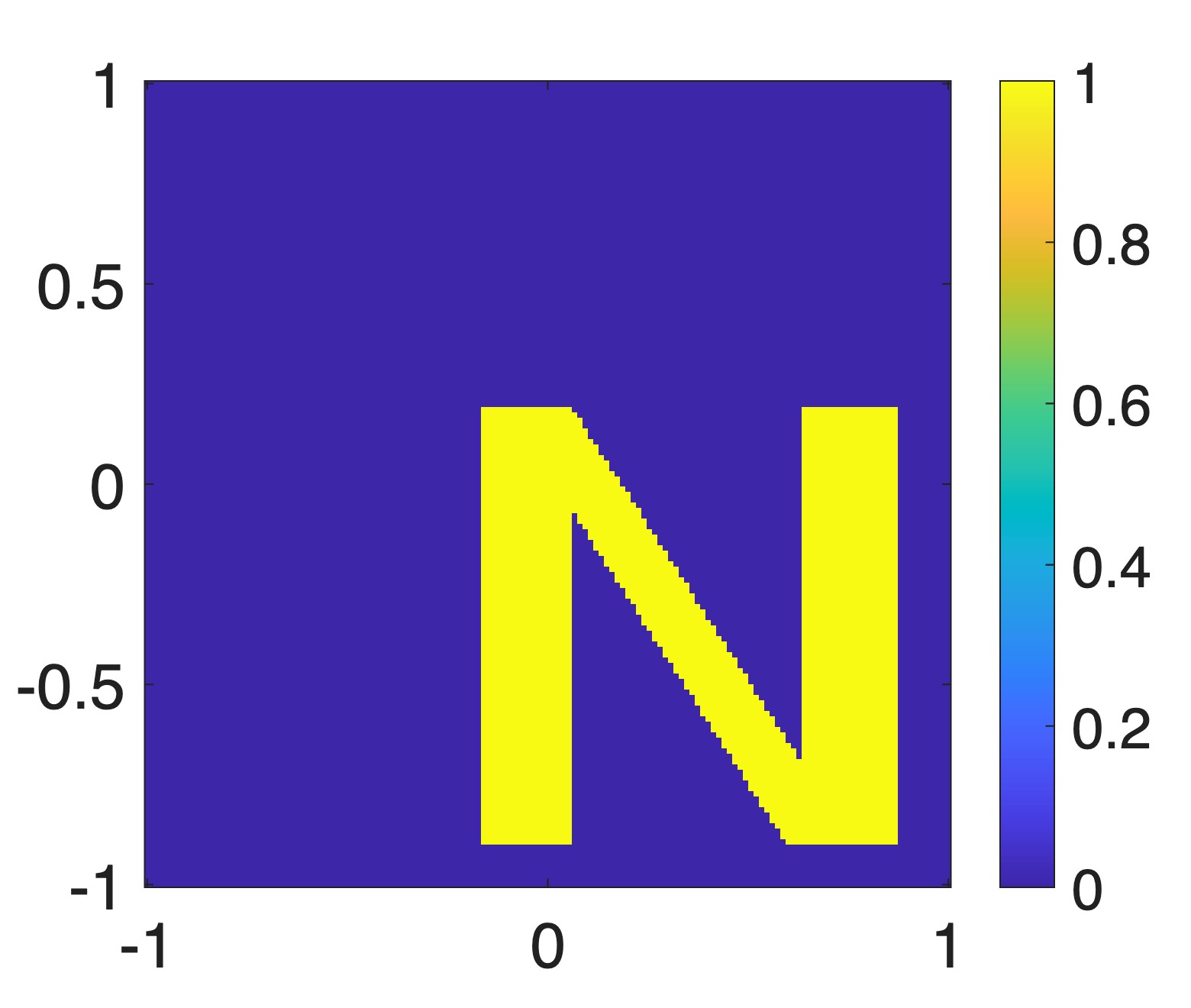}}
\hfill
\subfloat[Computed $\Im(u^0)$.\label{fig:case3_comp_imag}]{
    \includegraphics[width=0.235\textwidth]{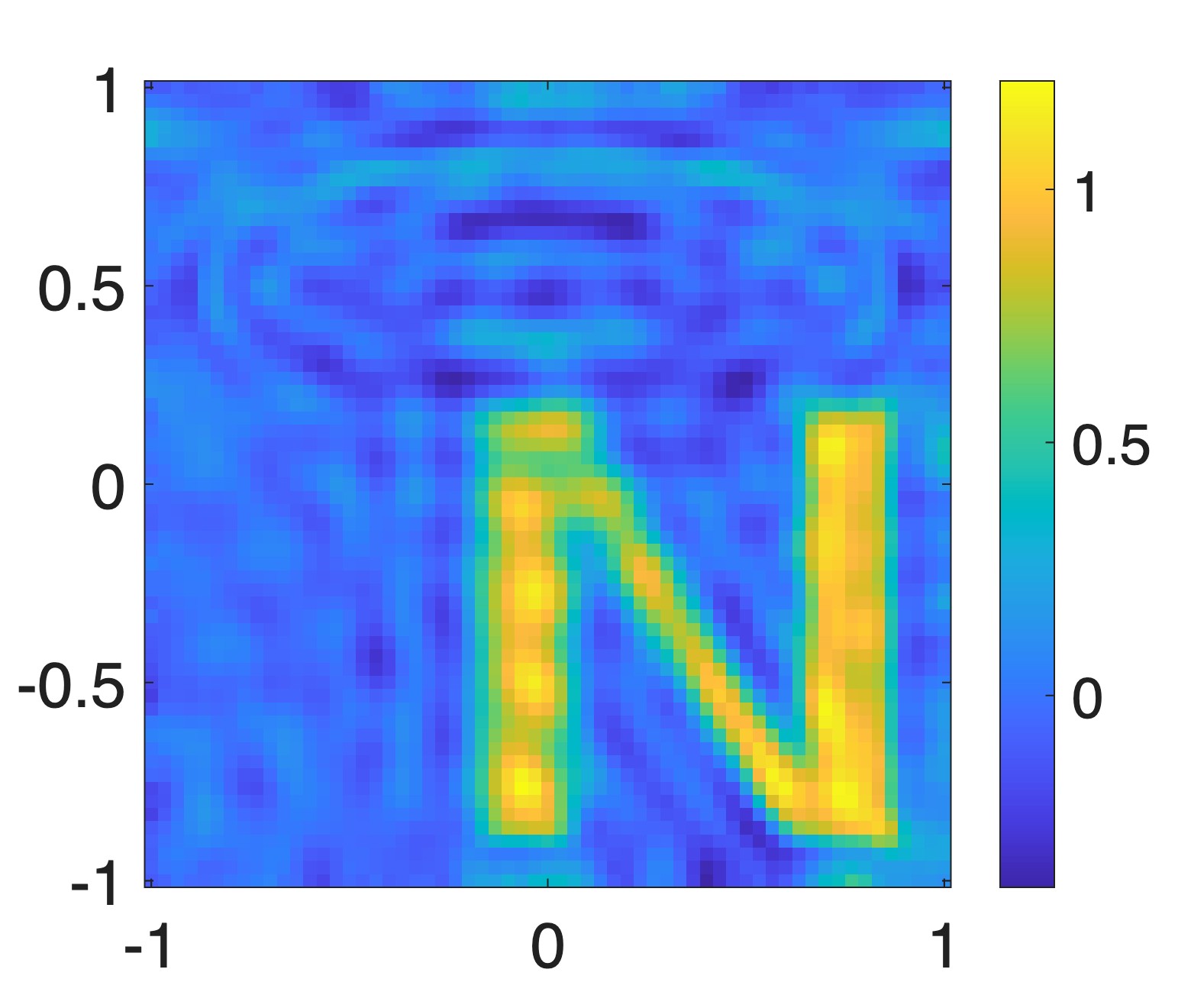}}
\hfill
\subfloat[$E_{\Im}$.\label{fig:case3_err_imag}]{
    \includegraphics[width=0.235\textwidth]{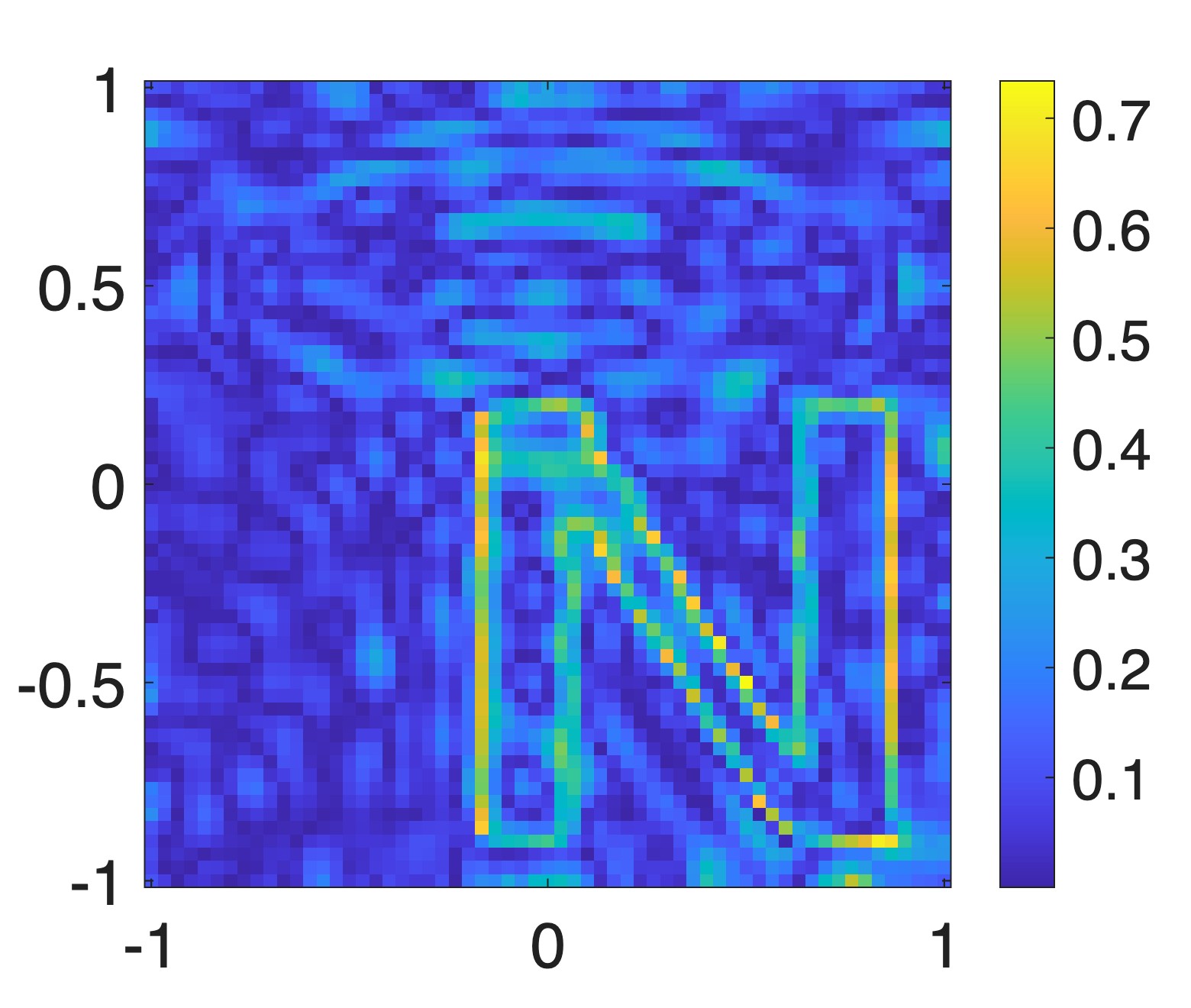}}
\hfill
\subfloat[Dimensionless residual.\label{fig:case3_residual}]{
    \includegraphics[width=0.235\textwidth]{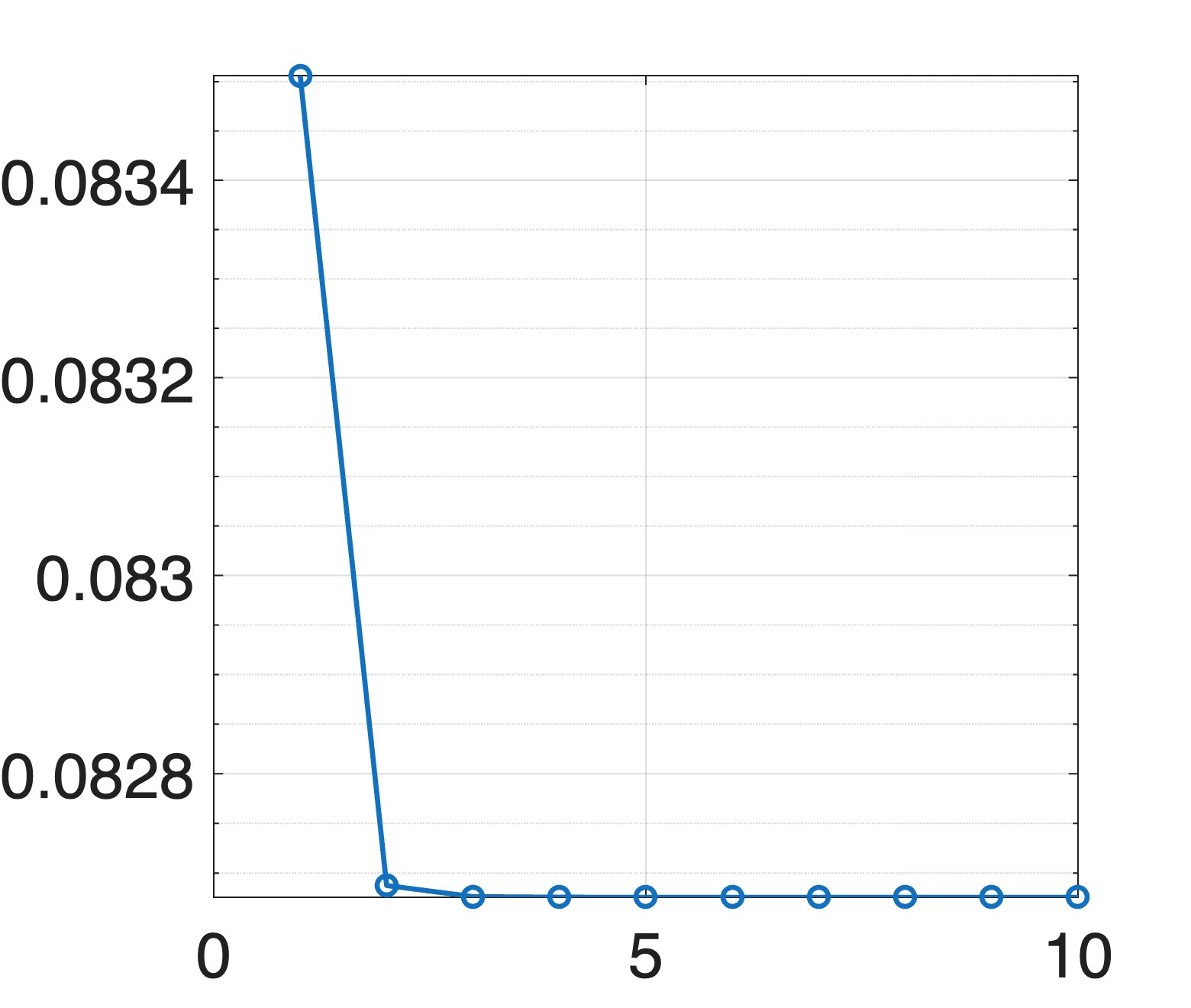}}
\caption{Test~3 with $10\%$ noise. The panels are arranged as in Figure~\ref{fig:case1_all}.}
\label{fig:case3_all}
\end{figure}

The results of Test~3 with $10\%$ noise are displayed in
Figure~\ref{fig:case3_all}. The elliptic ring and the N-shaped inclusion are
both recovered at the correct locations; the hole of the ring, the two
vertical bars, and the diagonal strip are clearly visible. Since both
components are discontinuous, their edges are smoothed by the $H^2$
regularization, and weak imprints of each inclusion appear in the other
component. The quantity \eqref{eq:picard_logdiff} decreases by about eight
orders of magnitude, and the residual \eqref{6.3} stabilizes after the second
iteration.
For the noise levels $5\%$, $10\%$, and $20\%$, the relative $L^2$ errors
are $e_{\Re}=40.43\%$, $40.56\%$, and $41.02\%$, and $e_{\Im}=43.18\%$,
$43.37\%$, and $43.92\%$, respectively. They are larger than in Tests~1 and~2,
mainly near the discontinuities of $u_R^0$ and $u_I^0$.
\begin{Remark}
In all three tests and for all noise levels $5\%$, $10\%$, and $20\%$, the locations and the main shapes of the inclusions are recovered, and the relative $L^2$ errors change only slightly as the noise level increases. The relative $L^2$ errors in Tests~1 and~2 are lower than those in Test~3.
\end{Remark}

\begin{Remark}
Tests~2 and~3 examine the method for initial fields that do not satisfy
the regularity assumptions used in the analysis. These examples model
situations with sharp spatial features, for which a smooth initial field
may not be appropriate. The reconstructions show that the main shapes and
locations of the inclusions are still recovered in these cases, despite
the lower regularity of the initial fields. However, the truncation-residual
estimate and the convergence and stability results proved under the
stated regularity assumptions do not directly apply to Tests~2 and~3.
Their results should therefore be interpreted as numerical evidence
rather than as consequences of those theorems.
\end{Remark}

\begin{Remark}[Sensitivity to the Carleman parameters]
To examine the sensitivity of the reconstruction to the Carleman
parameters $\x_0$, $\beta$, and $\lambda$, we repeat Test~3 with $10\%$
noise for the parameter configurations listed in
Table~\ref{tab:carleman_sensitivity}. The first row is the configuration
used in Tests~1--3; its errors are those reported above for Test~3
with $10\%$ noise. Across these six configurations, the relative $L^2$
errors range from $36.928\%$ to $41.696\%$ for the real part
and from $39.768\%$ to $44.705\%$ for the imaginary part.
The table allows us to compare different values of $\lambda$ while
fixing $(\x_0,\beta)=((0,-11),2)$ or $(\x_0,\beta)=((0,-8),2.5)$, and different values of
$\beta$ while fixing $(\x_0,\lambda)=((0,-8),500)$.

\begin{table}[ht]
\centering

\caption{Test 3 with $10\%$ noise for different Carleman
parameters. Relative errors in percent.}
\label{tab:carleman_sensitivity}
\begin{tabular}{ccccc}
\hline
$\x_0$ & $\beta$ & $\lambda$
& $e_{\Re} (\%)$ & $e_{\Im}(\%)$ \\
\hline
$(0,-11)$ & $2$   & $300$ & $40.562$ & $43.365$ \\
$(0,-11)$ & $2$   & $400$ & $39.589$ & $42.186$ \\
$(0,-11)$ & $2.5$ & $550$ & $41.696$ & $44.705$ \\
$(0,-8)$  & $2$   & $500$ & $36.928$ & $39.768$ \\
$(0,-8)$  & $2.5$ & $450$ & $38.529$ & $40.829$ \\
$(0,-8)$  & $2.5$ & $500$ & $38.285$ & $40.509$ \\
\hline
\end{tabular}
\end{table}
\end{Remark}

\begin{Remark}[Sensitivity to the truncation number]
To examine the sensitivity of the reconstruction to the truncation
number, we repeat Test~1 with $10\%$ noise for
$N=29,34,\ldots,69$, keeping all other parameters fixed and using
the same time-domain noisy boundary data for every $N$.
Figure~\ref{fig:truncation} shows the relative $L^2$ errors. Both errors decrease
markedly from $N=29$ to $N=39$ and then change only slightly.
For $54\leq N\leq69$, they remain between $20.09\%$ and $20.57\%$
for the real part and between $21.02\%$ and $21.49\%$ for the
imaginary part. Thus, the choice $N=64$ (65 temporal modes), used in Tests~1--3,
lies in a range where the reconstruction errors in Test~1 are
insensitive to the truncation number.

\begin{figure}[ht]
\centering
\includegraphics[width=0.5\textwidth]{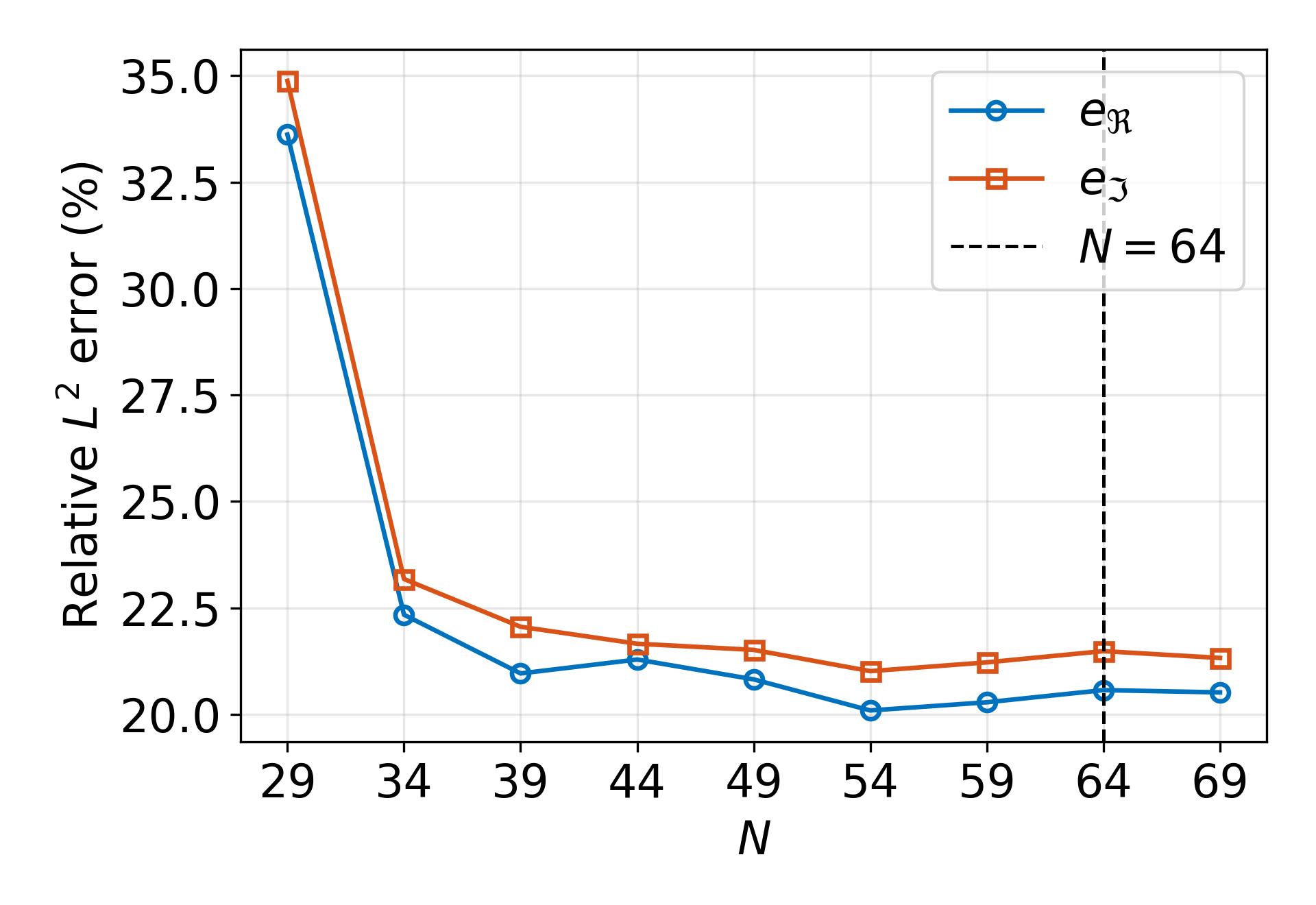}
\caption{Test~1 with $10\%$ noise: relative $L^2$ errors versus the truncation number $N$. The dashed line marks $N=64$.}
\label{fig:truncation}
\end{figure}
\end{Remark}

\begin{Remark}[Sensitivity to the inverse grid]\label{rem:grid}
We repeated Tests~1--3 on inverse grids with $41\times41$, $51\times51$,
and $61\times61$ nodes, keeping $N=64$, $\lambda=300$, $\beta=2$,
$\x_0=(0,-11)$, $\epsilon=10^{-4}$, the $H^2$ penalty, and $K_{\max}=10$
fixed. For each test, the forward problem was solved once on the
$151\times151$ grid, and a single noise realization of relative level
$10\%$ was added to the Neumann data on the boundary of that grid. These
noisy data were then interpolated linearly along each edge to the
boundary of each inverse grid. Therefore, the
$61\times61$ errors differ slightly from those reported above, where the
noise is added directly on the $61\times61$ boundary.
Table~\ref{tab:grid} shows the sensitivity of the reconstruction to the
inverse grid. In each test, both
relative $L^2$ errors decrease from the $41\times41$ grid to the
$61\times61$ grid.

\begin{table}[ht]
\centering

\caption{Tests~1--3 on different inverse grids, with the same
noisy data interpolated from the $151\times151$ boundary. Relative
$L^2$ errors, in percent.}
\label{tab:grid}
\begin{tabular}{cccc}
\hline
Test & Inverse grid & $e_{\Re}(\%)$ & $e_{\Im}(\%)$ \\
\hline
$1$ & $41\times41$ & $22.51$ & $23.32$ \\
$1$ & $51\times51$ & $21.01$ & $22.46$ \\
$1$ & $61\times61$ & $20.58$ & $21.54$ \\
\hline
$2$ & $41\times41$ & $25.45$ & $26.08$ \\
$2$ & $51\times51$ & $21.83$ & $21.12$ \\
$2$ & $61\times61$ & $21.34$ & $19.00$ \\
\hline
$3$ & $41\times41$ & $47.76$ & $48.82$ \\
$3$ & $51\times51$ & $41.41$ & $46.10$ \\
$3$ & $61\times61$ & $40.44$ & $43.24$ \\
\hline
\end{tabular}
\end{table}
\end{Remark}

\section{Concluding remarks} \label{sec7}

In this paper, we studied the reconstruction of the initial wave field of a nonlinear Schr\"odinger equation from lateral Neumann measurements. The main idea was to combine a Legendre-polynomial-exponential-time dimensional reduction with a Carleman-based contraction principle. This approach transforms the original inverse problem into a reduced nonlinear elliptic system for the time-expansion coefficients, and then solves that system using a globally convergent Picard iteration.

On the theoretical side, we constructed a contraction map on a suitable admissible set and proved that its unique fixed point is consistent with the exact reduced solution. We also established a stability estimate in the noisy-data case. Unlike the framework in \cite{nguyen2023carleman}, the present estimate does not require the boundary noise to admit a sufficiently regular extension into the domain.

On the numerical side, we proposed a practical reconstruction algorithm and tested it on several examples with different inclusion geometries and nonlinear exponents. The numerical results show that the relative $L^2$ reconstruction errors vary only slightly across the tested additive noise levels of $5\%$, $10\%$, and $20\%$. Across the considered tests, the method recovers the principal shapes and locations of the initial wave fields without requiring an initial approximation close to the exact solution.

\section*{CRediT authorship contribution statement}

\textbf{Navaraj Neupane:} Conceptualization, Methodology, Formal analysis, Software, Validation, Visualization, Writing--original draft, Writing--review and editing. \textbf{Loc Hoang Nguyen:} Conceptualization, Methodology, Formal analysis, Software, Validation, Visualization, Writing--original draft, Writing--review and editing. Both authors read and approved the final manuscript.

\section*{Funding}

This research did not receive any specific grant from funding agencies in the public, commercial, or not-for-profit sectors.

\section*{Declaration of competing interest}

The authors declare that they have no known competing financial interests or personal relationships that could have appeared to influence the work reported in this paper.

\section*{Data availability}

The numerical data and codes supporting the findings of this paper are available from the corresponding author upon reasonable request.


\begin{thebibliography}{10}

\bibitem{AbneyLeNguyenPeters}
Ray~Abney, Thuy~T. Le, Loc~H. Nguyen, and Cam~Peters.
\newblock A {C}arleman-{P}icard approach for reconstructing zero-order
  coefficients in parabolic equations with limited data.
\newblock {\em Applied Mathematics and Computation}, 494:129286, 2025.

\bibitem{arrepu2025}
Pranav Arrepu and Hanming Zhou.
\newblock Stable determination of coefficients in nonlinear dynamical
  {S}chr\"odinger equations by {C}arleman estimates.
\newblock {\em Inverse Problems and Imaging}, 2026. Early access,
  doi:10.3934/ipi.2026046.

\bibitem{baudouin2008}
Lucie Baudouin and Alberto Mercado.
\newblock An inverse problem for {S}chr\"odinger equations with discontinuous
  main coefficient.
\newblock {\em Applicable Analysis}, 87(10--11):1145--1165, 2008.

\bibitem{beilina2012}
Larisa Beilina and Michael~V. Klibanov.
\newblock {\em Approximate Global Convergence and Adaptivity for Coefficient
  Inverse Problems}.
\newblock Springer, New York, 2012.

\bibitem{bellassoued_benfraj2020}
Mourad Bellassoued and Oumaima Ben~Fraj.
\newblock Stability estimates for time-dependent coefficients appearing in the
  magnetic {S}chr\"odinger equation from arbitrary boundary measurements.
\newblock {\em Inverse Problems and Imaging}, 14(5):841--865, 2020.

\bibitem{bellassoued_choulli2009}
Mourad Bellassoued and Mourad Choulli.
\newblock Logarithmic stability in the dynamical inverse problem for the
  {S}chr\"odinger equation by arbitrary boundary observation.
\newblock {\em Journal de Math\'ematiques Pures et Appliqu\'ees},
  91(3):233--255, 2009.

\bibitem{bellassoued_choulli2010}
Mourad Bellassoued and Mourad Choulli.
\newblock Stability estimate for an inverse problem for the magnetic
  {S}chr\"odinger equation from the {D}irichlet-to-{N}eumann map.
\newblock {\em Journal of Functional Analysis}, 258(1):161--195, 2010.

\bibitem{bellassouedkiansoccorsi2016}
Mourad Bellassoued, Yavar Kian, and Eric Soccorsi.
\newblock An inverse stability result for non-compactly supported potentials by
  one arbitrary lateral {N}eumann observation.
\newblock {\em Journal of Differential Equations}, 260(10):7535--7562, 2016.

\bibitem{bellassoued2018cylindrical}
Mourad Bellassoued, Yavar Kian, and Eric Soccorsi.
\newblock An inverse problem for the magnetic {S}chr\"odinger equation in
  infinite cylindrical domains.
\newblock {\em Publications of the Research Institute for Mathematical
  Sciences}, 54:679--728, 2018.

\bibitem{benaicha2018}
Ibtissem Ben~A\"icha and Youssef Mejri.
\newblock Simultaneous determination of the magnetic field and the electric
  potential in the {S}chr\"odinger equation by a finite number of boundary
  observations.
\newblock {\em Journal of Inverse and Ill-Posed Problems}, 26(2):201--209,
  2018.


\bibitem{brezisgallouet1980}{
Haim Brezis and Thierry Gallou\"et.
\newblock Nonlinear {S}chr\"odinger evolution equations.
\newblock {\em Nonlinear Analysis: Theory, Methods \& Applications},
  4(4):677--681, 1980.}


\bibitem{cazenave2003}
Thierry Cazenave.
\newblock {\em Semilinear {S}chr\"odinger Equations}, volume~10 of {\em Courant
  Lecture Notes in Mathematics}.
\newblock American Mathematical Society, Providence, RI, 2003.

\bibitem{choulli2016elliptic}
{Mourad Choulli.
\newblock {\em Applications of Elliptic {C}arleman Inequalities to {C}auchy and
  Inverse Problems}.
\newblock Springer, Cham, 2016.}

\bibitem{choullikiansoccorsi2015}
Mourad Choulli, Yavar Kian, and Eric Soccorsi.
\newblock Stable determination of time-dependent scalar potential from boundary
  measurements in a periodic quantum waveguide.
\newblock {\em SIAM Journal on Mathematical Analysis}, 47(6):4536--4558, 2015.

\bibitem{crank1947}
{John Crank and Phyllis Nicolson.
\newblock A practical method for numerical evaluation of solutions of partial
  differential equations of the heat-conduction type.
\newblock {\em Proceedings of the Cambridge Philosophical Society},
  43(1):50--67, 1947.}

\bibitem{cristofol2011}
Michel Cristofol and Eric Soccorsi.
\newblock Stability estimate in an inverse problem for non-autonomous magnetic
  {S}chr\"odinger equations.
\newblock {\em Applicable Analysis}, 90(10):1499--1520, 2011.

\bibitem{dang2024hyperbolic}
Trong~D. Dang, Loc~H. Nguyen, and Huong T.~T. Vu.
\newblock Determining initial conditions for nonlinear hyperbolic equations
  with time dimensional reduction and the {C}arleman contraction principle.
\newblock {\em Inverse Problems}, 40:125021, 2024.

\bibitem{trong2025elastic}
Trong~D. Dang, Chanh~V. Le, Khoa~D. Luu, and Loc~H. Nguyen.
\newblock Recovery of initial displacement and velocity in anisotropic elastic
  systems by the time dimensional reduction method.
\newblock {\em Journal of Computational Physics}, 542:114371, 2025.

\bibitem{deimling1985}
Klaus Deimling.
\newblock {\em Nonlinear Functional Analysis}.
\newblock Springer-Verlag, Berlin, 1985.

\bibitem{delfour1981}
{Michel Delfour, Michel Fortin, and Guy Payre.
\newblock Finite-difference solutions of a non-linear {S}chr\"odinger
  equation.
\newblock {\em Journal of Computational Physics}, 44(2):277--288, 1981.}

\bibitem{deng2015}
Li~Deng.
\newblock An inverse problem for the {S}chr\"odinger equation with variable
  coefficients and lower order terms.
\newblock {\em Journal of Mathematical Analysis and Applications},
  427(2):930--940, 2015.

\bibitem{escauriaza2008hardy}
{Luis Escauriaza, Carlos~E. Kenig, Gustavo Ponce, and Luis Vega.
\newblock Hardy's uncertainty principle, convexity and {S}chr\"odinger
  evolutions.
\newblock {\em Journal of the European Mathematical Society},
  10(4):883--907, 2008.}

\bibitem{eskin2003}
Gregory Eskin.
\newblock Inverse problems for the {S}chr{\"o}dinger operators with
  electromagnetic potentials in domains with obstacles.
\newblock {\em Inverse Problems}, 19(4):985--996, 2003.

\bibitem{eskin2008}
Gregory Eskin.
\newblock Inverse problems for the {S}chr{\"o}dinger equations with
  time-dependent electromagnetic potentials and the {A}haronov--{B}ohm effect.
\newblock {\em Journal of Mathematical Physics}, 49(2):022105, 2008.

\bibitem{hormander2007iii}
{Lars H\"ormander.
\newblock {\em The Analysis of Linear Partial Differential Operators {III}:
  Pseudo-Differential Operators}.
\newblock Springer, Berlin, 2007.}

\bibitem{huang2019carleman}
Xinchi Huang, Yavar Kian, Eric Soccorsi, and Masahiro Yamamoto.
\newblock Carleman estimate for the {S}chr\"odinger equation and application to
  magnetic inverse problems.
\newblock {\em Journal of Mathematical Analysis and Applications},
  474(1):116--142, 2019.

%\bibitem{imanuvilov_yamamoto1998}
%Oleg~Yu. Imanuvilov and Masahiro Yamamoto.
%\newblock Lipschitz stability in inverse parabolic problems by the {C}arleman
 % estimate.
%\newblock {\em Inverse Problems}, 14(5):1229--1245, 1998.

\bibitem{isakov2017inverse}
{Victor Isakov.
\newblock {\em Inverse Problems for Partial Differential Equations}, 3rd
  edition, volume~127 of {\em Applied Mathematical Sciences}.
\newblock Springer, Cham, 2017.}

\bibitem{khoa2020}
Vo~Anh Khoa, Michael~V. Klibanov, and Loc~H. Nguyen.
\newblock Convexification for a {3D} inverse scattering problem with the moving
  point source.
\newblock {\em SIAM Journal on Imaging Sciences}, 13(2):871--904, 2020.

\bibitem{kian_soccorsi2019}
Yavar Kian and Eric Soccorsi.
\newblock {H}\"older stably determining the time-dependent electromagnetic
  potential of the {S}chr\"odinger equation.
\newblock {\em SIAM Journal on Mathematical Analysis}, 51(2):627--647, 2019.

\bibitem{klibanov2022convexification}
Michael~V. Klibanov, Thuy~T. Le, Loc~H. Nguyen, Anders Sullivan, and Lam
  Nguyen.
\newblock Convexification-based globally convergent numerical method for a {1D}
  coefficient inverse problem with experimental data.
\newblock {\em Inverse Problems and Imaging}, 16(6):1579--1618, 2022.

\bibitem{klibanov2021book}
Michael~V. Klibanov and Jingzhi Li.
\newblock {\em Inverse Problems and {C}arleman Estimates: Global Uniqueness,
  Global Convergence and Experimental Data}, volume~63 of {\em Inverse and
  Ill-Posed Problems Series}.
\newblock De Gruyter, Berlin, 2021.

\bibitem{krupchyk2023}
Katsiaryna Krupchyk and Gunther Uhlmann.
\newblock Inverse problems for nonlinear magnetic {S}chr\"odinger equations on
  conformally transversally anisotropic manifolds.
\newblock {\em Analysis and PDE}, 16(8):1825--1868, 2023.

\bibitem{lai_lu_zhou2024}
Ru-Yu Lai, Xuezhu Lu, and Ting Zhou.
\newblock Partial data inverse problems for the nonlinear time-dependent
  {S}chr\"odinger equation.
\newblock {\em SIAM Journal on Mathematical Analysis}, 56(4):4712--4741, 2024.

\bibitem{lai_zhou2023}
Ru-Yu Lai and Ting Zhou.
\newblock Partial data inverse problems for nonlinear magnetic {S}chr\"odinger
  equations.
\newblock {\em Mathematical Research Letters}, 30(5):1535--1563, 2023.

%\bibitem{lattes_lions1969}
%Robert Latt\`es and Jacques-Louis Lions.
%\newblock {\em The Method of Quasi-Reversibility: Applications to Partial
 % Differential Equations}.
%\newblock Elsevier, New York, 1969.

\bibitem{LeCON2023}
Thuy~T. Le.
\newblock Global reconstruction of initial conditions of nonlinear parabolic
  equations via the {C}arleman-contraction method.
\newblock In D-L. Nguyen, L.~H. Nguyen, and T-P. Nguyen, editors, {\em Advances
  in Inverse problems for Partial Differential Equations}, volume 784 of {\em
  Contemporary Mathematics}, pages 23--42. American Mathematical Society, 2023.

\bibitem{LeLeNguyen:2024}
Huynh P. N. Le, Thuy~T. Le, and Loc~H. Nguyen.
\newblock The {C}arleman convexification method for {H}amilton-{J}acobi
  equations.
\newblock {\em Computers and Mathematics with Applications}, 159:173--185,
  2024.

\bibitem{le2024timedim}
Thuy~T. Le, Linh~V. Nguyen, Loc~H. Nguyen, and Hyunha Park.
\newblock The time dimensional reduction method to determine the initial
  conditions without the knowledge of damping coefficients.
\newblock {\em Computers and Mathematics with Applications}, 166:77--90, 2024.

\bibitem{le2022parabolic}
Thuy~T. Le and Loc~H. Nguyen.
\newblock A convergent numerical method to recover the initial condition of
  nonlinear parabolic equations from lateral {C}auchy data.
\newblock {\em Journal of Inverse and Ill-Posed Problems}, 30(2):265--286,
  2022.

\bibitem{le2025maxwell}
Thuy~T. Le, Cong~B. Van, Trong~D. Dang, and Loc~H. Nguyen.
\newblock Inverse initial data reconstruction for {M}axwell's equations via
  time splitting method.
\newblock {\em Journal of Computational Physics}, 559:114896, 2026.

\bibitem{lee2007}
J.-H. Lee, O.~K. Pashaev, C.~Rogers, and W.~K. Schief.
\newblock The resonant nonlinear {S}chr{\"o}dinger equation in cold plasma
  physics. application of {B}{\"a}cklund--{D}arboux transformations and
  superposition principles.
\newblock {\em Journal of Plasma Physics}, 73(2):257--272, 2007.

\bibitem{mercado2008}
Alberto Mercado, Axel Osses, and Lionel Rosier.
\newblock Inverse problems for the {S}chr\"odinger equation via {C}arleman
  inequalities with degenerate weights.
\newblock {\em Inverse Problems}, 24(1):015017, 2008.

\bibitem{nguyen2022hyperbolic}
Dinh-Liem Nguyen, Loc~H. Nguyen, and Trung Truong.
\newblock The {C}arleman-based contraction principle to reconstruct the
  potential of nonlinear hyperbolic equations.
\newblock {\em Computers and Mathematics with Applications}, 128:239--248,
  2022.

\bibitem{nguyen2015cloaking}
Hoai-Minh Nguyen and Loc~H. Nguyen.
\newblock Cloaking using complementary media for the {H}elmholtz equation and a
  three spheres inequality for second order elliptic equations.
\newblock {\em Transactions of the American Mathematical Society, Series B},
  2:93--112, 2015.

\bibitem{nguyen2019}
Loc~H. Nguyen.
\newblock An inverse space-dependent source problem for hyperbolic equations
  and the {L}ipschitz-like convergence of the quasi-reversibility method.
\newblock {\em Inverse Problems}, 35:035007, 2019.

\bibitem{nguyen2023carleman}
Loc~H. Nguyen.
\newblock The {C}arleman contraction mapping method for quasilinear elliptic
  equations with over-determined boundary data.
\newblock {\em Acta Mathematica Vietnamica}, 48:401--422, 2023.

\bibitem{klibanov2022contraction}
Loc~H. Nguyen and Michael~V. Klibanov.
\newblock Carleman estimates and the contraction principle for an inverse
  source problem for nonlinear hyperbolic equations.
\newblock {\em Inverse Problems}, 38(3):035009, 2022.

\bibitem{NguyenNguyen2026}
Phuong~M. Nguyen and Loc~H. Nguyen.
\newblock A {C}arleman contraction method for inverse initial data recovery in
  the {N}avier--{S}tokes equations with unknown body force.
\newblock Preprint, arXiv:2604.09934, 2026, doi:10.48550/arXiv.2604.09934.

\bibitem{NguyenNguyenVu2026}
Phuong~M. Nguyen, Loc~H. Nguyen, and Huong~T.~T. Vu.
\newblock Solving the inverse scattering problem via {C}arleman-based
  contraction mapping.
\newblock {\em Computers and Mathematics with Applications}, 209:129--143,
  2026.

\bibitem{pitaevskii2016}
Lev Pitaevskii and Sandro Stringari.
\newblock {\em Bose-Einstein Condensation and Superfluidity}.
\newblock Oxford University Press, 2016.

\bibitem{protter1960}
Murray~H. Protter.
\newblock Unique continuation for elliptic equations.
\newblock {\em Transactions of the American Mathematical Society},
  95(1):81--91, 1960.

\bibitem{saci2023}
Atef Saci and Salah-Eddine Rebiai.
\newblock An inverse problem for the {S}chr\"odinger equation with {N}eumann
  boundary condition.
\newblock {\em Advances in Pure and Applied Mathematics}, 14(1):50--69, 2023.

\bibitem{sulem1999}
Catherine Sulem and Pierre-Louis Sulem.
\newblock {\em The Nonlinear {S}chr\"odinger Equation: Self-Focusing and Wave
  Collapse}.
\newblock Springer, New York, 1999.

\bibitem{tenenbaum2009fast}
{Gerald Tenenbaum and Marius Tucsnak.
\newblock Fast and strongly localized observation for the {S}chr\"odinger
  equation.
\newblock {\em Transactions of the American Mathematical Society},
  361(2):951--977, 2009.}

\bibitem{van2025navierstokes}
Cong~B. Van, Thuy~T. Le, and Loc~H. Nguyen.
\newblock The inverse initial data problem for anisotropic {N}avier--{S}tokes
  equations via {L}egendre time reduction method.
\newblock {\em Communications in Nonlinear Science and Numerical Simulation}, 161:110074, 2026.

\bibitem{vitanov2013}
Nikolay~K. Vitanov, Amin Chabchoub, and Norbert Hoffmann.
\newblock Deep-water waves: On the nonlinear {S}chr{\"o}dinger equation and its
  solutions.
\newblock {\em Journal of Theoretical and Applied Mechanics}, 43(2):43--54,
  2013.

\bibitem{yamamoto2025inverse}
{Masahiro Yamamoto.
\newblock {\em Introduction to Inverse Problems for Evolution Equations:
  Stability and Uniqueness by {C}arleman Estimates}.
\newblock Springer, Cham, 2025.}

\bibitem{zeidler1985}
Eberhard Zeidler.
\newblock {\em Nonlinear Functional Analysis and its Applications, Volume
  {III}: Variational Methods and Optimization}.
\newblock Springer-Verlag, New York, 1985.

\end{thebibliography}
\end{document}